\documentclass[9pt]{amsart}

\usepackage[margin=2cm]{geometry}
\usepackage{enumitem}
\usepackage{ upgreek }
\usepackage{amsmath,amsfonts,amssymb,amsthm,amscd,comment,euscript}
\usepackage{mathbbol, stmaryrd}

\usepackage[all]{xy}
\usepackage{graphicx}
\usepackage{enumerate} 
\usepackage[colorlinks,  linkcolor=blue]{hyperref} 
\usepackage[usenames,dvipsnames]{xcolor}

\usepackage{bookmark}
\usepackage{hyperref}

\newtheorem{theorem}{Theorem}[section]
\newtheorem{corollary}{Corollary}[section]
\newtheorem{lemma}{Lemma}[section]
\newtheorem{proposition}{Proposition}[section]
\newtheorem{definition}{Definition}[section]
\newtheorem{example}{Example}[section]
\newtheorem{remark}{Remark}[section]

\newtheorem{conjecture}{Conjecture}[section]

\newcommand{\Y}{\mathcal{Y}}
\newcommand{\Z}{\mathbb{Z}}      
\newcommand{\C}{\mathbb{C}} 
\newcommand{\T}{\mathcal{T}} 
\newcommand{\J}{\mathcal{J}}
\newcommand{\al}{\alpha}
\newcommand{\Res}{\operatorname{Res}}

\begin{document}

\title[]
{Quasi-lisse vertex (super)algebras}
\author{Hao Li}
\address{Research Institute for Mathematical Sciences, Kyoto University}
\email{hao.lilwh@gmail.com}
\begin{abstract}
 We study quasi-lisse vertex (super)algebras and establish new finiteness conditions for the convergence of genus-zero and genus-one $n$-point correlation functions.

\end{abstract}

\maketitle

\section{Introduction}

Vertex algebras and their representations provide a rigorous mathematical formulation of two-dimensional conformal field theory. In his seminal work \cite{zhu1996modular}, Zhu introduced the notion of the $C_2$-algebra of a vertex operator algebra $V$, which carries a Poisson algebra structure. Using the finite-dimensionality of this algebra, Zhu derived modular linear differential equations (MLDEs) satisfied by genus-one correlation functions associated to V. This idea was later generalized by Dong-Li-Mason \cite{dong2000modular} for twisted representations, by Miyamoto \cite{miyamoto2000intertwining}, and by Huang \cite{huang2005differential} for $q$-traces of products of intertwining operators. All these works assume that the vertex operator algebras and their modules are $C_2$-cofinite.

In order to formulate an analogue of the finiteness condition defined by Beilinson, Feigin, and Mazur~\cite{beilinson1991introduction}, Arakawa~\cite{arakawa2012remark} studied the $C_2$-algebra of $V$ from a geometric perspective, interpreting it as the coordinate ring of the \emph{associated variety} of $V$, denoted by $X_V$. He showed that $\dim(X_V) = 0$ is equivalent to $C_2$-cofiniteness, and that $C_2$-cofiniteness is also equivalent to the triviality of the \emph{singular support} of $V$. 

Later, Arakawa and Kawasetsu~\cite{arakawa2018quasi} discovered that Zhu's argument for deriving MLDEs still applies when the associated variety $X_V$ is allowed to have only finitely many symplectic leaves. Vertex algebras with this property are called \emph{quasi-lisse}. A crucial input is the following property of $\mathrm{Specm}\,R$, established in~\cite{etingof2010poisson}:
\begin{align}
    \label{introd1}
    \dim R/\{R, R\} < \infty.
\end{align}
In addition, Arakawa~\cite{arakawa2015associated} proved that the associated variety of an affine vertex operator algebra at an admissible level is the closure of a nilpotent orbit. Thus, quasi-lisse vertex algebras can be viewed as vertex algebraic generalizations of affine VOAs at admissible levels.

Meanwhile, Beem-Lemos-Liendo-Peelaers-Rastelli-van Rees \cite{beem2015infinite} constructed a remarkable map from four-dimensional superconformal field theory (4d SCFT) to two-dimensional conformal field theory (2d CFT). It was conjectured by Beem and Rastelli that the Higgs branch of 4d SCFT should coincide with the associated variety of the corresponding 2d CFT. One can identify the 4d Higgs branch with the 3d Coulomb branch via 3D mirror symmetry. When the mirror theory is quiver gauge theory, Braverman-Finkelberg-Nakajima rigorously defined the Coulomb branch  (\cite{nakajima2015towards,braverman2018towards,braverman1604coulomb}), which is expected to have finitely many symplectic leaves. Thus, there is a demand to systematically study the quasi-lisse vertex operator algebras. In the present paper, we make some modest efforts in this direction.

Recently, McRae \cite{mcrae2021rationality} proved that for certain $C_2$-cofinite vertex operator algebras, the tensor category of grading-restricted generalized $V$-modules is rigid. He used the method of Huang-Moore-Seiberg by expanding genus-one $2$-point correlation functions. Here, the $2$-point function refers to the $q$-trace of the product of two intertwining operators. Thus, to prove rigidity, the convergence of genus-one $2$-point functions is crucial. The convergence result for genus-one correlation functions was obtained by Huang \cite{huang2005differential} for $C_2$-cofinite algebras. Relaxing this convergence condition under the setting of quasi-lisse vertex operator (super)algebra is one of the main motivations of this paper.

Now, let us describe the main results in this paper. We extend the notion of quasi-lisse vertex algebras to the supercase, and show that such algebras have only finitely many simple ordinary $g$-twisted modules for a finite order automorphism $g$.
Huang \cite{huang2005differential1} showed that for $C_1$-cofinite vertex operator algebras and their $C_1$-cofinite modules, the genus-zero $n$-point functions satisfy certain differential equations with regular singularities. For any $V$-module $W$, one can define its $C_2$-module, $R_W$, which is an $R_V$-module. We prove that the $C_1$-cofiniteness condition can be replaced by $\text{dim} R_{W_0}\otimes_{R_V}\cdots \otimes_{R_V} R_{W_3}/\{R_{V},R_{W_0}\otimes_{R_V}\cdots \otimes_{R_V} R_{W_3}\}<\infty$ for certain vertex operator algebras.
\begin{theorem}[see Theorem~\ref{differentialintertwining}]\label{1.1}
Suppose $V$ is of CFT type, self-dual, and finitely generated by vectors of weight $1$.  Let $W_{i}$ for $i=0,1,2,3$ be weak $\mathbb{Q}$-graded $V$-modules. If the above finiteness condition holds,  then for any $w_{i}\in W_{i}$ $(i=0,1,2,3), $ there exist rational functions \[a_{k}(z_{1},z_{2}), b_{l}(z_{1},z_{2})\in \mathbb{C}[z_{1}^{\pm},z_{2}^{\pm},(z_{1}-z_{2})^{-1}]\]
for $k=1,...,m $ and $l=1,...,n$, such that for any discretely weak $\mathbb{Q}$-graded $V$-modules $W_{4},W_{5}$ and $W_{6}$, any intertwining operators $\mathcal{Y}_{1},\mathcal{Y}_{2}, \mathcal{Y}_{3},\mathcal{Y}_{4},\mathcal{Y}_{5}$ and $\mathcal{Y}_{6}$ of types \[\binom{W_{0}'}{W_{1}W_{4}},\binom{W_{4}}{W_{2}W_{3}},\binom{W_{5}}{W_{1}W_{2}},\binom{W_{0}'}{W_{5}W_{3}},\binom{W_{0}'}{W_{2}W_{6}}\,\,\,\, \text{and}\,\,\,\,  \binom{W_{6}}{W_{1}W_{3}},\] respectively, the series
\begin{align}
    &\langle w_{0}, \mathcal{Y}_{1}(w_{1},z_{1})\mathcal{Y}_{2}(w_{2},z_{2})w_{3}\rangle,\\
    &\langle w_{0}, \mathcal{Y}_{4}(\mathcal{Y}_{3}(w_{1},z_{1}-z_{2})w_{2},z_{2})w_{3}\rangle,
\end{align}
and
\begin{align}
    \langle w_{0}, \mathcal{Y}_{5}(w_{2},z_{2})\mathcal{Y}_{6}(w_{1},z_{1})w_{3}\rangle,
\end{align}
are convergent in the regions $|z_{1}|>|z_{2}|>0,$ $|z_{2}|>|z_{1}-z_{2}|>0$ and $|z_{2}|>|z_{1}|>0,$ respectively, and can be analytically extended to multivalued functions in the region
\[\{(z_1,z_2)\in \C^2|z_1,z_2\neq 0,z_1\neq z_2\}.\]
\end{theorem}
We also propose a natural extension of the convergence and analytic continuation properties of intertwining operators to settings involving twisted modules; see Conjecture~\ref{duality}.

For the genus-one case, we consider the weak $g$-twisted $V$-module $\tilde{W}_n$. Let $h$ be a current satisfying
\[L(n)h=\delta_{n,0}h,\quad h_{(n)}h=k\delta_{n,1}\mathbf{1},\]
for all $n\geq 0$ and some $k\in \mathbb{R}$, such that $h_{(0)}$ acts semisimply on $\tilde{W}_n$ with eigenvalues in $\mathbb{R}$.
We assume that
 \[\operatorname{tr}_{\tilde{W}_n}q^{h_{(0)}}_{s}q^{L_{(0)}}, \quad \text{where}\;\;q_s=e^{2\pi i s},\] is a well-defined $(q_s,q)$-series; that is, the simultaneous eigenspaces of $L_{(0)}$ and $h_{(0)}$ are finite-dimensional. The parameter $s$ is a complex number in the upper half plane. In the following, we call such an $h$ a Cartan element and say that the weak $g$-twisted module is $h$-stable. Under the assumption that {\em duality property} holds, we also get a similar weaker finiteness condition.  Let 
 \begin{align*}
R' =\;
& \mathbb{C}\left[\tilde{G}_{2k}(q) \;\middle|\; k \geq 2 \right] 
  \otimes \mathbb{C}\left[\tilde{G}_{k} \genfrac[]{0pt}{0}{\theta}{\phi}(q) \;\middle|\; k \geq 0 \right] \\
& \otimes \mathbb{C}\big[
  \tilde{P}_{1} \genfrac[]{0pt}{0}{\theta}{\phi}(z_{i}-z_{j};q),\,
  \tilde{P}_{2} \genfrac[]{0pt}{0}{\theta}{\phi}(z_{i}-z_{j};q), \\
& \hspace{5em}
  \tilde{\wp}_{2}(z_i - z_j;q),\,
  \tilde{\wp}_3(z_i - z_j;q)
\big]_{1 \leq i, j \leq n,\; i \neq j}.
\end{align*}
 We denote by \( R'_p \) the subspace of elements in \( R' \) of modular weight \( p \); further details are provided in Section~\ref{r'introduction} and Appendix~C. Then we can state our next main result.

\begin{theorem}[see Theorem~\ref{maint1} and Theorem~\ref{1pointfun}]\label{1.2}
Let $g$ be an automorphism of $V$ of finite order. Let $h$ be a Virasoro primary vector. Assume \[g(V\setminus C_2(V))\subset V\setminus C_2(V).\] Let $W_i$, $i=1,2,...,n$ be weak $V$-modules, and let $\tilde{W}_i$, $i=1,2,...,n$ be $h$-stable weak $g$-twisted $V$-modules with $\tilde{W}_0=\tilde{W}_n$. Suppose $\text{dim}(R_{W_1}\otimes_{R_V} \cdots \otimes_{R_V} R_{W_n}/\{(R_V)^g,R_{W_1}\otimes_{R_V} \cdots \otimes_{R_V} R_{W_n}\})<\infty$ and let $\mathcal{Y}_{i},$ $i=1,\cdots,n,$ be intertwining operators of type $\displaystyle\binom{\tilde{W}_{i-1}}{W_{i}\;\; \tilde{W}_{i}}$, then
\begin{itemize}
    \item For any homogeneous elements $w_{i}\in W_{i}$, $i=1,...,n,$ there exist \begin{align*}
& a_{p,i}(z_{1}, \dots, z_{n}; q; h) \in R'_{p}, \\
& b_{p,i}(z_{1}, \dots, z_{n}; q; h) \in R'_{2p}, \\
& c_{p,i}(z_{1}, \dots, z_{n}; q; h) \in R'_{l} \otimes \left( \mathbb{C}[z_1, \dots, z_n] \right)_{m}
\quad \text{with } l + m = p.
\end{align*} for $p=1,...,m$ and $i=1,...,n$, such that genus-one $n$-point function 
\begin{align*}
F_{\mathcal{Y}_{1}, \dots, \mathcal{Y}_{n}}&(w_{1}, \dots, w_{n}; z_{1}, \dots, z_{n}; q; h) \\
&= \operatorname{tr}_{\tilde{W}_n} \Big[
\mathcal{Y}_1\big(\mathcal{U}(q_{z_{1}})v_{1}, q_{z_{1}}\big)
\cdots
\mathcal{Y}_n\big(\mathcal{U}(q_{z_{n}})v_{n}, q_{z_{n}}\big)
e^{2\pi i s h_{(0)}} q^{L_{(0)}-\frac{c}{24}}
\Big]
\end{align*}
satisfies the following system of differential equations:

\begin{align}\label{equatio1}
    \frac{\partial^m f}{\partial z_i^m}
    + \sum_{p=1}^m a_{p,i}(z_1, \dots, z_n; q; h) \frac{\partial^{m-p} f}{\partial z_i^{m-p}} = 0,
\end{align}

\begin{align}\label{diffequ1}
\begin{split}
    & \prod_{k=1}^m 
    \mathcal{O}_i^h\left( \sum_{j=1}^n \operatorname{wt} w_j + 2(m - k) \right) f \\
    & \quad + \sum_{p=1}^m 
    b_{p,i}(z_1, \dots, z_n; q; h)\,
    \prod_{k=1}^{m-p} 
    \mathcal{O}_i^h\left( \sum_{j=1}^n \operatorname{wt} w_j + 2(m - p - k) \right) f \\
    & = 0.
\end{split}
\end{align}
\begin{align}\label{98}
\begin{split}
    \left(q_s \frac{\partial}{\partial q_s} \right)^m f
    + \sum_{p=1}^m c_{p,i}(z_1, \dots, z_n; q; h)
    \left( q_s \frac{\partial}{\partial q_s} \right)^{m-p} f = 0,
\end{split}
\end{align}
$i=1,...,n,$ in the regions $1>|q_{z_{1}}|>\cdots>|q_{z_{n}}|>|q|>0,0<|q_s|<1$
where for any $\al\in \mathbb{C}$
\begin{align*}
& \mathcal{O}_{j}^{h}(\alpha) = (2\pi i)^{2} q \frac{\partial}{\partial q} 
+ \tilde{G}_{2}(q) \alpha 
+ \tilde{G}_{2}(q) \sum_{i=1}^{n} z_{i} \frac{\partial}{\partial z_{i}} \\
&\quad - \sum_{i \neq j} \sum_{k=0}^{N-1} 
\tilde{\wp}_{1}(z_{i} - z_{j}; q) \frac{\partial}{\partial z_{i}}, 
\quad \text{for } j = 1, \dots, n, \\
& \prod_{j=1}^{m} \mathcal{O}^{h}(\alpha_{j}) 
= \mathcal{O}^{h}(\alpha_{1}) \cdots \mathcal{O}^{h}(\alpha_{m}).
\end{align*}

\item In the region 
\[\left\{(z_{1},...,z_{n},\tau,s)\,\middle|\, 1>|q_{z_{1}}|>\cdots>|q_{z_{n}}|>|q_{\tau}|>0, 0<|q_s|<1\right\},\] the function \[F_{\mathcal{Y}_{1},...,\mathcal{Y}_{n}}(w_{1},...,w_{n};z_{1},...,z_{n};q_\tau;h)\]
is absolutely convergent, and an analytic continuation as a multivalued analytic function in the region 
\[\tau\in \mathbb{H},\quad 0<|q_s|<1,\quad  z_{i}\neq z_{j}+k\tau+l\quad \text{for}\;\; i,j=1,...,n, \;\; i\neq j\;\; k,l\in\mathbb{Z}.\]
\end{itemize}
\end{theorem}
In particular, for a quasi-lisse vertex (super)algebra, Theorem~\ref{leaves}, originally due to Etingof and Schedler, and its reformulation in the vertex algebra setting given in Corollary~\ref{leaves'}, together imply that any finitely strongly generated module satisfies all the above finiteness conditions in the case $g = \mathrm{id}$.

As a special case of Theorem~\ref{1.2}, we obtain the following.
\begin{corollary}[see Corollary \ref{twistedmlde}]
If \( V \) is a quasi-lisse vertex superalgebra, then the supercharacter of any simple \( g \)-twisted \( V \)-module satisfies a twisted modular linear differential equation.
\end{corollary}

We now outline the main ideas behind the proofs of Theorems~\ref{1.1} and~\ref{1.2}, building on methods developed in~\cite{zhu1996modular, huang2005differential1,huang2005differential,arakawa2018quasi}. In general, one defines a linear map
$\Upsilon : W_1 \otimes \cdots \otimes W_n \longrightarrow \text{$n$-point correlation functions}.$
The Jacobi identity gives rise to relations in the kernel of $\Upsilon$, and the more such relations one finds, the weaker the convergence conditions that result. By identifying additional kernel elements and applying a suitable filtration, one obtains a surjective map $\Theta$:
\[
\frac{R_{W_1} \otimes_{R_V} \cdots \otimes_{R_V} R_{W_n}}{\{(R_V)^g,\ R_{W_1} \otimes_{R_V} \cdots \otimes_{R_V} R_{W_n}\}} \twoheadrightarrow \T / \operatorname{gr}(\ker({\Upsilon})).
\]
The finite-dimensionality of the source implies that of the target. In particular, it implies that $\T/\ker(\Upsilon)$ is finitely generated over a Noetherian ring. Using the $L(-1)$-derivative property and constructing an appropriate submodule of $\T/\ker(\Upsilon)$, we obtain the desired system of differential equations satisfied by the correlation functions.

The organization of this paper is as follows. Section 2 recalls the definitions and basic notations related to vertex operator (super)algebras. Section 3 introduces the notion of quasi-lisse vertex (super)algebras and proves that such algebras have only finitely many simple ordinary $g$-twisted modules. In Section 4, we study the convergence of genus-zero $n$-point correlation functions and derive a necessary condition. Section 5 forms the main part of the paper: we derive necessary conditions for the convergence of genus-one correlation functions, and verify that highest weight modules and their contragredient modules for $L_k(\mathfrak{sl}_2)$ at admissible levels satisfy these conditions. In Appendix A, we review geometrically modified vertex operators and explain their geometric interpretation. Appendices B and~C collect standard facts about modular forms.

\subsection{Notations and conventions}

$\quad$ $\delta(x)$: $\sum_{n\in \Z} x^n$,

$\mathbb{N}:$ all natural number,

$q_{z}:e^{2\pi i z}$,

$\mathbb{Z}_-:$ all negative integers,

$\mathbb{Z}_{+}:$ all positive integers,

$\mathbb{H}$: upper half plane,

$\mathbb{C}((x))$: ring of all Laurent series in $x$,

$\mathbb{C}\{x\}$: the space of all series of the form $\sum_{n\in \mathbb{R}}a_{n}x^{n}$, $n,a_n\in \mathbb{Q}$,

$\mathbb{G}_{|z_1|>\cdots>|z_n|>0}:$ the space of all multivalued analytic functions in $z_1,,...,z_n$ defined on the region $|z_1|>\cdots>|z_n|>0$ with preferred branches in the simply-connected region $|z_1|>\cdots >|z_n|>0,$ $0\leq \operatorname{arg}z_i<2\pi$, $i=1,...,n$.

\section*{Acknowledgments}

I would like to thank Antun Milas for suggesting the topic of quasi-lisse vertex algebras, and Yi-Zhi Huang for motivating questions raised during a conference organized by Corina Calinescu. I am also grateful to Tomoyuki Arakawa for discussions on finiteness conditions.  

This work was largely carried out during my postdoctoral appointment at the Yau Mathematical Sciences Center, Tsinghua University, where I benefited from discussions on convergence problems with Robert McRae and Jinwei Yang. I also thank the anonymous referee for valuable suggestions and for pointing out several issues in an earlier version of this paper. The major revision of this paper was completed during my stay at the Research Institute for Mathematical Sciences, Kyoto University, where I am supported by JSPS KAKENHI Grant Number 21H04993.

\section{Vertex superalgebras}\label{section1}

\begin{definition}
Let \( V \) be a superspace, i.e., a \( \mathbb{Z}_2 \)-graded vector space:
\[
V = V_{\overline{0}} \oplus V_{\overline{1}}, \quad \text{where } \mathbb{Z}_2 = \{ \overline{0}, \overline{1} \}.
\]
If \( a \in V_{p(a)} \), we say that \( a \) has parity \( p(a) \in \mathbb{Z}_2 \).

A \emph{field} is a formal series of the form
\[
a(z) = \sum_{n \in \mathbb{Z}} a_{(n)} z^{-n-1}, \quad a_{(n)} \in \operatorname{End}(V),
\]
such that for each \( v \in V \), we have \( a_{(n)} v = 0 \) for \( n \gg 0 \).

We say that a field \( a(z) \) has parity \( p(a) \in \mathbb{Z}_2 \) if
\[
a_{(n)} V_\alpha \subset V_{\alpha + p(a)}
\]
for all \( \alpha \in \mathbb{Z}_2 \), \( n \in \mathbb{Z} \).

A \emph{vertex superalgebra} consists of the following data:
\begin{itemize}
    \item a superspace \( V \),
    \item a vacuum vector \( \mathbf{1} \in V_{\overline{0}} \),
    \item a translation operator \( T \in \operatorname{End}(V) \),
    \item a state-field correspondence map \( Y(-,z): V \to (\operatorname{End}V)[[z, z^{-1}]] \), defined by
    \[
    Y(a,z) = \sum_{n \in \mathbb{Z}} a_{(n)} z^{-n-1},
    \]
    satisfying the following axioms:
    \begin{itemize}
        \item \textbf{(Translation covariance)}: \( [T, Y(a,z)] = \partial_z Y(a,z) \),
        \item \textbf{(Vacuum)}: \( Y(\mathbf{1}, z) = \mathrm{Id}_V \), and \( Y(a,z)\mathbf{1} \big|_{z=0} = a \),
        \item \textbf{(Locality)}: for \( N \gg 0 \),
        \[
        (z - w)^N Y(a,z) Y(b,w) = (-1)^{p(a)p(b)} (z - w)^N Y(b,w) Y(a,z).
        \]
    \end{itemize}
\end{itemize}
\end{definition}

A vertex superalgebra \( V \) is called \emph{supercommutative} if \( a_{(n)} = 0 \) for all \( n \in \mathbb{N} \). It is well known that the category of commutative vertex superalgebras is equivalent to the category of unital, commutative, associative superalgebras equipped with an even derivation.

We say that a vertex superalgebra \( V \) is \emph{generated} by a subset \( \mathcal{U} \subset V \) (or by the fields corresponding to elements in \( \mathcal{U} \)) if any element of \( V \) can be written as a finite linear combination of terms of the form
\[
b^{1}_{(i_1)} b^{2}_{(i_2)} \cdots b^{n}_{(i_n)} \mathbf{1},
\]
for \( b^k \in \mathcal{U} \), \( i_k \in \mathbb{Z} \), and \( n \in \mathbb{N} \).  
If every such term can be written with \( i_k \in \mathbb{Z}_{<0} \), we write \( V = \langle \mathcal{U} \rangle_S \) and say that \( V \) is \emph{strongly generated} by \( \mathcal{U} \) (or by the fields corresponding to \( \mathcal{U} \)).

\begin{example}[{\cite{kac1998vertex, zheng2017vertex}}] \label{aflie}
Let \( \mathfrak{g} \) be a finite-dimensional Lie superalgebra equipped with a nondegenerate even supersymmetric invariant bilinear form \( \langle \cdot, \cdot \rangle \). One can associate to the pair \( (\mathfrak{g}, \langle \cdot, \cdot \rangle) \) the affine Lie superalgebra \( \widehat{\mathfrak{g}} \).

Its universal vacuum representation of level \( k \in \mathbb{C} \), denoted \( V^k(\mathfrak{g}) \), is a vertex superalgebra. In particular, when \( \mathfrak{g} \) is a simple Lie algebra or Lie superalgebra of type \( A(m,n), C(n), B(m,n), D(m,n), D(2,1,a), F(4), G(3) \)~\cite{kac2001integrable}, the vacuum module \( V^k(\mathfrak{g}) \) has a unique maximal ideal \( I_{V^k(\mathfrak{g})} \), and the simple quotient
\[
L_k(\mathfrak{g}) := V^k(\mathfrak{g}) / I_{V^k(\mathfrak{g})}
\]
is also a vertex superalgebra.
\end{example}

\begin{definition}
A vertex superalgebra \( V \) is called a \( \frac{1}{2} \mathbb{Z} \)-graded (resp. \( \mathbb{Z} \)-graded) vertex operator superalgebra (VOSA) if it admits a grading
\[
V = \bigoplus_{n \in \frac{1}{2} \mathbb{Z}} V_{(n)} \quad \text{(resp. } V = \bigoplus_{n \in \mathbb{Z}} V_{(n)}),
\]
together with a conformal vector \( \omega \in V \) such that the operators
\[
L_{(n)} := \omega_{(n+1)}, \quad n \in \mathbb{Z},
\]
define a representation of the Virasoro algebra on \( V \), i.e.
\begin{align}\label{viro1}
[L_{(n)}, L_{(m)}] = (m - n) L_{(m + n)} + \frac{m^3 - m}{12} \delta_{m+n,0} \, c_V.
\end{align}
We call \( c_V \) the central charge of \( V \). We require that \( L_{(0)} \) is diagonalizable and determines the \( \frac{1}{2} \mathbb{Z} \)-grading (resp. \( \mathbb{Z} \)); its eigenvalues are called the (conformal) weights.

A field \( Y(v,z) \) is called \emph{primary of conformal weight \( \Delta \)} if
\[
Y(\omega,z)Y(v,w)\sim \frac{\Delta Y(v,w)}{(z-w)^2}+\frac{\partial_w Y(v,w)}{z-w}.\]

The vertex operator superalgebra is said to be of \emph{CFT type} if \( V_{(n)} = 0 \) for all \( n < 0 \), and \( V_{(0)} = \mathbb{C} \mathbf{1} \).
\end{definition}

\begin{example}[{\cite{kac1998vertex}}] \label{N=2}
The universal vertex operator (super)algebras associated with the Virasoro algebra and the \( N=4 \) superconformal algebra are, respectively, \( \mathbb{Z} \)-graded and \( \frac{1}{2} \mathbb{Z} \)-graded. Their simple quotients are denoted by
$L_{\mathrm{Vir}}(c,0)$ and $L^{N=4}_c.$
\end{example}

By~\cite{zhu1996modular}, define
\[
C_2(V) := \mathrm{span}_{\mathbb{C}} \left\{ u_{(-2)} v \mid u,v \in V \right\}.
\]
The quotient space \( R_V := V / C_2(V) \), called the \( C_2 \)-algebra of \( V \), carries the structure of a \emph{Poisson superalgebra}, with multiplication and Poisson bracket given by
\[
\overline{u} \cdot \overline{v} := \overline{u_{(-1)} v}, \qquad \{ \overline{u}, \overline{v} \} := \overline{u_{(0)} v},
\]
where \( \overline{u} := u + C_2(V) \).

We next compute the \( C_2 \)-algebras for several simple examples.

\begin{example}[{\cite{van2021chiral}}] \label{aff sl}
Let \( L_k(\mathfrak{g}) \) be the simple affine vertex algebra at level \( k \in \mathbb{N} \), where \( \mathfrak{g} \) is a simple Lie algebra. Then:
\[
R_{L_k(\mathfrak{g})} = \frac{\mathbb{C}[u^1_{(-1)} \mathbf{1}, \dots, u^n_{(-1)} \mathbf{1}]}{\langle U(\mathfrak{g}) \circ (e_\theta)_{(-1)}^{k+1} \mathbf{1} \rangle},
\]
where \( \{ u^1, \dots, u^n \} \) is a basis of \( \mathfrak{g} \), \( \theta \) is the highest root, and \( \circ \) denotes the adjoint action. In particular, for \( \mathfrak{g} = \mathfrak{sl}_2 \),
\[
R_{L_k(\mathfrak{sl}_2)} \cong \frac{\mathbb{C}[e, f, h]}{\langle f^i \circ e^{k+1} \mid i = 0, \dots, 2k+2 \rangle},
\]
where \( e, f, h \) correspond to \( e_{(-1)} \mathbf{1}, f_{(-1)} \mathbf{1}, h_{(-1)} \mathbf{1} \).
\end{example}

\begin{example}[\cite{van2021chiral}]
Let \( L_{\text{Vir}}(c_{(p,p')},0) \) be the simple Virasoro VOA, where \( c_{(p,p')} = 1 - \frac{6(p - p')^2}{pp'} \) for coprime integers \( p > p' \geq 2 \). Then
\[
R_{L_{\text{Vir}}(c_{(p,p')},0)} \cong \mathbb{C}[x]/\langle x^{\frac{(p-1)(p'-1)}{2}} \rangle,
\]
where \( x = \omega = L_{(-2)} \mathbf{1} \).
\end{example}

\begin{example}[{\cite{adamovic2005construction}}] \label{ex2.11}
The maximal submodule of \( V^{-4/3}(\mathfrak{sl}_2) \) is generated by the singular vector
\[
v_{\mathrm{sing}} = \left( e_{(-1)} L_{(-2)} + \tfrac{1}{3} e_{(-3)} - \tfrac{1}{2} h_{(-1)} e_{(-2)} + \tfrac{1}{2} h_{(-2)} e_{(-1)} \right) \mathbf{1},
\]
where
\[
L_{(-2)} \mathbf{1} = \tfrac{3}{4} \left( e_{(-1)} f_{(-1)} + f_{(-1)} e_{(-1)} + \tfrac{1}{2} h_{(-1)}^2 \right)
\]
is the conformal vector. Then
\[
R_{L_{-4/3}(\mathfrak{sl}_2)} = \frac{\mathbb{C}[x, y, z]}{\langle x(xy - z^2),\, y(xy - z^2),\, z(xy - z^2) \rangle}.
\]
\end{example}

\section{Quasi-lisse vertex superalgebras}

In this section, we define quasi-lisse vertex superalgebras and discuss their $g$-twisted modules.

\subsection{Definitions}

We begin by recalling the notion of a Poisson module.

\begin{definition}
Let $A$ be a Poisson algebra. An $A$-module $M$ is a module over the associative algebra $A$ equipped with a bilinear map
\[
\{\cdot, \cdot\}: A \times M \rightarrow M,
\]
satisfying, for all $x, y \in A$ and $m \in M$,
\begin{itemize}
    \item $\{x,y\}m = x(y m) - y(x m),$
    \item $\{x, y m\} = \{x, y\} m + y \{x, m\},$
    \item $\{x y, m\} = x \{y, m\} + y \{x, m\}.$
\end{itemize}
\end{definition}

If $M_1$ and $M_2$ are $A$-modules, then $M_1 \otimes M_2$ is also an $A$-module, with Poisson bracket defined by
\[
\{a, m_1 \otimes m_2\} = \{a, m_1\} \otimes m_2 + m_1 \otimes \{a, m_2\},
\]
for $a \in A$, $m_i \in M_i$.

Let $X$ be an affine Poisson variety, i.e., $\mathcal{O}(X)$ is a Poisson algebra. For any $f \in \mathcal{O}(X)$, one defines the Hamiltonian vector field $\mathcal{E}_f := \mathrm{ad}(f) = \{f, \cdot\}$. This gives rise to a map:
\[
\alpha: \mathcal{O}(X) \rightarrow \mathrm{Vect}(X), \quad \alpha(f) = \mathcal{E}_f,
\]
where $\mathrm{Vect}(X)$ denotes the space of vector fields on $X$.

\begin{definition}
A maximal locally closed connected subvariety $Z \subset X$ is called a \emph{symplectic leaf} if $\alpha_x: \mathcal{O}(X) \rightarrow T_x X$ is surjective for all $x \in Z$.
\end{definition}

We now recall a fundamental result of Etingof–Schedler \cite{etingof2010poisson}.
\begin{theorem}[\cite{etingof2010poisson}]\label{leaves}
Let \( X \) be an affine Poisson variety, and let \( G \) be a finite group of Poisson automorphisms of \( X \). Assume that \( X \) has finitely many symplectic leaves. Then
\begin{itemize}
    \item[(1)] For any coherent sheaf \( N \) of Poisson \( \mathcal{O}_X \)-modules, the quotient \( N / \{ \mathcal{O}_X, N \} \) is finite-dimensional.
    \item[(2)] $X^G$ has finitely many symmplectic leaves. The invariant quotient $\mathcal{O}_X / \{\mathcal{O}_{X/G}, \mathcal{O}_X\}$ is finite-dimensional. In particular, 
    \[
    \left((\mathcal{O}_X)^{G}\right)/\left\{(\mathcal{O}_X)^{G}, (\mathcal{O}_X)^{G}\right\}
    \]
    is finite-dimensional.\end{itemize}
Moreover, let \( A \) be a nonnegatively filtered associative algebra such that \( \operatorname{gr}(A) \) is a finitely generated module over its center \( Z \), and assume that \( \operatorname{Specm}(Z) \) has finitely many symplectic leaves. Then
\begin{itemize}
    \item[(3)] The algebra \( A \) has finitely many irreducible finite-dimensional representations.
\end{itemize}
\end{theorem}

\begin{example}
\begin{itemize}
    \item[\cite{chriss1997representation}] The algebra $\mathcal{O}(M)$ of regular functions on a symplectic manifold $M$ has a canonical Poisson structure. Each connected component of $M$ is a single symplectic leaf. Examples include $\mathbb{C}^{2n}$, cotangent bundles of smooth manifolds, and coadjoint orbits in $\mathfrak{g}^*$.
    \item[\cite{etingof2018poisson}] If a normal variety $X$ admits a symplectic resolution $\rho: \tilde{X} \rightarrow X$ (i.e., $\tilde{X}$ has a global nondegenerate closed 2-form), then $X$ has finitely many symplectic leaves. Examples include Nakajima quiver varieties, nilpotent cones, Slodowy slices, Kleinian singularities, and $\mathbb{C}^2/S_n$.
    \item (Conjectural) For vertex algebras arising from 4d SCFTs, the associated variety is expected to have finitely many symplectic leaves. This includes, for example, genus-zero Moore–Tachikawa symplectic varieties.
\end{itemize}
\end{example}

\begin{definition}[\cite{arakawa2012remark, arakawa2018quasi}]
A finitely strongly generated vertex (super)algebra $V$ is said to be:
\begin{itemize}
    \item \emph{lisse} if $\dim \mathrm{Spec}(R_V) = 0$;
    \item \emph{quasi-lisse} if the reduced Poisson variety $(\tilde{X}_V)_{\mathrm{red}}$ has finitely many symplectic leaves. 
\end{itemize}
\end{definition}

\begin{example}[\cite{gorelik2007simplicity}]
The affine vertex superalgebra $L_k(\mathfrak{g})$ is lisse if and only if $\mathfrak{g}$ is a simple Lie algebra or a Lie superalgebra of type $B(0,n)$ with $n, k \in \mathbb{Z}_{+}$.
\end{example}

We now explain how Theorem~\ref{leaves} specializes to the setting of vertex operator (super)algebras. In particular, we formulate a corresponding statement for Zhu's \( C_2 \)-algebra \( R_V \) and its modules $\mathcal{M}$.
\begin{corollary}\label{leaves'}
Let $V$ be a quasi-lisse vertex superalgebra,
and let $G$ be a finite group of automorphisms of $R_V$. Suppose $\mathcal{M}$ is a finitely generated $R_V$-module. Then
\begin{itemize}
    \item $\dim(\mathcal{M} / \{R_V, \mathcal{M}\}) < \infty$,
    \item $\mathrm{Specm}((R_V)^G)$ has finitely many symplectic leaves, and 
    \[
    \dim(R_V / \{(R_V)^G, R_V\}) < \infty.
    \]
\end{itemize}
\end{corollary}
\begin{proof}
Since \( V \) is quasi-lisse, \( \mathrm{Specm}((R_V)_{\overline{0}}) \) has finitely many symplectic leaves. Moreover, \( V \) is strongly finitely generated, so \( R_V \) is finitely generated over \( \mathbb{C} \), and hence Noetherian. Thus, coherent sheaves on \( \mathrm{Specm}((R_V)_{\overline{0}}) \) correspond to finitely generated \((R_V)_{\overline{0}}\)-modules.

By assumption, \( \mathcal{M} \) is finitely generated over \( R_V \). Since \( R_V \) is Noetherian and \((R_V)_{\overline{1}}\) is finitely generated over \((R_V)_{\overline{0}}\), it follows that \( \mathcal{M} \) is finitely generated over \((R_V)_{\overline{0}}\). Applying Theorem~\ref{leaves}, we obtain
\[
\dim(\mathcal{M} / \{(R_V)_{\overline{0}}, \mathcal{M}\}) < \infty.
\]

Since \( G \) is finite, \( R_V \) is finitely generated over \((R_V)^G\), and hence over \(((R_V)^G)_{\overline{0}}\). Applying Theorem~\ref{leaves} again yields
\[
\dim(R_V / \{((R_V)^G)_{\overline{0}}, R_V\}) < \infty.
\]
In particular,
\[
\dim(\mathcal{M} / \{R_V, \mathcal{M}\}) < \infty, \quad \dim(R_V / \{(R_V)^G, R_V\}) < \infty.
\]
\end{proof}

\begin{remark}\label{ESZhu}
In the following, we shall also apply Theorem~\ref{leaves} (3) in the vertex superalgebra setting. Specifically, we consider the associative algebra \( A_g(V) \) equipped with a non-negative filtration. Under the assumption that \( V \) is quasi-lisse, it follows from Lemma~\ref{epi} that \( \operatorname{Specm}(\operatorname{gr} A_g(V)) \) has finitely many symplectic leaves. Consequently, by Theorem~\ref{leaves} (3), the algebra \( A_g(V) \) has finitely many irreducible finite-dimensional representations.
\end{remark}

\subsection{Twisted modules}
Following~\cite{dong2006twisted}, we define an automorphism \( g \) of a vertex operator superalgebra \( V \), and the notion of an ordinary \( g \)-twisted \( V \)-module.

\begin{definition}
An automorphism \( g \) of \( V \) is a linear automorphism preserving the conformal vector \( \omega \) and satisfying the compatibility condition
\[
g Y(v,z) g^{-1} = Y(gv, z), \quad \text{for all } v \in V.
\]
\end{definition}

We denote the automorphism group of \( V \) by \( \mathrm{Aut}\,V \). Let \( \sigma \) be the canonical parity automorphism defined by \( \sigma|_{V_{\overline{i}}} = (-1)^i \, \mathrm{id}_V \). Let \( T \) and \( T' \) denote the orders of \( g \) and \( g\sigma \), respectively. Then we have the eigenspace decompositions:
\begin{align*}
    V &= \bigoplus_{r \in \mathbb{Z}/T'\mathbb{Z}} V^{r*}, \quad \text{where } V^{r*} = \left\{ v \in V \;\middle|\; g\sigma v = e^{-2\pi i r/T'} v \right\}, \\
    V &= \bigoplus_{r \in \mathbb{Z}/T\mathbb{Z}} V^r, \quad \text{where } V^r = \left\{ v \in V \;\middle|\; gv = e^{-2\pi i r/T} v \right\}.
\end{align*}

\begin{definition}[Ordinary twisted module]\label{4.7}
Let \( T \) be the order of \( g \). A (weak) \( g \)-twisted \( V \)-module is a complex vector space \( M \) equipped with a linear map
\[
V \to \mathrm{End}(M)[[z^{1/T}, z^{-1/T}]], \quad v \mapsto Y_M(v,z) = \sum_{n \in \mathbb{Q}} v_{(n)} z^{-n-1},
\]
such that
\begin{itemize}
    \item For all \( v \in V \), \( w \in M \), we have \( v_{(n)} w = 0 \) for sufficiently large \( n \).
    \item \( Y_M(\mathbf{1}, z) = \mathrm{id}_M \).
    \item For \( v \in V^r \) with \( 0 \leq r < T \), we have
    \[
    Y_M(v, z) = \sum_{n \in \frac{r}{T} + \mathbb{Z}} v_{(n)} z^{-n-1}.
    \]
    \item (\textbf{Twisted Jacobi identity}) For \( u \in V^r \), the following identity holds:
\begin{align*}
    & z_0^{-1} \delta\left( \frac{z_1 - z_2}{z_0} \right) Y_M(u, z_1) Y_M(v, z_2) \\
    &\quad - (-1)^{|u||v|} z_0^{-1} \delta\left( \frac{z_2 - z_1}{-z_0} \right) Y_M(v, z_2) Y_M(u, z_1) \\
    &= z_2^{-1} \left( \frac{z_1 - z_0}{z_2} \right)^{-r/T} 
    \delta\left( \frac{z_1 - z_0}{z_2} \right) 
    Y_M(Y(u, z_0)v, z_2).
\end{align*}
    \item (\textbf{Grading condition}) \( M \) admits a decomposition
    \[
    M = \bigoplus_{\lambda \in \mathbb{C}} M_\lambda, \quad \text{where } M_\lambda = \{ w \in M \mid L_{(0)} w = \lambda w \},
    \]
    such that each \( M_\lambda \) is finite-dimensional, and for fixed \( \lambda \), \( M_{\lambda + \frac{n}{T'}} = 0 \) for \( n \ll 0 \).
\end{itemize}
\end{definition}

\begin{remark}\label{4.8}
If the grading condition above is dropped, one refers to \( M \) as a weak \( g \)-twisted module. If \( M \) carries a \( \tfrac{1}{T'} \mathbb{N} \)-grading,
\[
M = \bigoplus_{n \in \tfrac{1}{T'} \mathbb{N}} M_{(h+n)} := \bigoplus_{n \in \tfrac{1}{T'} \mathbb{N}} M(n),
\]
and satisfies
\[
v_{(m)} M(n) \subset M(n + \mathrm{wt}(v) - m - 1),
\]
for homogeneous \( v \in V \), then \( M \) is called an admissible \( g \)-twisted module. 
\end{remark}

Finally, set
\[
Y_M(\omega, z) = \sum_{n \in \mathbb{Z}} L_{(n)} z^{-n-2}.
\]
Then, for all \( v \in V \),
\[
Y_M(L_{(-1)} v, z) = \frac{d}{dz} Y_M(v, z),
\]
and the operators \( L_{(n)} \) satisfy the Virasoro algebra relations with central charge \( c \) (see~\cite{dong1995regularity}).

Let $M$ be a weak $V$-module (as in Definitions~\ref{4.7} and~\ref{4.8}). Define
\begin{align}
&C_1(M) = \mathrm{span}_{\mathbb{C}} \{ \omega_{(0)} m,\, a_{(-1)} m \mid a \in \bigoplus_{\Delta>0} V_\Delta,\, m \in M \}, \\
&C_2(M) = \mathrm{span}_{\mathbb{C}} \{ a_{(-2)} m \mid a \in V,\, m \in M \}.
\end{align}

We say $M$ is $C_1$-cofinite (resp. $C_2$-cofinite) if $\dim(M / C_1(M)) < \infty$ (resp. $\dim(M / C_2(M)) < \infty$). Define $R_M := M / C_2(M)$, which naturally inherits a Poisson $R_V$-module structure:
\begin{align*}
\bar{a} \cdot \bar{m} = \overline{a_{(-1)} m}, \quad \{ \bar{a}, \bar{m} \} = \overline{a_{(0)} m},
\end{align*}
for $a \in V$, $m \in M$. We say that $M$ is \emph{strongly generated over $V$} if $R_M$ is finitely generated as an $R_V$-module.

\subsection{Ordinary twisted modules}

Let $V = \bigoplus_{n \in \frac{1}{2}\mathbb{Z}} V_{(n)}$ be a vertex superalgebra, where
\[
V_{\overline{0}} = \bigoplus_{n \in \mathbb{Z}} V_{(n)}, \quad V_{\overline{1}} = \bigoplus_{n \in \mathbb{Z} + \frac{1}{2}} V_{(n)}.
\]
Following~\cite{dong2006twisted}, we define twisted Zhu algebras as follows. For homogeneous $a, b \in V$, define
\begin{align}
\label{zhu1}
a *_g b &= 
\begin{cases}
\sum_{i \in \mathbb{N}} \binom{\mathrm{wt}(a)}{i} a_{(i-1)} b, & \text{if } a, b \in V^{0*}, \\
0, & \text{if } a \text{ or } b \in V^{r*},\ r \neq 0,
\end{cases} \\
\label{zhu2}
a \circ_g b &= 
\begin{cases}
\sum_{i \in \mathbb{N}} \binom{\mathrm{wt}(a)}{i} a_{(i-2)} b, & \text{if } a \in V^{0*}, \\
\sum_{i \in \mathbb{N}} \binom{\mathrm{wt}(a) - 1 + \frac{r}{T'}}{i} a_{(i-1)} b, & \text{if } a \in V^{r*},\ r \neq 0.
\end{cases}
\end{align}

Let $O_g(V)$ denote the linear span of elements of the form $a \circ_g b$ in $V$. The twisted Zhu algebra $A_g(V)$ is defined as the quotient space $V / O_g(V)$, with multiplication induced from $*_g$. According to~\cite{de2006finite}, there exists a filtration $\{\overline{F}_k(A_g(V))\}$ defined by
\[
\overline{F}_k(A_g(V)) := \left( \bigoplus_{i \in \frac{1}{2} \mathbb{Z},\, i \leq k} V_{(i)} + O_g(V) \right) / O_g(V).
\]
Its associated graded algebra is
\[
\mathrm{gr}^{\overline{F}}(A_g(V)) := \bigoplus_{k \geq 0} \overline{F}_k(A_g(V)) / \overline{F}_{k-1}(A_g(V)),
\]
which is a commutative algebra. Define $\overline{F}_k(V) := \bigoplus_{i \in \frac{1}{2} \mathbb{Z},\, i \leq k} V_{(i)}$.

\begin{lemma}\label{epi}
There exists a surjective Poisson algebra homomorphism
\[
\mathcal{G}: (R_V)^{\sigma g} \twoheadrightarrow \mathrm{gr}^{\overline{F}}(A_g(V)),
\]
given by $\mathcal{G}(a + (C_2(V))_{(p)}) = a + O_g(V) + \overline{F}_{p-1}(V)$ for $a \in V^{\sigma g}$.
\end{lemma}

\begin{proof}
To show that $\mathcal{G}$ is well-defined, it suffices to prove that
\[
(C_2(V))_{(p)} \subset O_g(V) + \overline{F}_{p-1}(V).
\]
Each element in $(C_2(V))_{(p)}$ is a linear combination of elements of the form $a_{(-2)} b$ with homogeneous $a, b \in V$ and $\mathrm{wt}(a_{(-2)} b) = p$. For such elements, we have $\mathrm{wt}(a_{(i-2)} b) \leq p - 1$ for all $i \in \mathbb{N}$. Hence
\[
a_{(-2)} b \equiv 
\begin{cases}
a \circ_g b \mod \overline{F}_{p-1}(V), & \text{if } a \in V^{0*}, \\
0 \mod \overline{F}_{p-1}(V) + O_g(V), & \text{if } a \in V^{r*},\, r \neq 0.
\end{cases}
\]
Thus, $\phi$ is well-defined.

To see that $\mathcal{G}$ is a Poisson algebra homomorphism, observe that for $a, b \in V^{\sigma g}$ of conformal weights $p$ and $q$, and for any representatives $u \in a + (C_2(V))_{(p)}$, $v \in b + (C_2(V))_{(q)}$, we have
\begin{align*}
u_{(-1)} v &\in a_{(-1)} b + (C_2(V))_{(p+q)}, \\
u_{(0)} v &\in a_{(0)} b + (C_2(V))_{(p+q-1)}.
\end{align*}
Meanwhile, using the definition of $*_g$, one sees
\begin{align*}
u *_g v &\in a_{(-1)} b + O_g(V) + \overline{F}_{p+q-1}(V), \\
u *_g v - v *_g u &\in a_{(0)} b + O_g(V) + \overline{F}_{p+q-2}(V).
\end{align*}
This shows that $\mathcal{G}$ respects both the multiplication and the Poisson bracket. Surjectivity follows from the fact that $V^{r*} \subset O_g(V)$ when $r \neq 0$.
\end{proof}

\begin{theorem}\label{finitelymanyordinary}
Let $V$ be a quasi-lisse vertex operator superalgebra. Then $V$ admits only finitely many isomorphism classes of simple ordinary $g$-twisted modules.
\end{theorem}

\begin{proof}
By Lemma~\ref{epi}, we have
\[
\mathrm{Specm}(\mathrm{gr}^{\overline{F}}(A_g(V))) \subset \mathrm{Specm}((R_V)^{\sigma g}) = (X_V)_{\mathrm{red}}^{\langle \sigma g \rangle}.
\]
Since $V$ is quasi-lisse, and by Corollary \ref{leaves'}, the right-hand side has finitely many symplectic leaves. Then, by Theorem~\ref{leaves}, together with the explanation in Remark~\ref{ESZhu}, the twisted Zhu algebra \( A_g(V) \) has finitely many finite-dimensional simple modules. The conclusion follows from~\cite[Theorem 7.2]{dong1995twisted}.
\end{proof}

\begin{example}[\cite{LLY2023spectral}]\label{twistedzhualgebra}
The $e^{\frac{1}{2} \pi i h_{(0)}}$-twisted Zhu algebra of $L_{-\frac{4}{3}}(\mathfrak{sl}_2)$ is
\[
\mathbb{C}[x] / \langle x(x - \tfrac{2}{3})(x + \tfrac{2}{3}) \rangle.
\]
There are three simple twisted modules, with Dynkin labels (i.e., eigenvalues of $h_{(0)}$ on the highest weight vectors) given by $0$, $\pm \frac{2}{3}$.
\end{example}

\section{Genus-zero correlation functions}

\subsection{Intertwining operators among weak \texorpdfstring{$g$}{Lg}-twisted modules}
Let  $(W_{1},Y_{1})$ be a weak module of vertex operator superalgebra $V$, and $(W_{j},Y_{j}),$ $j=2,3$ be $g$-twisted $V$-modules, where the order of $g$ is $T$.  Assume that $W_i$, $i=1,2,3$ are $\mathbb{Z}_{2}$-graded such that $v_{(n)}w\in (W_i)_{|v|+|w|}$ for $v\in V_{|v|}$ and $w\in (W_i)_{|w|}$.

\begin{definition}
An intertwining operator of type $\binom{W_{3}}{W_{1}\;W_{2}}$ is a linear map
\begin{align*}
    \mathcal{Y}(\cdot,x)\cdot :W_{1}\otimes W_2 \rightarrow W_3\{x\}
\end{align*}
such that for $w_{1}\in W_{1}$, $w_{2}\in W_{2}$,
\begin{align}
    (w_{1})_{(n)}w_{2}=0\quad \text{for $n$ sufficiently large},
\end{align}
where $\mathcal{Y}(w_1,x)=\sum_{n\in \mathbb{Q}}(w_1)_{(n)}x^{-n-1};$
the Jacobi identity
\begin{align}\label{iof1}
\begin{split}
  &  x_0^{-1}\delta\left(\frac{x_1-x_2}{x_0}\right)Y_{3}(u,x_{1})\mathcal{Y}(w_1,x_2) \\
  &  \quad -(-1)^{|u||w_1|}x_0^{-1}\delta\left(\frac{-x_2+x_1}{x_0}\right)\mathcal{Y}(w_1,x_2)Y_{2}(u,x_1) \\
  &= x_2^{-1}\delta\left(\frac{x_1-x_0}{x_2}\right)\left(\frac{x_1-x_0}{x_2}\right)^{-\frac{k}{T}}\mathcal{Y}(Y_{1}(u,x_0)w_1,x_2)
\end{split}
\end{align}
holds for $u\in V^{k}$;
\begin{align}
    \frac{d}{dx}\mathcal{Y}(w_{1},x)=\mathcal{Y}(L_{(-1)}w_1,x),
\end{align}
for $w_1\in W_1$.
\end{definition}

Taking $\text{Res}_{x_{0}}$ of (\ref{iof1}), one obtains the commutator formula
\begin{align}\label{commutator1}
    \begin{split}
        &Y_{3}(u,x_{1})\mathcal{Y}(w_1,x_2) - (-1)^{|u||w_1|}\mathcal{Y}(w_1,x_2)Y_{2}(u,x_1) \\
        &\quad = \text{Res}_{x_{0}} x_1^{-1}\delta\left(\frac{x_2+x_0}{x_1}\right)\left(\frac{x_2+x_0}{x_1}\right)^{\frac{k}{T}} \mathcal{Y}(Y_{1}(u,x_0)w_1,x_2)
    \end{split}
\end{align}

Taking $\text{Res}_{x_{1}}$ of (\ref{iof1}), one obtains the associator formula
\begin{align}\label{associator}
    \begin{split}
        &(x_2+x_0)^{\frac{k}{T}}\mathcal{Y}(Y_1(u,x_0)w_1,x_2) - Y_3^\circ(u,x_0+x_2)\mathcal{Y}(w_1,x_2) \\
        &\quad = -\text{Res}_{x_{1}}(-1)^{|u||w_1|}x_0^{-1}\delta\left(\frac{-x_2+x_1}{x_0}\right)\mathcal{Y}(w_1,x_2)Y_2^\circ(u,x_1)
    \end{split}
\end{align}
where $Y^\circ_i(u,x)=x^{\frac{k}{T}}Y_i(u,x)\in \text{End}M[[ x,x^{-1} ]].$

By using a similar argument as in \cite[Theorem 2.4]{huang2010generalized}, one has the following. 

\begin{theorem}
The twisted Jacobi identity (\ref{iof1}) is equivalent to the following property: for any $u\in V$, $w_1\in W_1$, $w_2\in W_2$ and $w'\in (W_3)'$, there exists a multivalued analytic function of the form
\[
    f(z_1,z_2)=\sum_{r,s=N_1}^{N_2}a_{rs}z_{1}^{r/k}z_2^{s/k}(z_1-z_2)^{-N}
\]
such that the series
\begin{align*}
    & \langle w', Y_3(u,z_1)\mathcal{Y}(w_1,z_2)w_2 \rangle,\\
    & (-1)^{|u||w_1|} \langle w', \mathcal{Y}(w_1,z_2)Y_2(u,z_{1})w_2\rangle,\\
    & \langle w',\mathcal{Y}(Y_1(u,z_1-z_2)w_1,z_2)w_2\rangle
\end{align*}
are absolutely convergent in the regions $|z_1|>|z_2|>0$, $|z_2|>|z_1|>0$, $|z_2|>|z_1-z_2|>0$, respectively, and converge to the branch
\begin{align*}
    \sum_{r,s=N_1}^{N_2}a_{rs}e^{(r/k)\log(z_1)}e^{(s/k)\log(z_2)}(z_1-z_2)^{-N}\end{align*}of $f(z_1,z_2)$, where $\log z_i=\log |z|+i\operatorname{arg}z_i$, $i=1,2$ and $0\leq   \operatorname{arg}z_i<2\pi.$
\end{theorem}
\noindent This property is called the duality property.

\begin{conjecture}\label{duality}
 Let $V$ be a quasi-lisse vertex superalgebra. Let $W_1$, $W_2$ be finitely strongly generated $V$-modules, and $ W_0$, $W_3$ be an ordinary $g$-twisted modules. Let $W_4$, $W_6$ be weak $g$-twisted modules, and $W_5$ be a weak $V$-module. Let $\Y_1,\Y_2,\Y_3,\Y_4,\Y_5,\Y_6$ be intertwining operators of types $\binom{W_0'}{W_1\;W_4},\binom{W_4}{W_2\;W_3},\binom{W_5}{W_1\;W_2},\binom{W_0'}{W_5\;W_3},\binom{W_0'}{W_2\;W_6},\binom{W_6}{W_1\;W_3}$, respectively.  For $w_1\in W_1$, $w_2\in W_2$, $w_3\in W_3$, and $w_0\in W'_0$, there exists a maximally extended multivalued analytic function $f(z_1,z_2;w_1,w_2,w_3,w_4)$ defined on $\{(z_1,z_2)\mid z_i\neq 0,\; z_i\neq z_j,\; i\neq j\}$ such that the series
 \begin{align*}
 &\langle w_0,\Y(w_1,z_1)\Y(w_2,z_2)w_3\rangle,\\
 &\langle w_0,\Y_4(\Y_3(w_1,z_1-z_2)w_2,z_2)w_3\rangle,\\
 &(-1)^{|w_1||w_2|}\langle w_0,\Y_5(w_2,z_2)\Y_6(w_1,z_1)w_3\rangle	
 \end{align*}
are absolutely convergent in the respective regions $|z_1|>|z_2|>0$, $|z_2|>|z_1-z_2|>0$, and $|z_2|>|z_1|>0$ to some branch of $f(z_1,z_2;w_1,w_2,w_3,w_4)$. 
\end{conjecture}

\subsection{Convergence of genus-zero correlation functions}
In this subsection, we prove that under the assumptions of Conjecture \ref{duality} and in the untwisted case $g = 1$, the corresponding correlation functions are indeed convergent. Let $V$ be a quasi-lisse vertex operator superalgebra. Let $W_{i}$, $i=0,1,2,3$, be weak $\mathbb{Q}$-graded $V$-modules, and set $R = \mathbb{C}[z_{1}^{\pm 1}, z_{2}^{\pm 1}, (z_{1}-z_{2})^{-1}]$. Define $\T = R \otimes W_{0} \otimes W_{1} \otimes W_{2} \otimes W_{3}$, which has a natural $R$-module structure. Let $W' = \bigoplus_{n\in \mathbb{Q}} W(n)^{*}$ be the {\em contragredient module} of $W$, whose vertex operator map $Y'$ is defined by
\[
\langle Y'(u,x)w',w\rangle = \langle w', Y_{W}(e^{xL_{(1)}}(-x^{-2})^{L_{(0)}}u, x^{-1})w \rangle,
\]
for any $u \in V$, $w' \in W'$, and $w \in W$. In particular, for $u \in V$ and $n \in \mathbb{Q}$, we have the operators $u_{(n)}$ on $W$. Let $u_{(n)}^{*}$ denote the adjoint of $u_{(n)}$ acting on $W'$. Note that $\operatorname{wt} u_{(n)} = \operatorname{wt} u - n - 1$ and $\operatorname{wt} u_{(n)}^{*} = -\operatorname{wt} u + n + 1$.

Let $\mathcal{Y}_{1}$ and $\mathcal{Y}_{2}$ be intertwining operators of types $\binom{W_{0}}{W_{1}\,\,W_{5}}$ and $\binom{W_{5}}{W_{2}\,\,W_{3}}$, respectively. By (\ref{commutator1}),
\begin{align}\label{commutator}
    \begin{split}
        \mathcal{Y}_{2}(w_{2}, x_{2}) Y_{W_{3}}(u, y_{1}) w_{3} &= Y_{W_{5}}(u, y_{1}) \mathcal{Y}_{2}(w_{2}, x_{2}) w_{5} \\
        &\quad + \operatorname{Res}_{y_{2}} y_{1}^{-1} \delta\left(\frac{x_{2} + y_{2}}{y_{1}}\right) \mathcal{Y}_{2}(Y_{W_{2}}(u, y_{2}) w_{2}, x_{2}) w_{5}.
    \end{split}
\end{align}
Taking $\operatorname{Res}_{y_{1}}$ of both sides yields
\[
\mathcal{Y}_{2}(w_{2}, x_{2}) u_{(0)} w_{3} = u_{(0)} \mathcal{Y}_{2}(w_{2}, x_{2}) w_{5} + \mathcal{Y}_{2}(u_{(0)} w_{2}, x_{2}) w_{5}.
\]
Applying the commutator formula again, we get
\begin{align}\label{relatione}
    \begin{split}
        \langle w_{4}', \mathcal{Y}_{1}(w_{1}, x_{1}) \mathcal{Y}_{2}(w_{2}, x_{2}) u_{(0)} w_{3} \rangle
        &= \langle u_{(0)}^{*} w_{4}', \mathcal{Y}_{1}(w_{1}, x_{1}) \mathcal{Y}_{2}(w_{2}, x_{2}) w_{3} \rangle \\
        &\quad - \langle w_{4}', \mathcal{Y}_{1}(u_{(0)} w_{1}, x_{1}) \mathcal{Y}_{2}(w_{2}, x_{2}) w_{3} \rangle \\
        &\quad - \langle w_{4}', \mathcal{Y}_{1}(w_{1}, x_{1}) \mathcal{Y}_{2}(u_{(0)} w_{2}, x_{2}) w_{3} \rangle.
    \end{split}
\end{align}
For any $u \in \bigoplus_ {n\in \Z_+} V_{(n)}$ and $w_i \in W_i$ for $i = 0,1,2,3$, let $\J$ be the submodule of $\mathcal{T}$ generated by elements of the following forms:
{\small 
\allowdisplaybreaks
\begin{align*}
\mathcal{A}(u, w_0, w_1, w_2, w_3) &= \sum_{k \geq 0} \binom{-1}{k} (-z_1)^k u_{(-1-k)}^{*} w_0 \otimes w_1 \otimes w_2 \otimes w_3 \\
&\quad - w_0 \otimes u_{(-1)} w_1 \otimes w_2 \otimes w_3 \\
&\quad - \sum_{k \geq 0} \binom{-1}{k} (-(z_1 - z_2))^{-1-k} \otimes w_0 \otimes w_1 \otimes u_{(k)} w_2 \otimes w_3 \\
&\quad - \sum_{k \geq 0} \binom{-1}{k} (-z_1)^{-1-k} \otimes w_0 \otimes w_1 \otimes w_2 \otimes u_{(k)} w_3,
\end{align*}

\begin{align*}
\mathcal{B}(u, w_0, w_1, w_2, w_3) &= \sum_{k \geq 0} \binom{-1}{k} (-z_2)^k u_{(-1-k)}^{*} w_0 \otimes w_1 \otimes w_2 \otimes w_3 \\
&\quad - \sum_{k \geq 0} \binom{-1}{k} (z_1 - z_2)^{-1-k} \otimes w_0 \otimes u_{(k)} w_1 \otimes w_2 \otimes w_3 \\
&\quad - w_0 \otimes w_1 \otimes u_{(-1)} w_2 \otimes w_3 \\
&\quad - \sum_{k \geq 0} \binom{-1}{k} (-z_2)^{-1-k} \otimes w_0 \otimes w_1 \otimes w_2 \otimes u_{(k)} w_3,
\end{align*}

\begin{align*}
\mathcal{C}(u, w_0, w_1, w_2, w_3) &= u_{(-1)}^{*} w_0 \otimes w_1 \otimes w_2 \otimes w_3 \\
&\quad - \sum_{k \geq 0} \binom{-1}{k} z_1^{-1-k} \otimes w_0 \otimes u_{(k)} w_1 \otimes w_2 \otimes w_3 \\
&\quad - \sum_{k \geq 0} \binom{-1}{k} z_2^{-1-k} \otimes w_0 \otimes w_1 \otimes u_{(k)} w_2 \otimes w_3 \\
&\quad - w_0 \otimes w_1 \otimes w_2 \otimes u_{(-1)} w_3,
\end{align*}

\begin{align*}
\mathcal{D}(u, w_0, w_1, w_2, w_3) &= u_{(-1)} w_0 \otimes w_1 \otimes w_2 \otimes w_3 \\
&\quad - \sum_{k \geq 0} \binom{-1}{k} z_1^{1+k} w_0 \otimes 
\big[ e^{z_1^{-1} L_{(1)}} (-z_1^2)^{L_{(0)}} u_{(k)} \\
&\hspace{12em} \cdot (-z_1^{-2})^{L_{(0)}} e^{-z_1^{-1} L_{(1)}} w_1 \big]
\otimes w_2 \otimes w_3 \\
&\quad - \sum_{k \geq 0} \binom{-1}{k} z_2^{1+k} w_0 \otimes w_1 \otimes 
\big[ e^{z_2^{-1} L_{(1)}} (-z_2^2)^{L_{(0)}} u_{(k)} \\
&\hspace{12em} \cdot (-z_2^{-2})^{L_{(0)}} e^{-z_2^{-1} L_{(1)}} w_2 \big]
\otimes w_3 \\
&\quad + w_0 \otimes w_1 \otimes w_2 \otimes u_{(-1)}^{*} w_3,
\end{align*}

\begin{align*}
\mathcal{E}(u, w_0, w_1, w_2, w_3) &= u_{(0)}^{*} w_0 \otimes w_1 \otimes w_2 \otimes w_3 \\
&\quad - w_0 \otimes u_{(0)} w_1 \otimes w_2 \otimes w_3 \\
&\quad - w_0 \otimes w_1 \otimes u_{(0)} w_2 \otimes w_3 \\
&\quad - w_0 \otimes w_1 \otimes w_2 \otimes u_{(0)} w_3.
\end{align*}
}

The weights on $W_i$ for $i=0,1,2,3$ induce a weight grading on $W_0 \otimes W_1 \otimes W_2 \otimes W_3$. Let $\mathcal{T}_{(r)}$ denote the homogeneous subspace of $\mathcal{T}$ of weight $r \in \mathbb{N}$. This defines a filtration $I$ on $\mathcal{T}$ by
\[
I^n(\T) = \bigoplus_{i \leq n} \T_{(i)}.
\]
We define the associated graded space with respect to this filtration as $\mathrm{gr}^I(\T)$. This filtration also induces a filtration on $\J$ by setting $I^n(\J) = I^n(\T) \cap \J$.

\begin{lemma}\label{quasilissecondition}
Suppose $V$ is of CFT type, self-dual, and finitely generated by weight $1$ vectors. If
\begin{align}\label{newfinitenss}
\dim \left( R_{W_0} \otimes R_{W_1} \otimes R_{W_2} \otimes R_{W_3} \big/ \{ R_V, R_{W_0} \otimes R_{W_1} \otimes R_{W_2} \otimes R_{W_3} \} \right) < \infty,
\end{align}
where the tensor product is taken over $R_V$, then the $R$-module $\T/\J$ is finitely generated.
\end{lemma}

\begin{proof}
The equivalence classes of the elements
\[
\begin{split}
&\mathcal{A}(u, w_0, w_1, w_2, w_3),\quad \mathcal{B}(u, w_0, w_1, w_2, w_3),\\
&\mathcal{C}(u, w_0, w_1, w_2, w_3),\quad \mathcal{D}(u, w_0, w_1, w_2, w_3)
\end{split}
\]
in $\mathrm{gr}^I(\T)$ are
\begin{align*}
& -w_0 \otimes u_{(-1)} w_1 \otimes w_2 \otimes w_3, \\
& -w_0 \otimes w_1 \otimes u_{(-1)} w_2 \otimes w_3, \\
& -w_0 \otimes w_1 \otimes w_2 \otimes u_{(-1)} w_3, \\
& u_{(-1)} w_0 \otimes w_1 \otimes w_2 \otimes w_3.
\end{align*}
Replacing $u$ by $L_{(-1)} u$ shows that the following elements also lie in $\mathrm{gr}^I(\mathcal{J})$:
\begin{align*}
& w_0 \otimes u_{(-2)} w_1 \otimes w_2 \otimes w_3, \\
& w_0 \otimes w_1 \otimes u_{(-2)} w_2 \otimes w_3, \\
& w_0 \otimes w_1 \otimes w_2 \otimes u_{(-2)} w_3, \\
& u_{(-2)} w_0 \otimes w_1 \otimes w_2 \otimes w_3.
\end{align*}
Now, define a map
\[
\Theta: \left( R_{W_0} \otimes R_{W_1} \otimes R_{W_2} \otimes R_{W_3} \right) \big/ \{ R_V, R_{W_0} \otimes R_{W_1} \otimes R_{W_2} \otimes R_{W_3} \} \longrightarrow \T / \mathrm{gr}^I(\J)
\]
by $\Theta(\overline{a_0} \otimes \overline{a_1} \otimes \overline{a_2} \otimes \overline{a_3}) = \overline{a_0 \otimes a_1 \otimes a_2 \otimes a_3}$ for $a_i \in W_i$.

To show that $\Theta$ is well-defined, it suffices to prove that
\begin{align}\label{genus01}
\begin{split}
& C_2(W_0) \otimes W_1 \otimes W_2 \otimes W_3 + W_0 \otimes C_2(W_1) \otimes W_2 \otimes W_3 \\
& + W_0 \otimes W_1 \otimes C_2(W_2) \otimes W_3 + W_0 \otimes W_1 \otimes W_2 \otimes C_2(W_3)
\end{split}
\end{align}
and
\begin{align}\label{genus02}
\begin{split}
& \left\{ u_{(0)} w_0 \otimes w_1 \otimes w_2 \otimes w_3 + w_0 \otimes u_{(0)} w_1 \otimes w_2 \otimes w_3 + w_0 \otimes w_1 \otimes u_{(0)} w_2 \otimes w_3 \right. \\
& \left. + w_0 \otimes w_1 \otimes w_2 \otimes u_{(0)} w_3 \,\middle|\, u \in V_{(1)},\, w_i \in W_i \right\}
\end{split}
\end{align}
belong to $\mathrm{gr}^I(\J)$.

We have already established \eqref{genus01}. To prove \eqref{genus02}, recall that if $V$ is of CFT type and self-dual, then all weight one vectors in $V$ are primary (see \cite{li1994symmetric}). In particular, $u_{(0)}^{*} = -u_{(0)}$ for $u \in V_{(1)}$. Since the elements $\mathcal{E}(u, w_0, w_1, w_2, w_3)$ belong to $\J$, we conclude that \eqref{genus02} belongs to $\mathrm{gr}^{I}(\mathcal{J})$.

It is clear that $\Theta$ is surjective. The finite-dimensionality of the quotient in \eqref{newfinitenss} then implies that $\T / \mathrm{gr}^I(\J)$ is finitely generated as an $R$-module.
\end{proof}

We set $\sigma := \operatorname{wt} w_0 + \operatorname{wt} w_1 + \operatorname{wt} w_2 + \operatorname{wt} w_3$. For $n \in \mathbb{Z}_{+}$, let $F_n^{z_1 = z_2}(R)$ be the vector space spanned by elements of the form $f(z_1, z_2)(z_1 - z_2)^{-n}$ for $f(z_1, z_2) \in \mathbb{C}[z_1^{\pm}, z_2^{\pm}]$. For $r \in \mathbb{Q}$, let $F_r^{(z_1 = z_2)}(\T)$ be the subspace of $\T$ spanned by elements of the form $f(z_1, z_2)(z_1 - z_2)^{-n} w_0 \otimes w_1 \otimes w_2 \otimes w_3$, where $f(z_1, z_2) \in \mathbb{C}[z_1^{\pm}, z_2^{\pm}]$ and $n + \sigma \leq r$. Define $F_r^{(z_1 = z_2)}(\J) := F_r^{(z_1 = z_2)}(\T) \cap \J$.

\begin{proposition}\label{decom}
Suppose $\T/\J$ is finitely generated as an $R$-module. Then there exists $M \in \mathbb{Z}_{+}$ such that for all $r \in \mathbb{N}$,
\[
F_r^{(z_1 = z_2)}(\T) \subset F_r^{(z_1 = z_2)}(\J) + I_M(\T).
\]
\end{proposition}

\begin{proof}
Since $\T/\J$ is finitely generated, there exists $M > 0$ such that $\T \setminus I_M(\T) \subset \J$. If $r \leq M$, then $F_r^{(z_1 = z_2)}(\T) \subset I_M(\T)$. Now assume $r > M$ and consider
\[
(z_1 - z_2)^{-i} g(z_1, z_2) w_0 \otimes w_1 \otimes w_2 \otimes w_3 \in F_r^{(z_1 = z_2)}(T),
\]
where each $w_i$ is homogeneous and $g(z_1, z_2) \in \mathbb{C}[z_1^{\pm}, z_2^{\pm}]$. If $r - i \leq M$, then the element lies in $I_M(\T)$. Otherwise, $w_0 \otimes w_1 \otimes w_2 \otimes w_3 \in I_{r-i}(\J)$, so the whole tensor lies in $F_r^{(z_1 = z_2)}(\J)$.
\end{proof}

Let
\[
\T^{(z_1 - z_2)} := \mathbb{C}[z_1^{\pm}, z_2^{\pm}] \otimes W_0 \otimes W_1 \otimes W_2 \otimes W_3,
\]
and for $r \in \mathbb{Q}$, let $\T_{(r)}^{(z_1 - z_2)}$ denote the subspace of weight $r$. By Proposition \ref{decom}, for $w_i \in W_i$, we may write
\[
w_0 \otimes w_1 \otimes w_2 \otimes w_3 = \mathcal{W}_1 + \mathcal{W}_2
\]
where $\mathcal{W}_1 \in F_r^{(z_1 = z_2)}(\J)$ and $\mathcal{W}_2 \in I_M(\T)$. Then by the argument in \cite[Lemma 2.2]{huang2005differential1}, we obtain the following.

\begin{lemma}\label{lem9.3}
For any $s \in [0,1)$, there exists $S \in \mathbb{R}$ such that $s + S \in \mathbb{Z}_{+}$ and for all $w_i \in W_i$ with $\sigma \in s + \mathbb{Z}$, we have
\[
(z_1 - z_2)^{\sigma + S} \mathcal{W}_2 \in \T^{(z_1 = z_2)}.
\]
\end{lemma}

\begin{theorem}\label{differentialintertwining}
Let $V$ be as in Lemma \ref{quasilissecondition}, and suppose that $W_i$ for $i=0,1,2,3$ satisfy the finiteness condition \eqref{newfinitenss}. Then for any $w_i \in W_i$ $(i = 0,1,2,3)$, there exist
\[
a_k(z_1, z_2),\ b_l(z_1, z_2) \in \mathbb{C}[z_1^{\pm}, z_2^{\pm}, (z_1 - z_2)^{-1}]
\]
for $k = 1, \dots, m$ and $l = 1, \dots, n$ such that for any discretely $\mathbb{Q}$-graded $V$-modules $W_4$, $W_5$, and $W_6$, and any intertwining operators $\mathcal{Y}_1$, $\mathcal{Y}_2$, $\mathcal{Y}_3$, $\mathcal{Y}_4$, $\mathcal{Y}_5$, $\mathcal{Y}_6$ of types
\[
\binom{W_0'}{W_1\ W_4},\quad \binom{W_4}{W_2\ W_3},\quad \binom{W_5}{W_1\ W_2},\quad \binom{W_0'}{W_5\ W_3},\quad \binom{W_0'}{W_2\ W_6},\quad \binom{W_6}{W_1\ W_3},
\]
respectively, the series
\begin{align}
&\langle w_0, \mathcal{Y}_1(w_1, z_1)\mathcal{Y}_2(w_2, z_2)w_3 \rangle, \label{equn1} \\
&\langle w_0, \mathcal{Y}_4(\mathcal{Y}_3(w_1, z_1 - z_2)w_2, z_2)w_3 \rangle, \label{equn2}
\end{align}
and
\begin{align}
\langle w_0, \mathcal{Y}_5(w_2, z_2)\mathcal{Y}_6(w_1, z_1)w_3 \rangle \label{equn3}
\end{align}
satisfy the expansions of a system of differential equations with regular singular points:
\begin{align}\label{equn4}
\frac{\partial^m f}{\partial z_1^m} + a_1(z_1, z_2)\frac{\partial^{m-1} f}{\partial z_1^{m-1}} + \cdots + a_m(z_1, z_2)f = 0,
\end{align}
\begin{align}\label{equn5}
\frac{\partial^n f}{\partial z_2^n} + b_1(z_1, z_2)\frac{\partial^{n-1} f}{\partial z_2^{n-1}} + \cdots + b_n(z_1, z_2)f = 0
\end{align}
in the regions $|z_1| > |z_2| > 0$, $|z_2| > |z_1 - z_2| > 0$, and $|z_2| > |z_1| > 0$, respectively.
\end{theorem}

\begin{proof}
Let $\Delta = \operatorname{wt} w_0 - \operatorname{wt} w_1 - \operatorname{wt} w_2 - \operatorname{wt} w_3$. Consider the maps
\begin{align*}
&\Upsilon_{\mathcal{Y}_1,\mathcal{Y}_2}: \T \rightarrow z_1^{\Delta} \mathbb{C}(\{z_2 / z_1\})[z_1^{\pm}, z_2^{\pm}], \\
&\Upsilon_{\mathcal{Y}_3,\mathcal{Y}_4}: \T \rightarrow z_2^{\Delta} \mathbb{C}(\{(z_1 - z_2)/z_2\})[z_2^{\pm}, (z_1 - z_2)^{\pm}], \\
&\Upsilon_{\mathcal{Y}_5,\mathcal{Y}_6}: \T \rightarrow z_2^{\Delta} \mathbb{C}(\{z_1 / z_2\})[z_1^{\pm}, z_2^{\pm}],
\end{align*}
defined by
\begin{align*}
&\Upsilon_{\mathcal{Y}_1,\mathcal{Y}_2}(f(z_1, z_2) w_0 \otimes w_1 \otimes w_2 \otimes w_3) \\
&\quad := \iota_{|z_1| > |z_2| > 0}(f(z_1, z_2)) \langle w_0, \mathcal{Y}_1(w_1, z_1)\mathcal{Y}_2(w_2, z_2)w_3 \rangle, \\
&\Upsilon_{\mathcal{Y}_3,\mathcal{Y}_4}(f(z_1, z_2) w_0 \otimes w_1 \otimes w_2 \otimes w_3) \\
&\quad := \iota_{|z_2| > |z_1 - z_2| > 0}(f(z_1, z_2)) \langle w_0, \mathcal{Y}_4(\mathcal{Y}_3(w_1, z_1 - z_2)w_2, z_2)w_3 \rangle, \\
&\Upsilon_{\mathcal{Y}_5,\mathcal{Y}_6}(f(z_1, z_2) w_0 \otimes w_1 \otimes w_2 \otimes w_3) \\
&\quad := \iota_{|z_2| > |z_1| > 0}(f(z_1, z_2)) \langle w_0, \mathcal{Y}_5(w_2, z_2)\mathcal{Y}_6(w_1, z_1)w_3 \rangle,
\end{align*}
where the $\iota$ maps denote expansion in the respective regions:
\begin{align*}
&\iota_{|z_1| > |z_2| > 0}: R \to \mathbb{C}[[z_2 / z_1]][z_1^{\pm}, z_2^{\pm}], \\
&\iota_{|z_2| > |z_1 - z_2| > 0}: R \to \mathbb{C}[[(z_1 - z_2)/z_2]][z_2^{\pm}, (z_1 - z_2)^{\pm}], \\
&\iota_{|z_2| > |z_1| > 0}: R \to \mathbb{C}[[z_1 / z_2]][z_1^{\pm}, z_2^{\pm}].
\end{align*}

By \eqref{relatione}, we know that $\mathcal{E}(u, w_0, w_1, w_2, w_3)$ lies in the intersection of the kernels of $\Upsilon_{\mathcal{Y}_1,\mathcal{Y}_2}$, $\Upsilon_{\mathcal{Y}_3,\mathcal{Y}_4}$, and $\Upsilon_{\mathcal{Y}_5,\mathcal{Y}_6}$. As noted in \cite{huang2005differential1}, the same holds for $\mathcal{A}, \mathcal{B}, \mathcal{C}, \mathcal{D}$ as well. Thus, we obtain induced maps:
\begin{align*}
&\overline{\Upsilon_{\mathcal{Y}_1,\mathcal{Y}_2}}: T/J \rightarrow z_1^{\Delta} \mathbb{C}(\{z_2 / z_1\})[z_1^{\pm}, z_2^{\pm}], \\
&\overline{\Upsilon_{\mathcal{Y}_3,\mathcal{Y}_4}}: T/J \rightarrow z_2^{\Delta} \mathbb{C}(\{(z_1 - z_2)/z_2\})[z_2^{\pm}, (z_1 - z_2)^{\pm}], \\
&\overline{\Upsilon_{\mathcal{Y}_5,\mathcal{Y}_6}}: T/J \rightarrow z_2^{\Delta} \mathbb{C}(\{z_1 / z_2\})[z_1^{\pm}, z_2^{\pm}].
\end{align*}

Since $\T/\J$ is finitely generated over $R$, and using the $L(-1)$-derivative property of intertwining operators, it follows from the argument in \cite[Theorem 1.4]{huang2005differential1} that the series \eqref{equn1}, \eqref{equn2}, \eqref{equn3} satisfy the differential equations \eqref{equn4} and \eqref{equn5}.

Finally, by Proposition \ref{decom}, Lemma \ref{lem9.3}, and the arguments in \cite[Theorem 2.3]{huang2005differential1}, the singularities of these differential equations are regular.
\end{proof}

\begin{proposition}
Let $W_i$, $i = 0, 1, 2, 3$, be finitely strongly generated $V$-modules. Then
\[
\dim \left( R_{W_0} \otimes_{R_V} \cdots \otimes_{R_V} R_{W_3} \Big/ \{ R_V,\ R_{W_0} \otimes_{R_V} \cdots \otimes_{R_V} R_{W_3} \} \right) < \infty.
\]
\end{proposition}

\begin{proof}
Since each $W_i$ is a finitely strongly generated $V$-module, $R_{W_i}$ is finitely generated over $R_V$. Hence the tensor product $R_{W_0} \otimes_{R_V} \cdots \otimes_{R_V} R_{W_3}$ is finitely generated over $R_V$. The result then follows from Corollary~\ref{leaves'}.
\end{proof}

\begin{remark}
The converse of this proposition is not true in general. See the counterexample in the last section.
\end{remark}

\begin{remark}
The finitely strongly generated condition for a $V$-module $M$ is weaker than Huang's $C_1$-cofiniteness condition. Indeed, by \cite{arakawa2012remark}, for $V$ of CFT type, $M$ is strongly finitely generated over $V$ if and only if it satisfies Li's $C_1$-cofiniteness condition. Since Li's $C_1(M)$ is defined as
\[
\mathrm{span}_{\mathbb{C}} \{ \omega_{(0)}m,\ a_{(-1)}m \mid a \in \bigoplus_{\Delta > 0} V_\Delta,\ m \in M \},
\]
and contains Huang's $C_1(M)$, which is
\[
\mathrm{span}_{\mathbb{C}} \{ a_{(-1)}m \mid a \in \bigoplus_{\Delta > 0} V_\Delta,\ m \in M \},
\]
Li's $C_1$-cofiniteness condition is therefore strictly weaker.
\end{remark}

\section{Genus-one correlation functions}

Let $V$ be a vertex superalgebra, $W_1$ be its weak $V$-module, and $W_2$ and $W_3$ be its weak $g$-twisted modules, where $g$ is of finite order $T$. Let $u\in V^{r}$ and $w\in W_1$ be $\mathbb{Z}_2$-homogeneous elements. We first study some basic properties of geometrically-modified twisted intertwining operators, which are discussed in Appendix A.  Let 
\begin{align*}
  \frac{1}{2\pi i}(e^{2\pi i y}-1)=\left(\exp\left(-\sum_{j\in\mathbb{Z}_{+}}A_{j}y^{j+1}\frac{\partial}{\partial y}\right)\right)y.
\end{align*}
Denote the operator $\sum_{j\in \Z_{+}}L_{(j)}$ by $L_+(A)$.  
Let $\mathcal{U}(x)=(2\pi i x)^{L_{0}}e^{-L^+(A)}$, where $x$ is any number or a formal variable that makes sense.
\begin{proposition}
Let $\mathcal{Y}$ be an intertwining operator of type $\binom{W_3}{W_1\;W_2}$. Then
\begin{align}\label{coorcha}
\mathcal{U}(1)\mathcal{Y}(w,x)(\mathcal{U}(1))^{-1}=\mathcal{Y}\left(\mathcal{U}(e^{2\pi i x})w,\, e^{2\pi i x}-1\right)
\end{align}
and
\begin{align}\label{coorcha1}
\begin{split}
\left[Y\left(\mathcal{U}(x_{1})u,\, x_{1}\right),\, 
\mathcal{Y}\left(\mathcal{U}(x_{2})w,\, x_{2}\right)\right]
=\, & 2\pi i\, \operatorname{Res}_{y}\, 
\delta\left(\frac{x_{1}}{e^{2\pi i y}x_{2}}\right)
\left(\frac{x_{1}}{e^{2\pi i y}x_{2}}\right)^{r/T} \\
& \cdot \mathcal{Y}\left(\mathcal{U}(x_{2})Y(u,y)w,\, x_{2}\right).
\end{split}
\end{align}
\end{proposition}
\begin{proof}
The first identity follows immediately from the same argument in Example \ref{ctft}.
We prove the second identity. Recall the twisted commutator formula
\begin{align*}
    \left[Y(u,x_{1}),\, \mathcal{Y}(w,x_{2})\right]
    =\text{Res}_{x_{0}}\, x_{2}^{-1}\delta\left(\frac{x_{1}-x_{0}}{x_{2}}\right)\left(\frac{x_{1}-x_{0}}{x_{2}}\right)^{-r/T}\mathcal{Y}(Y(u,x_{0})w,x_{2}).
\end{align*}
Thus
\begin{align}\label{commm}
\begin{split}
    \left[Y\left(x_{1}^{L_{(0)}}u,x_{1}\right),\, \mathcal{Y}\left(x_{2}^{L_{(0)}}w,x_{2}\right)\right]
    &=\text{Res}_{x_{0}}\, x_{2}^{-1}\delta\left(\frac{x_{1}-x_{0}}{x_{2}}\right)\left(\frac{x_{1}-x_{0}}{x_{2}}\right)^{-r/T} \\
    &\quad\cdot \mathcal{Y}\left(Y\left(\left(\frac{x_{1}}{x_{2}}\right)^{L_{(0)}}u,\, \frac{x_{0}}{x_{2}}\right)w,\, x_{2}\right),
\end{split}
\end{align}
where we used the $L_{(0)}$-conjugation formula
\begin{align}\label{l0conjugatin}
q^{-L_{(0)}}Y(q^{L_{(0)}}u,x)q^{L_{(0)}}=Y\left(u,\, \frac{x}{q}\right).
\end{align}
Now we change the variable $x_{0}$ to $y$ as follows:
\begin{align*}
    x_{0}=\sum_{k\in \mathbb{N}} \frac{x_{2}(2\pi i y)^{k}}{k!}=(e^{2\pi i y}-1)x_{2}.
\end{align*}
Then RHS of (\ref{commm}) becomes
\begin{align}\label{gotcf1}
\begin{split}
       & \text{Res}_{y}\, \frac{\partial x_{0}}{\partial y} x_{2}^{-1} \delta\left(\frac{x_{1}-(e^{2\pi i y}-1)x_{2}}{x_{2}}\right)\left(\frac{x_{1}-(e^{2\pi i y}-1)x_{2}}{x_{2}}\right)^{-r/T} \\
       &\quad\cdot \mathcal{Y}\left(x_{2}^{L_{(0)}}Y\left(\left(\frac{x_{1}}{x_{2}}\right)^{L_{(0)}}u,\, e^{2\pi i y}-1\right)w,\, x_{2}\right) \\
       &= 2\pi i\, \text{Res}_{y}\, \delta\left(\frac{x_{1}}{e^{2\pi i y}x_{2}}\right)\left(\frac{x_{1}}{e^{2\pi i y}x_{2}}\right)^{-r/T}
       \mathcal{Y}\left(x_{2}^{L_{(0)}}Y\left(e^{2\pi i y L_{(0)}}u,\, e^{2\pi i y}-1\right)w,\, x_{2}\right) \\
       &= 2\pi i\, \text{Res}_{y}\, \delta\left(\frac{x_{1}}{e^{2\pi i y}x_{2}}\right)\left(\frac{x_{1}}{e^{2\pi i y}x_{2}}\right)^{-r/T} \\
       &\quad\cdot \mathcal{Y}\left(x_{2}^{L_{(0)}}\mathcal{U}(1)Y\left(\mathcal{U}(1)^{-1}u,\, y\right)\mathcal{U}(1)^{-1}w,\, x_{2}\right),
\end{split}
\end{align}
where the two identities ($\alpha\in \mathbb{C}$)
\begin{align*}
    & x_{0}^{-1}\delta\left(\frac{x_{1}-x_{2}}{x_{0}}\right)\left(\frac{x_{1}-x_{2}}{x_{0}}\right)^{\alpha}=x_{1}^{-1}\delta\left(\frac{x_{0}+x_{2}}{x_{1}}\right)\left(\frac{x_{0}+x_{2}}{x_{1}}\right)^{-\alpha}, \\
    & x^{-1}\delta\left(\frac{y}{x}\right)\left(\frac{y}{x}\right)^{\alpha}=y^{-1}\delta\left(\frac{x}{y}\right)\left(\frac{x}{y}\right)^{-\alpha}
\end{align*}
were used. The last equality follows from (\ref{coorcha}).

Equations (\ref{commm}) and (\ref{gotcf1}) together give
\begin{align}\label{gotcf2}
\begin{split}
\left[
Y\left(x_{1}^{L_{(0)}} u,\, x_{1} \right),\,
\mathcal{Y}\left(x_{2}^{L_{(0)}} w,\, x_{2} \right)
\right]
=\, & 2\pi i\, \operatorname{Res}_{y} \,
\delta\left( \frac{x_{1}}{e^{2\pi i y}x_{2}} \right)
\left( \frac{x_{1}}{e^{2\pi i y}x_{2}} \right)^{-r/T} \\
& \cdot \mathcal{Y}\left( 
\mathcal{U}(x_{2}) Y\left( \mathcal{U}(1)^{-1} u,\, y \right) 
\mathcal{U}(1)^{-1} w,\, x_{2} \right).
\end{split}
\end{align}
Equation (\ref{coorcha1}) follows from (\ref{gotcf2}) by substituting $\mathcal{U}(1)u$ for $u$ and $\mathcal{U}(1)w$ for $w$ and using the identity
\begin{align*}
    \mathcal{U}(xy)=x^{L_{(0)}}\mathcal{U}(y).
\end{align*}
\end{proof}

\begin{proposition}
For any $u \in V^{r}$,
\begin{align}\label{gotcf123}
    2\pi i x \frac{d}{dx} \mathcal{Y}(\mathcal{U}(x)u, x) = \mathcal{Y}(\mathcal{U}(x)L_{(-1)}u, x).
\end{align}
\end{proposition}

\begin{proof}
Using the $L_{(-1)}$-derivative property for $\mathcal{Y}$, we obtain
\begin{align}\label{newprop1}
\begin{split}
    2\pi i x \frac{d}{dx} \mathcal{Y}(\mathcal{U}(x)u, x)
    &= 2\pi i x \frac{d}{dx} \mathcal{Y}\big((2\pi i x)^{L_{(0)}} e^{-L_{+}(A)} u, x\big) \\
    &= \mathcal{Y}\big((2\pi i L_{(0)} + 2\pi i x L_{(-1)})(2\pi i x)^{L_{(0)}} e^{-L_{+}(A)} u, x\big) \\
    &= \mathcal{Y}\big(x^{L_{(0)}}(2\pi i L_{(0)} + 2\pi i x L_{(-1)}) \mathcal{U}(1) u, x\big).
\end{split}
\end{align}

Let $\omega$ be the conformal vector in $V$. Then
\begin{align}\label{newprop2}
\begin{split}
    \mathcal{U}(1)L_{(-1)} &= \text{Res}_{x} \, \mathcal{U}(1) Y(\omega, x) \\
    &= \text{Res}_{x} \, Y\big(\mathcal{U}(e^{2\pi i x}) \omega, e^{2\pi i x} - 1\big) \mathcal{U}(1) \\
    &= \text{Res}_{x} \, Y\big(e^{2\pi i x L_{(0)}} \omega, e^{2\pi i x} - 1\big) \mathcal{U}(1) \\
    &= \text{Res}_{x} \, (2\pi i)^2 Y\big(e^{2\pi i x L_{(0)}} (\omega - \tfrac{c}{24} \mathbf{1}), e^{2\pi i x} - 1\big) \mathcal{U}(1) \\
    &= \text{Res}_{x} \, (2\pi i)^2 e^{4\pi i x} Y(\omega, e^{2\pi i x} - 1) \mathcal{U}(1) \\
    &=\text{Res}_{y} \, (2\pi i + 2\pi i y) Y(\omega, y) \mathcal{U}(1) \\
    &= (2\pi i L_{(-1)} + 2\pi i L_{(0)}) \mathcal{U}(1),
\end{split}
\end{align}
where the second identity follows from (\ref{coorcha}), the fourth identity uses the identity 
\[
\mathcal{U}(1) \omega = (2\pi i)^2 \left(\omega - \tfrac{c}{24} \mathbf{1} \right),
\]
and the sixth identity follows by changing variables from $x$ to $y = e^{2\pi i x} - 1$.

Using (\ref{newprop2}), the right-hand side of (\ref{newprop1}) equals
\[
\mathcal{Y}(\mathcal{U}(x)L_{(-1)}u, x),
\]
proving the proposition.
\end{proof}

Let $W_1, W_2, W_5$ be finitely strongly generated $V$-modules, and $W_0, W_3, W_4$ be ordinary $g$-twisted modules. Let $W_0'$ be the contragredient module of $W_0$, which is a $g^{-1}$-twisted $V$-module. Let $\mathcal{Y}_1, \mathcal{Y}_2, \mathcal{Y}_3, \mathcal{Y}_4$ be intertwining operators of types
\[
\binom{W_0}{W_1 \; W_4},\quad \binom{W_4}{W_2 \; W_3},\quad \binom{W_5}{W_1 \; W_2},\quad \binom{W_0}{W_5 \; W_3},
\]
respectively.

If Conjecture \ref{duality} holds, then for any $w_0' \in W_0'$, $w_i \in W_i$ for $i = 1, 2, 3$, and when \(|z_1| > |z_2| > |z_1 - z_2| > 0\), we have the associativity
\begin{align}\label{newprop3}
    \langle w_0', \mathcal{Y}_1(w_1, z_1) \mathcal{Y}_2(w_2, z_2) w_3 \rangle 
    = \langle w_0', \mathcal{Y}_4(\mathcal{Y}_3(w_1, z_1 - z_2) w_2, z_2) w_3 \rangle.
\end{align}
In particular, (\ref{newprop3}) holds when \(g\) is the identity and \(W_0, W_1, W_2, W_3\) satisfy the finiteness condition (\ref{newfinitenss}).

Assume the associativity property holds. Let \(q_z := e^{2\pi i z}\). We also have the following.

\begin{proposition}
When \(|q_{z_1 - z_2}| > 1 > |q_{z_1 - z_2} - 1| > 0\), we have
\begin{align*}
    &\langle w_0', \mathcal{Y}_1(\mathcal{U}(q_{z_1}) w_1, q_{z_2}) \mathcal{Y}_2(\mathcal{U}(q_{z_2}) w_2, q_{z_2}) w_3 \rangle \\
    &= \langle w_0', \mathcal{Y}_4(\mathcal{U}(q_{z_2}) \mathcal{Y}_3(w_1, z_1 - z_2) w_2, q_{z_2}) w_3 \rangle.
\end{align*}
\end{proposition}

\begin{proof}
By (\ref{newprop2}), (\ref{l0conjugatin}), and (\ref{coorcha}),
\begin{align*}
    &\langle w_0', \mathcal{Y}_1(\mathcal{U}(q_{z_1}) w_1, q_{z_2}) \mathcal{Y}_2(\mathcal{U}(q_{z_2}) w_2, q_{z_2}) w_3 \rangle \\
    &= \langle w_0', \mathcal{Y}_4(\mathcal{Y}_3(\mathcal{U}(q_{z_1}) w_1, q_{z_1} - q_{z_2}) \mathcal{U}(q_{z_2}) w_2, q_{z_2}) w_3 \rangle \\
    &= \langle w_0', \mathcal{Y}_4(\mathcal{Y}_3(\mathcal{U}(q_{z_1}) w_1, q_{z_1} - q_{z_2}) q_{z_2}^{L_{(0)}} \mathcal{U}(1) w_2, q_{z_2}) w_3 \rangle \\
    &= \langle w_0', \mathcal{Y}_4(q_{z_2}^{L_{(0)}} \mathcal{Y}_3(\mathcal{U}(q_{z_1 - z_2}) w_1, q_{z_1 - z_2} - 1) \mathcal{U}(1) w_2, q_{z_2}) w_3 \rangle \\
    &= \langle w_0', \mathcal{Y}_4(q_{z_2}^{L_{(0)}} \mathcal{U}(1) \mathcal{Y}_3(\mathcal{U}(w_1, z_1 - z_2) w_2, q_{z_2}) w_3 \rangle \\
    &= \langle w_0', \mathcal{Y}_4(\mathcal{U}(q_{z_2}) \mathcal{Y}_3(\mathcal{U}(w_1, z_1 - z_2) w_2, q_{z_2}) w_3 \rangle.
\end{align*}
The result follows since \( |q_{z_1}| > |q_{z_2}| > |q_{z_1} - q_{z_2}| > 0 \), or equivalently \( |q_{z_1 - z_2}| > 1 > |q_{z_1 - z_2} - 1| > 0 \).
\end{proof}
\color{black}
\subsection{Identities for formal \texorpdfstring{$q$}{Lg}-traces}
In the following, we adopt the following conventions for notation and expansions. For any function \( f \) defined on the upper half-plane, we denote by \( \tilde{f}(q) \) its \( q \)-expansion, where \( q = e^{2\pi i \tau} \). More generally, for any variable \( \bullet \in \mathbb{C} \), we write
\[
q_\bullet := e^{2\pi i \bullet}.
\]

Let \( (\theta, \phi) \) be a pair of complex numbers in the unit disk, with \( \phi = e^{2\pi i r/T} \) for some rational number \( r/T \). We define the twisted \( P \)-functions by
\begin{align}\label{pqfunction3}
  P_{1}^{\theta,\phi}(z,\tau) := 
  \begin{cases}
    2\pi i \displaystyle\sum_{n \in \mathbb{Z},\, n \neq 0} \frac{q_z^n}{1 - \theta^{-1} q_\tau^n}, & \phi = 1, \\[2mm]
    2\pi i \displaystyle\sum_{n \in r/T + \mathbb{Z}} \frac{q_z^n}{1 - \theta^{-1} q_\tau^n}, & \phi \neq 1,
  \end{cases}
\end{align}
which converge uniformly and absolutely on compact subsets of the region \( |q_\tau| < |q_z| < 1 \). For \( m \geq 2 \), define
\begin{align}\label{pqfunction7}
  P_{m}^{\theta,\phi}(z,\tau) := \frac{1}{(m-1)!} \frac{\partial^{m-1}}{\partial z^{m-1}} P_{1}^{\theta,\phi}(z,\tau).
\end{align}

In our context, we will often treat these as series in the variable \( q \), regarded as \( q = e^{2\pi i \tau} \). When \( \phi \neq 1 \), the twisted \( P \)-function admits the following \( q \)-expansion over the ring \( \mathbb{C}((q^{1/T})) \):
\begin{align}
\begin{split}
\tilde{P}_m^{\theta, \phi}(x, q) = \frac{1}{m!} \Bigg( 
&\sum_{n \geq 0} \frac{(n + r/T)^m x^{n + r/T}}{1 - \theta^{-1} q^{n + r/T}} \\
&- \sum_{n > 0} \frac{(-1)^m \theta (n - r/T)^m x^{-(n - r/T)} q^{n - r/T}}{1 - \theta q^{n - r/T}} 
\Bigg),
\end{split}
\end{align}
and when \( \phi = 1 \), we have
\begin{align}\label{pqfunction8}
\tilde{P}_m^{\theta, 1}(x, q) = (2\pi i)^m \frac{1}{m!} \left( 
\sum_{n > 0} \frac{n^m x^n}{1 - \theta^{-1} q^n} 
- \sum_{n > 0} \frac{(-1)^m \theta n^m x^{-n} q^n}{1 - \theta q^n} 
\right).
\end{align}
The relation between the twisted $P$-type functions and the twisted Weierstrass functions \( P_k\genfrac[]{0pt}{0}{\theta}{\phi}(z,\tau) \) is reviewed in Appendix~C.

In this subsection, we derive some important identities of genus-one correlation functions. These identities involve the twisted modular forms introduced in Appendices~B and~C.

Let $W_i$, $i = 1, 2, \ldots, n$, be weak $V$-modules, and let $\tilde{W}_i$, $i = 1, 2, \ldots, n$, be weak $g$-twisted $V$-modules. Let $h$ be a Cartan element of $V$. Assume that $\tilde{W}_0 = \tilde{W}_n$ is $h$-stable. We will use this assumption for the rest of this section.

\begin{proposition}\label{iffq1}
Let $g$ be of order $T$, and suppose $h_{(0)}(u) = \lambda^{-1} u$ for a $\mathbb{Z}_{2}$-homogeneous element $u \in V^{r}$. For any $\mathbb{Z}_{2}$-homogeneous elements $u_i \in W_i$, $i = 1, \ldots, n$, denote $e^{2\pi i r/T}$ by $\phi$ and $|u|(|u_1| + \cdots + |u_n|)$ by $t$. Let $\mathcal{Y}_i$, $i = 1, \ldots, n$, be intertwining operators of type $\binom{\tilde{W}_{i-1}}{W_i\; \tilde{W}_{i}}$, respectively. Let $\theta = e^{2\pi i s \lambda^{-1}}$.

Then we have
\begin{align}\label{iffq2}
\begin{split}
&\operatorname{tr}_{\tilde{W}_n}Y(\mathcal{U}(x)u,x)\mathcal{Y}_1(\mathcal{U}(x_1)u_1,x_1) \cdots \mathcal{Y}_n(\mathcal{U}(x_n)u_n,x_n) e^{2\pi i s h_{(0)}} q^{L_{(0)}} \\
&\quad= \sum_{i=1}^{n} \sum_{m \in \mathbb{N}} \tilde{P}_{m+1}^{(-1)^{t} \theta, \phi}\left( \frac{x_i}{x}, q \right) \operatorname{tr}_{\tilde{W}_n} \mathcal{Y}_1(\mathcal{U}(x_1)u_1,x_1) \cdots \\
&\quad\quad\quad\quad\quad \cdots \mathcal{Y}_{i-1}(\mathcal{U}(x_{i-1})u_{i-1},x_{i-1}) \mathcal{Y}_i(\mathcal{U}(x_i) u_{(m)} u_i, x_i) \cdots \\
&\quad\quad\quad\quad\quad \cdots \mathcal{Y}_n(\mathcal{U}(x_n)u_n,x_n) e^{2\pi i s h_{(0)}} q^{L_{(0)}} \\
&\quad\quad + \delta_{(-1)^t \theta^{-1}, 1} \operatorname{tr}_{\tilde{W}_n} o(\mathcal{U}(1) u) \mathcal{Y}_1(\mathcal{U}(x_1)u_1,x_1) \cdots \mathcal{Y}_n(\mathcal{U}(x_n)u_n,x_n) e^{2\pi i s h_{(0)}} q^{L_{(0)}}.
\end{split}
\end{align}

Moreover, when $u \in V^0$, we also have
\begin{align}\label{iffq8}
\begin{split}
\sum_{i=1}^{n} & \operatorname{tr}_{\tilde{W}_n} \mathcal{Y}_1(\mathcal{U}(x_1)u_1,x_1) \cdots \mathcal{Y}_{i-1}(\mathcal{U}(x_{i-1})u_{i-1},x_{i-1}) \\
&\quad\quad \cdot \mathcal{Y}_i(\mathcal{U}(x_i) u_{(0)} u_i, x_i) \cdots \mathcal{Y}_n(\mathcal{U}(x_n)u_n,x_n) e^{2\pi i s h_{(0)}} q^{L_{(0)}} = 0.
\end{split}
\end{align}
\end{proposition}
\begin{proof}
The LHS of (\ref{iffq2}) equals
\begin{align}\label{iffq3}
\begin{split}
&\sum_{i=1}^{n} \operatorname{tr}_{\tilde{W}_n}
\Y_1(\mathcal{U}(x_{1})u_{1},x_{1}) \cdots \Y_{i-1}(\mathcal{U}(x_{i-1})u_{i-1},x_{i-1}) \\
&\qquad \cdot \big[Y(\mathcal{U}(x)u,x),\Y_i(\mathcal{U}(x_{i})u_{i},x_{i})\big]
\Y_{i+1}(\mathcal{U}(x_{i+1})u_{i+1},x_{i+1}) \cdots \\
&\qquad \cdot \Y_n(\mathcal{U}(x_{n})u_{n},x_{n}) e^{2\pi i s h_{(0)}} q^{L_{(0)}} \\
&\quad + (-1)^{t}
\Y_1(\mathcal{U}(x_{1})u_{1},x_{1}) \cdots \Y_n(\mathcal{U}(x_{n})u_{n},x_{n}) \\
&\qquad \cdot Y(\mathcal{U}(x)u,x) e^{2\pi i s h_{(0)}} q^{L_{(0)}} \\
&=\sum_{i=1}^{n} 2\pi i \operatorname{Res}_{y}
\delta\left(\frac{x}{e^{2\pi i y} x_i}\right)
\left(\frac{x}{e^{2\pi i y} x_i}\right)^{-r/T} \cdot \\
&\qquad \cdot \operatorname{tr}_{\tilde{W}_n}
\Y_1(\mathcal{U}(x_{1})u_{1},x_{1}) \cdots
\Y_{i-1}(\mathcal{U}(x_{i-1})u_{i-1},x_{i-1}) \\
&\qquad \cdot Y_i\left(\mathcal{U}(x_i) Y(u,y) u_i, x_i \right)
\Y_{i+1}(\mathcal{U}(x_{i+1})u_{i+1},x_{i+1}) \cdots \\
&\qquad \cdot \Y_n(\mathcal{U}(x_{n})u_{n},x_{n})
e^{2\pi i s h_{(0)}} q^{L_{(0)}} \\
&\quad + (-1)^{t}
\operatorname{tr}_{\tilde{W}_n} \theta^{-1} Y\left(\mathcal{U}\left(\frac{x}{q}\right) u, \frac{x}{q} \right)
\mathcal{Y}_1(\mathcal{U}(x_1 w_1),x_1) \cdots \\
&\qquad \cdot \mathcal{Y}_n(\mathcal{U}(x_n w_n),x_n)
e^{2\pi i s h_{(0)}} q^{L_{(0)}} \\
&=\sum_{i=1}^{n} \sum_{m \geq 0}
\operatorname{Res}_{y} \frac{(2\pi i)^{m+1} y^m}{m!}
\left( \left( x_i \frac{\partial}{\partial x_i} \right)^m
\delta\left( \frac{x}{x_i} \right)
\left( \frac{x}{x_i} \right)^{-r/T} \right) \cdot \\
&\qquad \cdot \operatorname{tr}_{\tilde{W}_n}
\Y_1(\mathcal{U}(x_{1})u_{1},x_{1}) \cdots
\Y_{i-1}(\mathcal{U}(x_{i-1})u_{i-1},x_{i-1}) \\
&\qquad \cdot Y_i(\mathcal{U}(x_i) Y(u,y) u_i, x_i)
\Y_{i+1}(\mathcal{U}(x_{i+1})u_{i+1},x_{i+1}) \cdots \\
&\qquad \cdot \Y_n(\mathcal{U}(x_{n})u_{n},x_{n})
e^{2\pi i s h_{(0)}} q^{L_{(0)}} \\
&\quad + (-1)^{t} \theta^{-1} q^{-x \frac{\partial}{\partial x}} 
\operatorname{tr}_{\tilde{W}_n} Y\big( \mathcal{U}(\mathcal{U}(x),x) \cdot 
\mathcal{Y}_1(\mathcal{U}(x_1 w_1),x_1) \cdots \\
&\qquad \cdot \mathcal{Y}_n(\mathcal{U}(x_n w_n),x_n)
e^{2\pi i s h_{(0)}} q^{L_{(0)}} \big),
\end{split}
\end{align}
where we use (\ref{coorcha1}) for the first identity, and the $L_{(0)}$-conjugation property and trace property ($\operatorname{tr}(AB) = \operatorname{tr}(BA)$) for the second identity. From (\ref{iffq3}), we obtain
\begin{align}\label{iffq4}
\begin{split}
  &(1 - (-1)^t \theta^{-1} q^{-x\frac{\partial}{\partial x}}) \\
  &\quad \cdot \operatorname{tr}_{\tilde{W}_n}
    \Y_1(\mathcal{U}(x_1)u_1,x_1) \cdots \Y_n(\mathcal{U}(x_n)u_n,x_n)
    Y(\mathcal{U}(x)u,x) e^{2\pi i s h_{(0)}} q^{L_{(0)}} \\
  &=
    \operatorname{tr}_{\tilde{W}_n}
    \Res_{y} \sum_{i=1}^{n} \sum_{m \geq 0}
    \frac{(2\pi i)^{m+1} y^m}{m!}
    \left( \left( x_i \frac{\partial}{\partial x_i} \right)^m
    \left( \delta\left( \frac{x}{x_i} \right)
    \left( \frac{x}{x_i} \right)^{-r/T} \right) \right) \\
  &\quad\quad\cdot
    \operatorname{tr}_{\tilde{W}_n}
    \Y_1(\mathcal{U}(x_1)u_1,x_1) \cdots
    \Y_{i-1}(\mathcal{U}(x_{i-1})u_{i-1},x_{i-1})
    Y_i(\mathcal{U}(x_i)Y(u,y)u_i,x_i) \cdots \\
  &\quad\quad\cdot
    \Y_n(\mathcal{U}(x_n)u_n,x_n) e^{2\pi i s h_{(0)}} q^{L_{(0)}} \\
  &=
    \sum_{i=1}^n \sum_{m \in \mathbb{N}} \sum_{l \in \mathbb{Z}_+}
    \frac{(2\pi i)^{m+1}}{m!}
    \left( \left( x_i \frac{\partial}{\partial x_i} \right)^m
    \left( \frac{x^{l - r/T}}{x_i^{l - r/T}} \right) \right) \\
  &\quad\quad\cdot
    \operatorname{tr}_{\tilde{W}_n}
    \Y_1(\mathcal{U}(x_1)u_1,x_1) \cdots
    \Y_{i-1}(\mathcal{U}(x_{i-1})u_{i-1},x_{i-1})
    \Y_i(\mathcal{U}(x_i)u_{(m)}u_i,x_i) \cdots \\
  &\quad\quad\cdot
    \Y_n(\mathcal{U}(x_n)u_n,x_n) e^{2\pi i s h_{(0)}} q^{L_{(0)}} \\
  &\quad\quad
    + \sum_{i=1}^n \sum_{m \in \mathbb{N}} \sum_{l \in \mathbb{N}}
    \frac{(2\pi i)^{m+1}}{m!}
    \left( \left( x_i \frac{\partial}{\partial x_i} \right)^m
    \left( \frac{x^{-l - r/T}}{x_i^{-l - r/T}} \right) \right) \\
  &\quad\quad\cdot
    \operatorname{tr}_{\tilde{W}_n}
    \Y_1(\mathcal{U}(x_1)u_1,x_1) \cdots
    \Y_{i-1}(\mathcal{U}(x_{i-1})u_{i-1},x_{i-1})
    \Y_i(\mathcal{U}(x_i)u_{(m)}u_i,x_i) \cdots \\
  &\quad\quad\cdot
    \Y_n(\mathcal{U}(x_n)u_n,x_n) e^{2\pi i s h_{(0)}} q^{L_{(0)}}.
\end{split}
\end{align}
Since \( \left(1 - q^{-x\frac{\partial}{\partial x}}\right) \) acting on any term independent of \( x \) is zero, we obtain
\begin{align}\label{iffq5}
\begin{split}
  &\left(1 - (-1)^t \theta^{-1} q^{-x\frac{\partial}{\partial x}}\right)
    \Bigg(
      \operatorname{tr}_{\tilde{W}_n}
      \Y_1(\mathcal{U}(x_1)u_1,x_1) \cdots \Y_n(\mathcal{U}(x_n)u_n,x_n) \\
  &\quad\quad\quad\quad
      \cdot Y(\mathcal{U}(x)u,x) e^{2\pi i s h_{(0)}} q^{L_{(0)}} 
      - \delta_{(-1)^t\theta^{-1},1}
      \operatorname{tr}_{\tilde{W}_n} o(\mathcal{U}(1)u) \\
  &\quad\quad\quad\quad
      \cdot \Y_1(\mathcal{U}(x_1)u_1,x_1) \cdots \Y_n(\mathcal{U}(x_n)u_n,x_n)
      e^{2\pi i s h_{(0)}} q^{L_{(0)}}
    \Bigg) \\
  &=
    \sum_{i=1}^{n} \sum_{m \in \mathbb{N}}
    \left(
      \sum_{l \in \mathbb{Z}_+}
      \frac{(2\pi i)^{m+1}}{m!}
      \left(x_i \frac{\partial}{\partial x_i}\right)^m
      \left(\frac{x^{l - r/T}}{x_i^{l - r/T}}\right) \right. \\
  &\quad\quad\quad\quad\left.
      +
      \sum_{l \in \mathbb{N}}
      \frac{(2\pi i)^{m+1}}{m!}
      \left(x_i \frac{\partial}{\partial x_i}\right)^m
      \left(\frac{x^{-l - r/T}}{x_i^{-l - r/T}}\right)
    \right) \\
  &\quad\cdot
    \operatorname{tr}_{\tilde{W}_n}
    \Y_1(\mathcal{U}(x_1)u_1,x_1) \cdots
    \Y_{i-1}(\mathcal{U}(x_{i-1})u_{i-1},x_{i-1})
    \Y_i(\mathcal{U}(x_i)u_{(m)}u_i,x_i) \cdots \\
  &\quad\cdot
    \Y_n(\mathcal{U}(x_n)u_n,x_n)
    e^{2\pi i s h_{(0)}} q^{L_{(0)}}.
\end{split}
\end{align}
Note that when \( r = 0 \) and \( m = 0 \), the RHS contains a term
\begin{align}\label{iffq9}
\begin{split}
  2\pi i \sum_{i=1}^{n} \operatorname{tr}_{\tilde{W}_n}&
  \Y_1(\mathcal{U}(x_{1})u_{1},x_{1}) \cdots 
  \Y_{i-1}(\mathcal{U}(x_{i-1})u_{i-1},x_{i-1})
  \Y_i(\mathcal{U}(x_{i})u_{(0)}u_{i},x_{i}) \cdot\\
  &\quad\cdot 
  \Y_{i+1}(\mathcal{U}(x_{i+1})u_{i+1},x_{i+1}) \cdots 
  \Y_n(\mathcal{U}(x_{n})u_{n},x_{n})
  e^{2\pi i s h_{(0)}} q^{L_{(0)}}.
\end{split}
\end{align}
Since the LHS of \eqref{iffq5} has no constant term as a series in \( x \), \eqref{iffq9} must vanish, and we obtain \eqref{iffq8}.

In the rest of the proof, we assume that \( \phi \neq 1 \), i.e., \( r \neq 0 \).  
We next compute the inverse of \( \left(1 - (-1)^{t} \theta^{-1} q^{-x\frac{\partial}{\partial x}} \right) \) on the ring \( \mathbb{C}((q^{\frac{1}{T}})) \).  
(We work in \( \mathbb{C}((q^{\frac{1}{T}})) \) because the weights of a \( g \)-twisted module are always lower truncated.)

 First, 
\begin{align}\label{iffq6}
  \left(1 - (-1)^{t} \theta^{-1} q^{-x\frac{\partial}{\partial x}} \right)^{-1} x^{l - r/T}
  =
  \begin{cases}
    \displaystyle\frac{-(-1)^{t} \theta\, q^{l - r/T} x^{l - r/T}}{1 - (-1)^{t} \theta\, q^{l - r/T}} & \text{if } l > 0,\\[1ex]
    \displaystyle\frac{x^{l - r/T}}{1 - (-1)^{t} \theta^{-1} q^{-(l - r/T)}} & \text{if } l \leq 0.
  \end{cases}
\end{align}

Then,
\begin{align}\label{iffq7}
\begin{split}
  &\left(1 - (-1)^t \theta^{-1} q^{-x\frac{\partial}{\partial x}}\right)^{-1}
  \sum_{i=1}^{n} \sum_{m \in \mathbb{N}} \Bigg(
    \sum_{l \in \mathbb{Z}_+}
    \frac{(2\pi i)^{m+1}}{m!}
    \left(x_i \frac{\partial}{\partial x_i}\right)^m
    \left(\frac{x^{l - r/T}}{x_i^{l - r/T}}\right) \\
  &\quad +
    \sum_{l \in \mathbb{N}}
    \frac{(2\pi i)^{m+1}}{m!}
    \left(x_i \frac{\partial}{\partial x_i}\right)^m
    \left(\frac{x^{-l - r/T}}{x_i^{-l - r/T}}\right)
  \Bigg)\\
  &=
  \sum_{i=1}^{n} \sum_{m \in \mathbb{N}} \Bigg(
    \sum_{l \in \mathbb{Z}_+}
    \frac{(2\pi i)^{m+1}}{m!}
    \left(x_i \frac{\partial}{\partial x_i}\right)^m
    \left(
      \frac{-(-1)^t \theta\, q^{l - r/T}\, x^{l - r/T}}{(1 - (-1)^t \theta\, q^{l - r/T}) x_i^{l - r/T}}
    \right)\\
  &\quad +
    \sum_{l \in \mathbb{N}}
    \frac{(2\pi i)^{m+1}}{m!}
    \left(x_i \frac{\partial}{\partial x_i}\right)^m
    \left(
      \frac{x^{-l - r/T}}{(1 - (-1)^t \theta^{-1} q^{l + r/T}) x_i^{-l - r/T}}
    \right)
  \Bigg)\\
  &=
  \sum_{i=1}^{n} \sum_{m \in \mathbb{N}} \Bigg(
    \sum_{l \in \mathbb{Z}_+}
    \frac{(-1)^{m + t + 1} (2\pi i)^{m+1} (l + r/T)^m}{m!}
    \frac{\theta\, q^{l - r/T} \left(\frac{x_i}{x}\right)^{-l + r/T}}{1 - (-1)^t \theta\, q^{l - r/T}}\\
  &\quad +
    \sum_{l \in \mathbb{N}}
    \frac{(2\pi i)^{m+1} (l + r/T)^m}{m!}
    \frac{\left(\frac{x_i}{x}\right)^{l + r/T}}{1 - (-1)^t \theta^{-1} q^{l + r/T}}
  \Bigg)\\
  &= \sum_{i=1}^{n} \sum_{m \in \mathbb{N}} \tilde{P}_{m+1}^{(-1)^t \theta, \phi}\left(\frac{x_i}{x}, q\right).
\end{split}
\end{align}

We obtain \eqref{iffq2} from \eqref{iffq7} and \eqref{iffq5}.

\end{proof}

Assume that Conjecture \ref{duality} holds 
when $g$ is trivial. Let $W_i$, $\tilde{W}_i$, $i=1,\dots,n$, satisfy the assumptions of Conjecture \ref{duality} or the finiteness condition. Then for any $v_i\in W_i$, $\tilde{w}_n\in \tilde{W}_n$, and $\tilde{w}_0' \in \tilde{W}_0'$, the expression
\[
    \langle \tilde{w}_0',\, \Y_1(v_{1},z_{1})\cdots \Y_n(v_{n},z_{n})\,\tilde{w}_n\rangle
\]
is absolutely convergent when $|z_{1}| > \cdots > |z_{n}| > 0$.

\begin{theorem}\label{iffq22}
Let $W_i$, $\tilde{W}_i$, $i=1,\dots,n$, be as above. For any $0\leq j\leq n$, we have
\begin{align}\label{iffq21}
\begin{split}
   &\operatorname{tr}_{\tilde{W}_n}\Y_1(\mathcal{U}(x_{1})u_{1},x_{1})\cdots \Y_{j-1}(\mathcal{U}(x_{j-1})u_{j-1},x_{j-1})\, \Y_j(\mathcal{U}(x_{j})Y(u,y)u_{j},x_{j})\\
   &\quad \cdot \Y_{j+1}(\mathcal{U}(x_{j+1})u_{j+1},x_{j+1})\cdots \Y_n(\mathcal{U}(x_{n})u_{n},x_{n})\,e^{2\pi i s h_{(0)}}q^{L_{(0)}}\\
   &=\sum_{m\in \mathbb{N}} \tilde{P}_{m+1}^{(-1)^{t}\theta,\phi}\biggl(\frac{1}{e^{2\pi i y}},q\biggr) \operatorname{tr}_{\tilde{W}_n}\Y_1(\mathcal{U}(x_{1})u_{1},x_{1})\cdots\\
   &\quad \cdot \Y_{j-1}(\mathcal{U}(x_{j-1})u_{j-1},x_{j-1})\,\Y_j(\mathcal{U}(x_{j})u_{(m)}u_j,x_{j})\,\Y_{j+1}(\mathcal{U}(x_{j+1})u_{j+1},x_{j+1})\\
   &\quad \cdots \Y_n(\mathcal{U}(x_{n})u_{n},x_{n})\,e^{2\pi i s h_{(0)}}q^{L_{(0)}}\\
   &\quad + \sum_{i\neq j} \sum_{m\in \mathbb{N}} \tilde{P}_{m+1}^{(-1)^{t}\theta,\phi}\biggl(\frac{x_{i}}{x_{j}e^{2\pi i y}},q\biggr) \operatorname{tr}_{\tilde{W}_n} \Y_1(\mathcal{U}(x_{1})u_{1},x_{1}) \cdots\\
   &\quad \cdot \Y_{i-1}(\mathcal{U}(x_{i-1})u_{i-1},x_{i-1})\,\Y_i(\mathcal{U}(x_{i})u_{(m)}u_{i},x_{i})\,\Y_{i+1}(\mathcal{U}(x_{i+1})u_{i+1},x_{i+1})\cdots\\
   &\quad \cdot \Y_n(\mathcal{U}(x_{n})u_{n},x_{n})\,e^{2\pi i s h_{(0)}}q^{L_{(0)}}\\
   &\quad + \delta_{(-1)^t\theta^{-1},1} \operatorname{tr}_{\tilde{W}_n}o(\mathcal{U}(1)u)\, \Y_1(\mathcal{U}(x_{1})u_{1},x_{1}) \cdots \Y_n(\mathcal{U}(x_{n})u_{n},x_{n})\,e^{2\pi i s h_{(0)}}q^{L_{(0)}}.
\end{split}
\end{align}
\end{theorem}
\begin{proof}
We prove the theorem by induction on \( j \).

First, we prove the base case \( j=1 \). From \eqref{coorcha} and the \( L_{(0)} \)-conjugation property, we have
\begin{align}\label{iffq23}
\begin{split}
\mathcal{U}(x_1) Y(u, y) &= x_1^{L_{(0)}} \mathcal{U}(1) Y(u, y) \mathcal{U}(1)^{-1} \mathcal{U}(1) \\
&= x_1^{L_{(0)}} Y(\mathcal{U}(e^{2\pi i y}) u, e^{2\pi i y} - 1) \mathcal{U}(1) \\
&= Y(\mathcal{U}(x_1 e^{2\pi i y}) u, x_1(e^{2\pi i y} - 1)) \mathcal{U}(x_1).
\end{split}
\end{align}

\noindent Thus, the left-hand side of \eqref{iffq21} for \( j=1 \) becomes
\begin{align}\label{iffq24}
\begin{split}
\operatorname{tr}_{\tilde{W}_n} &\mathcal{Y}_1(\mathcal{U}(x_1) Y(u, y) u_1, x_1) \mathcal{Y}_2(\mathcal{U}(x_2) u_2, x_2) \cdots \mathcal{Y}_n(\mathcal{U}(x_n) u_n, x_n) e^{2\pi i s h_{(0)}} q^{L_{(0)}} \\
&= \operatorname{tr}_{\tilde{W}_n} \mathcal{Y}_1(Y(\mathcal{U}(x_1 e^{2\pi i y}) u, x_1(e^{2\pi i y} - 1)) \mathcal{U}(x_1) u_1, x_1) \cdots \\
&\qquad \cdots \mathcal{Y}_n(\mathcal{U}(x_n) u_n, x_n) e^{2\pi i s h_{(0)}} q^{L_{(0)}}.
\end{split}
\end{align}

\noindent Now let \( y = z \), \( x_1 = z_1 \in \mathbb{C} \) such that \( |z_1 q_z| > |z_1| > |z_1(q_z - 1)| > 0 \), where \( z_1^n := e^{n \log z_1} \) with \( 0 \leq \arg z_1 < 2\pi \). By \eqref{newprop3},
\begin{align}\label{iffq24b}
\begin{split}
\operatorname{tr}_{\tilde{W}_n} &\mathcal{Y}_1(Y(\mathcal{U}(z_1 q_z) u, z_1(q_z - 1)) \mathcal{U}(z_1) u_1, z_1) \cdots \\
&\quad \cdot \mathcal{Y}_n(\mathcal{U}(x_n) u_n, x_n) e^{2\pi i s h_{(0)}} q^{L_{(0)}} \\
&= \operatorname{tr}_{\tilde{W}_n} Y(\mathcal{U}(z_1 q_z) u, z_1 q_z) \mathcal{Y}_1(\mathcal{U}(z_1) u_1, z_1) \cdots \mathcal{Y}_n(\mathcal{U}(x_n) u_n, x_n) e^{2\pi i s h_{(0)}} q^{L_{(0)}}.
\end{split}
\end{align}

\noindent Replacing \( z_1 \) with the formal variable \( x_1 \), and applying \eqref{iffq2}, we obtain
\begin{align}\label{iffq25}
\begin{split}
\operatorname{tr}_{\tilde{W}_n} &\,
Y(\mathcal{U}(x_1 q_z) u, x_1 q_z)\,
\mathcal{Y}_1(\mathcal{U}(x_1) u_1, x_1) \cdots
\mathcal{Y}_n(\mathcal{U}(x_n) u_n, x_n)\,
e^{2\pi i s h_{(0)}} q^{L_{(0)}} \\
&= \sum_{i=1}^n \sum_{m \in \mathbb{N}}
\tilde{P}_{m+1}^{(-1)^t \theta, \phi}
\left( \frac{x_i}{x_1 q_z}, q \right) \\
&\quad \cdot \operatorname{tr}_{\tilde{W}_n}\,
\mathcal{Y}_1(\mathcal{U}(x_1) u_1, x_1) \cdots
\mathcal{Y}_i(\mathcal{U}(x_i) u_{(m)} u_i, x_i) \cdots \\
&\quad\quad \cdots \mathcal{Y}_n(\mathcal{U}(x_n) u_n, x_n)\,
e^{2\pi i s h_{(0)}} q^{L_{(0)}} \\
&\quad + \delta_{(-1)^t \theta^{-1}, 1} \cdot
\operatorname{tr}_{\tilde{W}_n}\,
o(\mathcal{U}(1) u)\,
\mathcal{Y}_1(\mathcal{U}(x_1) u_1, x_1) \cdots \\
&\quad\quad \cdots \mathcal{Y}_n(\mathcal{U}(x_n) u_n, x_n)\,
e^{2\pi i s h_{(0)}} q^{L_{(0)}}.
\end{split}
\end{align}

\noindent Hence \eqref{iffq21} holds when \( j=1 \).

Now assume \eqref{iffq21} holds for \( j=k < n \). We prove it for \( j = k+1 \). Fix \( z_i^0 \in \mathbb{C} \) such that \( |z_1^0| > \cdots > |z_n^0| > 0 \), and define the path
\[
\gamma_k(t) = \left( z_1^0, \ldots, e^{(1-t)\log z_k^0 + t \log z_{k+1}^0}, e^{t \log z_k^0 + (1 - t) \log z_{k+1}^0}, \ldots, z_n^0 \right)
\]
in
\[
M^n = \left\{ (z_1, \ldots, z_n) \in \mathbb{C}^n \,\middle|\, z_i \neq 0,\, z_i \neq z_j \text{ for } i \neq j \right\}.
\]
Any branch of a multivalued analytic function on the simply connected region \( |z_1| > \cdots > |z_n| \), \( 0 \leq \arg z_i < 2\pi \), extends uniquely along \( \gamma_k \) to the region:
\begin{align}\label{iffq27}
\begin{split}
&|z_1| > \cdots > |z_{k-1}| > |z_{k+1}| > |z_k| > |z_{k+2}| > \cdots > |z_n| > 0, \\
&0 \leq \arg z_i < 2\pi, \quad i = 1, \ldots, n.
\end{split}
\end{align}
By the duality property of intertwining operators, the analytic continuation along \( \gamma_k \) of
\[
\langle \tilde{w}_0', \mathcal{Y}_1(u_1, z_1) \cdots \mathcal{Y}_k(u_k, z_k) \mathcal{Y}_{k+1}(u_{k+1}, z_{k+1}) \cdots \mathcal{Y}_n(u_n, z_n) \tilde{w}_n \rangle
\]
coincides with
\[
\langle \tilde{w}_0', \mathcal{Y}_1(u_1, z_1) \cdots \hat{\mathcal{Y}}_{k+1}(u_{k+1}, z_{k+1}) \hat{\mathcal{Y}}_k(u_k, z_k) \cdots {\mathcal{Y}}_n(u_n, z_n) \tilde{w}_n \rangle,
\] for certain intertwining operator $\hat{\mathcal{Y}}_k$ and $\hat{\mathcal{Y}}_{k+1}$. 
Now consider
\begin{align}\label{iffq28}
\begin{split}
\operatorname{tr}_{\tilde{W}_n} &\mathcal{Y}_1(u_1, x_1) \cdots \mathcal{Y}_k(\mathcal{U}(x_k) u_k, x_k) \mathcal{Y}_{k+1}(\mathcal{U}(x_{k+1}) Y(u, y) u_{k+1}, x_{k+1}) \cdots \\
&\cdots \mathcal{Y}_n(\mathcal{U}(x_n) u_n, x_n) e^{2\pi i s h_{(0)}} q^{L_{(0)}}.
\end{split}
\end{align}
Substituting \( x_1 = z_1, \ldots, x_k = z_{k+1}, x_{k+1} = z_k, \ldots, x_n = z_n \), we land in the region \eqref{iffq27}. The analytic extension of these coefficients back along \( \gamma_k^{-1} \), combined with the induction hypothesis, shows that \eqref{iffq21} holds for \( j = k+1 \). We are done.

\end{proof}

Let $\mathbb{G}_{|z_{1}|>\cdots>|z_{n}|>0}$ be the space of all multivalued analytic functions in $z_{1},...,z_{n}$ defined on the region $|z_{1}|>\cdots>|z_{n}|>0$ with preferred branches in the simply-connected region $|z_{1}|>\cdots>|z_{n}|>0$, $0\leq \operatorname{arg}z_{i}<2\pi,$ $i=1,...,n$.
Let $W_i$ $\tilde{W}_i$, $h$ be as in Proposition \ref{iffq22}, for any $v_i\in W_i$, we see that
\begin{align*}
     \operatorname{tr}_{\tilde{W}_n}\Y_1(\mathcal{U}(q_{z_{1}})v_{1},q_{z_{1}})\cdots \Y_n(\mathcal{U}(q_{z_{n}})v_{n},q_{z_{n}})e^{2\pi i s h_{(0)}}q^{L_{(0)}}
\end{align*}
is absolutely convergent when $|q_{z_{1}}|>\cdots>|q_{z_{n}}|>0$.

\begin{theorem}\label{iffq29}
For any integer $j$ satisfying $1 \leq j \leq n$ and any $l \in \mathbb{Z}_{+}$, in $\mathbb{G}_{|q_{z_{1}}|>\cdots>|q_{z_{n}}|>0}$, we have when $\phi = 1$ and $(-1)^{t}\theta \neq 1$,
\begin{align}\label{iffq210}
\begin{split}
  &\operatorname{tr}_{\tilde{W}_n} \Y_1(\mathcal{U}(q_{z_1})u_1, q_{z_1}) \cdots \Y_{j-1}(\mathcal{U}(q_{z_{j-1}})u_{j-1}, q_{z_{j-1}}) \\
  &\quad \cdot \Y_j(\mathcal{U}(q_{z_j})u_{(-l)}u_j, q_{z_j}) \cdot \Y_{j+1}(\mathcal{U}(q_{z_{j+1}})u_{j+1}, q_{z_{j+1}}) \cdots \\
  &\quad \cdot \Y_n(\mathcal{U}(q_{z_n})u_n, q_{z_n})\, e^{2\pi i s h_{(0)}} q^{L_{(0)}} \\
  &= \sum_{k \geq l} \left( (-1)^{l-1} \binom{k-1}{k-l} \tilde{G}_k 
     \genfrac[]{0pt}{0}{(-1)^t\theta}{\phi}(q) 
     - \delta_{k,l} \cdot \frac{2\pi i}{1 - \theta^{-1}} \right) \\
  &\quad \cdot \operatorname{tr}_{\tilde{W}_n} \Y_1(\mathcal{U}(q_{z_1})u_1, q_{z_1}) \cdots 
     \Y_{j-1}(\mathcal{U}(q_{z_{j-1}})u_{j-1}, q_{z_{j-1}}) \\
  &\quad \cdot \Y_j(\mathcal{U}(q_{z_j})u_{(k-l)}u_j, q_{z_j}) \cdots 
     \Y_n(\mathcal{U}(q_{z_n})u_n, q_{z_n})\, e^{2\pi i s h_{(0)}} q^{L_{(0)}} \\
  &\quad + \sum_{i \neq j} \sum_{m \in \mathbb{N}} 
     \left( (-1)^{m+1} \binom{m+l-1}{l-1} 
     \tilde{P}_{m+l} \genfrac[]{0pt}{0}{(-1)^t\theta}{\phi}(z_i - z_j, q) 
     - \delta_{0,m} \cdot \frac{2\pi i}{1 - \theta^{-1}} \right) \\
  &\quad \cdot \operatorname{tr}_{\tilde{W}_n} 
     \Y_1(\mathcal{U}(q_{z_1})u_1, q_{z_1}) \cdots 
     \Y_{i-1}(\mathcal{U}(q_{z_{i-1}})u_{i-1}, q_{z_{i-1}}) \\
  &\quad \cdot \Y_i(\mathcal{U}(q_{z_i})u_{(m)}u_i, q_{z_i}) \cdots 
     \Y_n(\mathcal{U}(q_{z_n})u_n, q_{z_n})\, e^{2\pi i s h_{(0)}} q^{L_{(0)}} \\
  &\quad + \delta_{l,1} \delta_{(-1)^t\theta^{-1},1} 
     \operatorname{tr}_{\tilde{W}_n} o(\mathcal{U}(1)u) 
     \Y_1(\mathcal{U}(q_{z_1})u_1, q_{z_1}) \cdots \\
  &\quad \cdot \Y_n(\mathcal{U}(q_{z_n})u_n, q_{z_n})\, 
     e^{2\pi i s h_{(0)}} q^{L_{(0)}}.
\end{split}
\end{align}
When $\phi \neq 1$ or $(-1)^t \theta = 1$, the LHS of \eqref{iffq210} remains the same and one omits the terms $\delta_{k,l} \cdot 2\pi i \cdot \frac{1}{1 - \theta^{-1}}$ and $\delta_{0,m} \cdot 2\pi i \cdot \frac{1}{1 - \theta^{-1}}$ on the RHS.
\end{theorem}
\begin{proof}
As a $q$-series, the coefficients of $P_{m}^{\theta, \phi}\left(\frac{q_{z_i}}{q_{z_j} q_z}, q\right)$ for $m > 0$ are convergent in the region $\left| \frac{q_{z_i}}{q_{z_j} q_z} \right| < 1$. Let $z, z_1, \ldots, z_n \in \mathbb{C}$ satisfy $|q_{z_1}| > \cdots > |q_{z_n}| > 0$ and $\left| \frac{q_{z_i}}{q_{z_j} q_z} \right| < 1$ for $1 \leq i,j \leq n$. 
We substitute $z, q_{z_1}, \ldots, q_{z_n}$ for $y, x_1, \ldots, x_n$ in \eqref{iffq21}. Then both sides of \eqref{iffq21} become $q$-series with analytic functions in $z, z_1, \ldots, z_n$. Additionally, by using \eqref{pqfunction7}, we get
\begin{align}\label{iffq211}
\begin{split}
  &\operatorname{tr}_{\tilde{W}_n} \Y_1(\mathcal{U}(q_{z_1})u_1, q_{z_1}) \cdots \Y_{j-1}(\mathcal{U}(q_{z_{j-1}})u_{j-1}, q_{z_{j-1}}) \cdot \Y_j(\mathcal{U}(q_{z_j})Y(u, z)u_j, q_{z_j}) \cdot \\
  &\quad\cdot \Y_{j+1}(\mathcal{U}(q_{z_{j+1}})u_{j+1}, q_{z_{j+1}}) \cdots \Y_n(\mathcal{U}(q_{z_n})u_n, q_{z_n}) e^{2\pi i s h_{(0)}} q^{L_{(0)}} \\
  &= \sum_{m \in \mathbb{N}} \tilde{P}_{m+1}^{(-1)^t\theta, \phi}(-z, q) \cdot \operatorname{tr}_{\tilde{W}_n} \Y_1(\mathcal{U}(q_{z_1})u_1, q_{z_1}) \cdots \Y_{j-1}(\mathcal{U}(q_{z_{j-1}})u_{j-1}, q_{z_{j-1}}) \\
  &\quad\cdot \Y_j(\mathcal{U}(q_{z_j})u_{(m)}u_j, q_{z_j}) \cdots \Y_n(\mathcal{U}(q_{z_n})u_n, q_{z_n}) e^{2\pi i s h_{(0)}} q^{L_{(0)}} \\
  &\quad + \sum_{i \neq j} \sum_{m \in \mathbb{N}} \tilde{P}_{m+1}^{(-1)^t\theta, \phi}(z_i - z_j - z, q) \cdot \operatorname{tr}_{\tilde{W}_n} \Y_1(\mathcal{U}(q_{z_1})u_1, q_{z_1}) \cdots \\
  &\quad\cdot \Y_{i-1}(\mathcal{U}(q_{z_{i-1}})u_{i-1}, q_{z_{i-1}}) \Y_i(\mathcal{U}(q_{z_i})u_{(m)}u_i, q_{z_i}) \cdots \Y_n(\mathcal{U}(q_{z_n})u_n, q_{z_n}) e^{2\pi i s h_{(0)}} q^{L_{(0)}} \\
  &\quad + \delta_{(-1)^t\theta^{-1}, 1} \cdot \operatorname{tr}_{\tilde{W}_n} o(\mathcal{U}(1)u) \Y_1(\mathcal{U}(q_{z_1})u_1, q_{z_1}) \cdots \Y_n(\mathcal{U}(q_{z_n})u_n, q_{z_n}) e^{2\pi i s h_{(0)}} q^{L_{(0)}}.
\end{split}
\end{align}
By \eqref{pqfunction6}, the Taylor expansion of the twisted Weirstrass function $\tilde{P}_{m+1}\genfrac[]{0pt}{0}{\theta}{\phi}(z_i - z_j - z; q)$ as a function in $z$ at $0$ is
\begin{align}\label{iffq211taylor}
  \sum_{n=0}^{\infty} \binom{m+n}{m} \tilde{P}_{m+n+1}\genfrac[]{0pt}{0}{\theta}{\phi}(z_i - z_j, q) z^n.
\end{align}

\noindent Then, by \eqref{pqfunction4}, \eqref{iffq211taylor}, \eqref{pqfunction71}, and taking the coefficient of $z^{l-1}$ on both sides of \eqref{iffq211}, we obtain \eqref{iffq210}.
\end{proof}

\begin{corollary}
If $u$ is an even element fixed by $g$, and $\tilde{W}_n$ is an ordinary module, then for $l \in \mathbb{Z}_{+}$, we have
\begin{align}\label{iffq212}
\begin{split}
   &\operatorname{tr}_{\tilde{W}_n} \Y_1(\mathcal{U}(q_{z_1})u_1, q_{z_1}) \cdots \Y_{j-1}(\mathcal{U}(q_{z_{j-1}})u_{j-1}, q_{z_{j-1}}) \cdot \Y_j(\mathcal{U}(q_{z_j})u_{(-l)}u_j, q_{z_j}) \cdot \\
   &\quad\cdot \Y_{j+1}(\mathcal{U}(q_{z_{j+1}})u_{j+1}, q_{z_{j+1}}) \cdots \Y_n(\mathcal{U}(q_{z_n})u_n, q_{z_n}) e^{2\pi i s h_{(0)}} q^{L_{(0)}} \\
   &= \sum_{\substack{2k \geq l \\ k \geq 2}} (-1)^{l-1} \binom{2k - 1}{2k - l} \tilde{G}_{2k}(q) \cdot \operatorname{tr}_{\tilde{W}_n} \Y_1(\mathcal{U}(q_{z_1})u_1, q_{z_1}) \cdots \\
   &\quad\cdot \Y_{j-1}(\mathcal{U}(q_{z_{j-1}})u_{j-1}, q_{z_{j-1}}) \cdot \Y_j(\mathcal{U}(q_{z_j})u_{(2k - l)}u_j, q_{z_j}) \cdots \\
   &\quad\cdot \Y_n(\mathcal{U}(q_{z_n})u_n, q_{z_n}) e^{2\pi i s h_{(0)}} q^{L_{(0)}} \\
   &\quad+ \sum_{i \neq j} \sum_{m \in \mathbb{N}} (-1)^{m+1} \binom{m + l - 1}{l - 1} \tilde{\wp}_{m+l}(z_i - z_j, q) \cdot \\
   &\quad\cdot \operatorname{tr}_{\tilde{W}_n} \Y_1(\mathcal{U}(q_{z_1})u_1, q_{z_1}) \cdots \Y_{i-1}(\mathcal{U}(q_{z_{i-1}})u_{i-1}, q_{z_{i-1}}) \cdot \Y_i(\mathcal{U}(q_{z_i})u_{(m)}u_i, q_{z_i}) \cdot \\
   &\quad\cdot \Y_{i+1}(\mathcal{U}(q_{z_{i+1}})u_{i+1}, q_{z_{i+1}}) \cdots \Y_n(\mathcal{U}(q_{z_n})u_n, q_{z_n}) e^{2\pi i s h_{(0)}} q^{L_{(0)}} \\
   &\quad+ \delta_{l,1} \tilde{G}_2(q) \sum_{i=1}^{n} \operatorname{tr}_{\tilde{W}_n} \Y_1(\mathcal{U}(q_{z_1})u_1, q_{z_1}) \cdots \\
   &\quad\cdot \Y_{i-1}(\mathcal{U}(q_{z_{i-1}})u_{i-1}, q_{z_{i-1}}) \cdot \Y_i(\mathcal{U}(q_{z_i})(u_{(0)} + z_i u_{(1)})u_i, q_{z_i}) \cdots \\
   &\quad\cdot \Y_n(\mathcal{U}(q_{z_n})u_n, q_{z_n}) e^{2\pi i s h_{(0)}} q^{L_{(0)}} \\
   &\quad+ \delta_{l,1} \operatorname{tr}_{\tilde{W}_n} o(\mathcal{U}(1)u) \Y_1(\mathcal{U}(q_{z_1})u_1, q_{z_1}) \cdots \Y_n(\mathcal{U}(q_{z_n})u_n, q_{z_n}) e^{2\pi i s h_{(0)}} q^{L_{(0)}}.
\end{split}
\end{align}
\end{corollary}

\begin{proof}
In this case, we have $t = 0$ and $\theta = \phi = 1$. By \eqref{pqfunction3} and Theorem~\ref{iffq29}, we obtain
\begin{align}\label{iffq2121}
\begin{split}
   &\operatorname{tr}_{\tilde{W}_n} \Y_1(\mathcal{U}(q_{z_1})u_1, q_{z_1}) \cdots \Y_{j-1}(\mathcal{U}(q_{z_{j-1}})u_{j-1}, q_{z_{j-1}}) \cdot \Y_j(\mathcal{U}(q_{z_j})u_{(-l)}u_j, q_{z_j}) \cdot \\
   &\quad\cdot \Y_{j+1}(\mathcal{U}(q_{z_{j+1}})u_{j+1}, q_{z_{j+1}}) \cdots \Y_n(\mathcal{U}(q_{z_n})u_n, q_{z_n}) e^{2\pi i s h_{(0)}} q^{L_{(0)}} \\
   &= \pi i \delta_{l,1} \operatorname{tr}_{\tilde{W}_n} \Y_1(\mathcal{U}(q_{z_1})u_1, q_{z_1}) \cdots \Y_{j-1}(\mathcal{U}(q_{z_{j-1}})u_{j-1}, q_{z_{j-1}}) \cdot \\
   &\quad\cdot \Y_j(\mathcal{U}(q_{z_j})u_{(0)}u_j, q_{z_j}) \cdots \Y_n(\mathcal{U}(q_{z_n})u_n, q_{z_n}) e^{2\pi i s h_{(0)}} q^{L_{(0)}} \\
   &\quad+ \sum_{\substack{2k \geq l \\ k \geq 2}} (-1)^{l-1} \binom{2k - 1}{2k - l} \tilde{G}_{2k}(q) \cdot \operatorname{tr}_{\tilde{W}_n} \Y_1(\mathcal{U}(q_{z_1})u_1, q_{z_1}) \cdots \\
   &\quad\cdot \Y_{j-1}(\mathcal{U}(q_{z_{j-1}})u_{j-1}, q_{z_{j-1}}) \cdot \Y_j(\mathcal{U}(q_{z_j})u_{(2k - l)}u_j, q_{z_j}) \cdots \\
   &\quad\cdot \Y_n(\mathcal{U}(q_{z_n})u_n, q_{z_n}) e^{2\pi i s h_{(0)}} q^{L_{(0)}} \\
   &\quad+ \sum_{i \neq j} \sum_{m \in \mathbb{N}} (-1)^{m+1} \binom{m + l - 1}{l - 1} \tilde{P}_{m+l}^{1,1}(z_i - z_j, q) \cdot \\
   &\quad\cdot \operatorname{tr}_{\tilde{W}_n} \Y_1(\mathcal{U}(q_{z_1})u_1, q_{z_1}) \cdots \Y_{i-1}(\mathcal{U}(q_{z_{i-1}})u_{i-1}, q_{z_{i-1}}) \cdot \Y_i(\mathcal{U}(q_{z_i})u_{(m)}u_i, q_{z_i}) \cdots \\
   &\quad\cdot \Y_n(\mathcal{U}(q_{z_n})u_n, q_{z_n}) e^{2\pi i s h_{(0)}} q^{L_{(0)}} \\
   &\quad+ \delta_{l,1} \operatorname{tr}_{\tilde{W}_n} o(\mathcal{U}(1)u) \Y_1(\mathcal{U}(q_{z_1})u_1, q_{z_1}) \cdots \Y_n(\mathcal{U}(q_{z_n})u_n, q_{z_n}) e^{2\pi i s h_{(0)}} q^{L_{(0)}}.
\end{split}
\end{align}
The result follows from \eqref{weirstrasswp} and Theorem~\ref{iffq29}.
\end{proof}

\begin{remark}
Identity (\ref{iffq212}) recovers Theorem 2.4 in \cite{huang2005differential}.
\end{remark}
\begin{corollary}
For any integers $j, l$ satisfying $1 \leq j, l \leq n$ and any $l \in \mathbb{Z}_{+}$, in $\mathbb{G}_{|q_{z_1}| > \cdots > |q_{z_n}| > 0}$, we have, when $\phi = 1$ and $(-1)^t \theta \neq 1$,
\begin{equation}\label{iffq214}
\begin{aligned}
  &\operatorname{tr}_{\tilde{W}_n} \Y_1(\mathcal{U}(q_{z_1})u_1, q_{z_1}) \cdots \Y_{j-1}(\mathcal{U}(q_{z_{j-1}})u_{j-1}, q_{z_{j-1}}) \cdot \Y_j(\mathcal{U}(q_{z_j})u_{(-1)}u_j, q_{z_j}) \cdot \\
  &\quad\cdot \Y_{j+1}(\mathcal{U}(q_{z_{j+1}})u_{j+1}, q_{z_{j+1}}) \cdots \Y_n(\mathcal{U}(q_{z_n})u_n, q_{z_n}) e^{2\pi i s h_{(0)}} q^{L_{(0)}} \\
  &\quad- \sum_{k \in \mathbb{Z}_+} \left( \tilde{G}_k\genfrac[]{0pt}{0}{(-1)^t \theta}{\phi}(q) - \delta_{k,1} 2\pi i \frac{1}{1 - \theta^{-1}} \right) \cdot \\
  &\quad\cdot \operatorname{tr}_{\tilde{W}_n} \Y_1(\mathcal{U}(q_{z_1})u_1, q_{z_1}) \cdots \Y_{j-1}(\mathcal{U}(q_{z_{j-1}})u_{j-1}, q_{z_{j-1}}) \cdot \\
  &\quad\cdot \Y_j(\mathcal{U}(q_{z_j})u_{(n-1)}u_j, q_{z_j}) \cdots \Y_n(\mathcal{U}(q_{z_n})u_n, q_{z_n}) e^{2\pi i s h_{(0)}} q^{L_{(0)}} \\
  &\quad- \sum_{i \neq j} \sum_{m \in \mathbb{N}} \left( (-1)^{m+1} \tilde{P}_{m+1}\genfrac[]{0pt}{0}{(-1)^t \theta}{\phi}(z_i - z_j, q) - \delta_{0,m} 2\pi i \frac{1}{1 - \theta^{-1}} \right) \cdot \\
  &\quad\cdot \operatorname{tr}_{\tilde{W}_n} \Y_1(\mathcal{U}(q_{z_1})u_1, q_{z_1}) \cdots \Y_{i-1}(\mathcal{U}(q_{z_{i-1}})u_{i-1}, q_{z_{i-1}}) \cdot \\
  &\quad\cdot \Y_i(\mathcal{U}(q_{z_i})u_{(m)}u_i, q_{z_i}) \cdots \Y_n(\mathcal{U}(q_{z_n})u_n, q_{z_n}) e^{2\pi i s h_{(0)}} q^{L_{(0)}}
\end{aligned}
\end{equation}

\vspace{-1.2em}

\[ =\]

\vspace{-1.2em}

\begin{equation*}
\begin{aligned}
  &\operatorname{tr}_{\tilde{W}_n} \Y_1(\mathcal{U}(q_{z_1})u_1, q_{z_1}) \cdots \Y_{l-1}(\mathcal{U}(q_{z_{l-1}})u_{l-1}, q_{z_{l-1}}) \cdot \Y_l(\mathcal{U}(q_{z_l})u_{(-1)}u_l, q_{z_l}) \cdot \\
  &\quad\cdot \Y_{l+1}(\mathcal{U}(q_{z_{l+1}})u_{l+1}, q_{z_{l+1}}) \cdots \Y_n(\mathcal{U}(q_{z_n})u_n, q_{z_n}) e^{2\pi i s h_{(0)}} q^{L_{(0)}} \\
  &\quad- \sum_{k \in \mathbb{Z}_+} \left( \tilde{G}_k\genfrac[]{0pt}{0}{(-1)^t \theta}{\phi}(q) - \delta_{k,1} 2\pi i \frac{1}{1 - \theta^{-1}} \right) \cdot \\
  &\quad\cdot \operatorname{tr}_{\tilde{W}_n} \Y_1(\mathcal{U}(q_{z_1})u_1, q_{z_1}) \cdots \Y_{l-1}(\mathcal{U}(q_{z_{l-1}})u_{l-1}, q_{z_{l-1}}) \cdot \\
  &\quad\cdot \Y_l(\mathcal{U}(q_{z_l})u_{(n-1)}u_l, q_{z_l}) \cdots \Y_n(\mathcal{U}(q_{z_n})u_n, q_{z_n}) e^{2\pi i s h_{(0)}} q^{L_{(0)}} \\
  &\quad- \sum_{i \neq l} \sum_{m \in \mathbb{N}} \left( (-1)^{m+1} \tilde{P}_{m+1}\genfrac[]{0pt}{0}{(-1)^t \theta}{\phi}(z_i - z_l, q) - \delta_{0,m} 2\pi i \frac{1}{1 - \theta^{-1}} \right) \cdot \\
  &\quad\cdot \operatorname{tr}_{\tilde{W}_n} \Y_1(\mathcal{U}(q_{z_1})u_1, q_{z_1}) \cdots \Y_{i-1}(\mathcal{U}(q_{z_{i-1}})u_{i-1}, q_{z_{i-1}}) \cdot \\
  &\quad\cdot \Y_i(\mathcal{U}(q_{z_i})u_{(m)}u_i, q_{z_i}) \cdots \Y_n(\mathcal{U}(q_{z_n})u_n, q_{z_n}) e^{2\pi i s h_{(0)}} q^{L_{(0)}}
\end{aligned}
\end{equation*}
When $\phi \neq 1$ or $(-1)^t \theta = 1$, everything in \eqref{iffq214} remains the same after omitting the terms $\delta_{k,1} 2\pi i \frac{1}{1 - \theta^{-1}}$ and $\delta_{0,m} 2\pi i \frac{1}{1 - \theta^{-1}}$ on both sides.
\end{corollary}

\begin{proof}
By Theorem~\ref{iffq29}, both the left-hand side and right-hand side of \eqref{iffq214} equal
\[
\delta_{(-1)^t \theta^{-1}, 1} \cdot \operatorname{tr}_{\tilde{W}_n} o(\mathcal{U}(1)u) \Y_1(\mathcal{U}(q_{z_1})u_1, q_{z_1}) \cdots \Y_n(\mathcal{U}(q_{z_n})u_n, q_{z_n}) e^{2\pi i s h_{(0)}} q^{L_{(0)}},
\]
which proves the result.
\end{proof}

\subsection{The ring \texorpdfstring{$R'$}{Lg}}\label{r'introduction}
Let \( G_{2k}(\tau) \) for \( k \in \mathbb{Z}_{+} \) denote the classical \emph{Eisenstein series}. Denote by \( \wp_1(z,\tau) \) and \( \wp_2(z,\tau) \) the Weierstrass zeta and \(\wp\)-functions, respectively. More generally, for \( m \geq 1 \), we define the higher Weierstrass functions recursively by
\begin{align}
    \wp_{m+1}(z,\tau) := -\frac{1}{m} \frac{\partial}{\partial z} \wp_m(z,\tau).
\end{align}

For each integer \( k \in \mathbb{Z}_{+} \), we write \( P_k\genfrac[]{0pt}{0}{\theta}{\phi}(z,\tau) \) for the \emph{twisted Weierstrass function} on \( \mathbb{C} \times \mathbb{H} \), and for each \( n \in \mathbb{Z}_{+} \), we denote by \( G_n\genfrac[]{0pt}{0}{\theta}{\phi}(\tau) \) the corresponding \emph{twisted Eisenstein series}.

Recall that we adopt the convention that for a function \( f(\tau) \), its \( q \)-expansion is denoted by \( \tilde{f}(q) \), where \( q = e^{2\pi i \tau} \). In particular, we write
\[
\tilde{P}_k\genfrac[]{0pt}{0}{\theta}{\phi}(z, q) \quad \text{and} \quad \tilde{G}_n\genfrac[]{0pt}{0}{\theta}{\phi}(q)
\]
to denote the \( q \)-expansions of the twisted Weierstrass function \( P_k\genfrac[]{0pt}{0}{\theta}{\phi}(z,\tau) \) and the twisted Eisenstein series \( G_n\genfrac[]{0pt}{0}{\theta}{\phi}(\tau) \), respectively.

In our case, since $\theta = e^{2\pi i s h_{(0)}}$, define
\[
N_s:= \mathbb{C}[\tilde{G}_{2k}(q) \mid k \geq 2] \otimes \mathbb{C}\left[\tilde{G}_k\genfrac[]{0pt}{0}{\theta}{\phi}(q) \mid k \geq 0\right].
\]

\noindent Let
\begin{align*}
R' := N_s \otimes \mathbb{C}\big[ &\tilde{P}_1\genfrac[]{0pt}{0}{\theta}{\phi}(z_i - z_j; q),\ 
\tilde{P}_2\genfrac[]{0pt}{0}{\theta}{\phi}(z_i - z_j; q), \\
&\tilde{\wp}_2(z_i - z_j; q),\ 
\tilde{\wp}_3(z_i - z_j; q) \big]_{i, j = 1,\dots,n,\ i \neq j}.
\end{align*}
\begin{lemma}
For $m \geq 1$, $n \geq 2$, and $i \neq j$, we have $\tilde{P}_m\genfrac[]{0pt}{0}{\theta}{\phi}(z_i - z_j, q),\; \tilde{\wp}_n(z_i - z_j, q) \in R'$.
\end{lemma}

\begin{proof}
By the expansion formulas (see \eqref{pqfunction4}),
\begin{align}
\begin{split}
    P_1\genfrac[]{0pt}{0}{\theta}{\phi}(z, \tau) &= \frac{1}{z} - G_1\genfrac[]{0pt}{0}{\theta}{\phi}(\tau) - G_2\genfrac[]{0pt}{0}{\theta}{\phi}(\tau)z + \cdots, \\
    P_2\genfrac[]{0pt}{0}{\theta}{\phi}(z, \tau) &= \frac{1}{z^2} + G_2\genfrac[]{0pt}{0}{\theta}{\phi}(\tau) + 2 G_3\genfrac[]{0pt}{0}{\theta}{\phi}(\tau)z + 3 G_4\genfrac[]{0pt}{0}{\theta}{\phi}(\tau)z^2 + \cdots,
\end{split}
\end{align}
one obtains the identity
\begin{align}\label{trr'}
\begin{split}
    &P_2\genfrac[]{0pt}{0}{\theta}{\phi}(z, \tau) 
    - P_1\genfrac[]{0pt}{0}{\theta}{\phi}(z, \tau)^2 
    - 2 G_1\genfrac[]{0pt}{0}{\theta}{\phi}(\tau) P_1\genfrac[]{0pt}{0}{\theta}{\phi}(z, \tau) \\
    &= G_2\genfrac[]{0pt}{0}{\theta}{\phi}(\tau) 
    + 3 G_1\genfrac[]{0pt}{0}{\theta}{\phi}(\tau)^2 + O(z).
\end{split}
\end{align}
The left-hand side of \eqref{trr'} is a bounded entire function in $z$ by Theorem~\ref{Mason}, and by Liouville’s theorem, the $O(z)$ term must vanish.  
This implies $P_m\genfrac[]{0pt}{0}{\theta}{\phi}(z, \tau)$ can be written as a polynomial in $P_1\genfrac[]{0pt}{0}{\theta}{\phi}(z, \tau)$ and elements of $N_s$ by induction using \eqref{trr'}.

Using the identity
\[
4 \wp_2(z, \tau)^2 = 4 \wp_2(z, \tau)^3 - 60 G_4(\tau) \wp_2(z, \tau) - 140 G_6(\tau),
\]
and induction, we conclude that $\wp_n(z, \tau)$ for $n \geq 2$ can be expressed as polynomials in $G_4(\tau)$, $G_6(\tau)$, $\wp_2(z, \tau)$, and $\wp_3(z, \tau)$.

Since $\tilde{P}_n\genfrac[]{0pt}{0}{\theta}{\phi}(z, q)$ and $\tilde{\wp}_n(z, q)$ are $q$-expansions of $P_n\genfrac[]{0pt}{0}{\theta}{\phi}(z, \tau)$ and $\wp_n(z, \tau)$ respectively, it follows that
\[
\tilde{P}_m\genfrac[]{0pt}{0}{\theta}{\phi}(z_i - z_j, q),\; \tilde{\wp}_n(z_i - z_j, q) \in R'
\]
for all $m \geq 1$, $n \geq 2$, and $i \neq j$.
\end{proof}

\subsection{Partial differential equations satisfied by genus-one \texorpdfstring{$n$}{Lg} point correlation functions}
In the rest of this paper, we assume $g(V \setminus C_2(V)) \subset V \setminus C_2(V)$, which implies that
\[
(R_V)^g \cong V^g / \left(C_2(V) \cap V^g\right).
\]
Let \( \mathcal{T}' = R' \otimes W_1 \otimes \cdots \otimes W_n \) for \( n \geq 2 \). Let \( \mathcal{J}' \) denote the \( R' \)-submodule of \( \mathcal{T}' \) generated by the following elements, where the indices \( i, j, l \) vary over \( \{1, \ldots, n\} \),
\begin{itemize}
    \item when $\theta = \phi = 1$ and $t = 0$:
   \begin{align} \label{11}
\begin{split}
    & u_1 \otimes \cdots \otimes u_{(-2)} u_i \otimes \cdots \otimes u_n \\
    &\quad + \sum_{k \in \Z_+} (2k + 1)\, \tilde{G}_{2k+2}(q)\,
    u_1 \otimes \cdots \otimes u_{(2k)} u_j \otimes \cdots \otimes u_n \\
    &\quad + \sum_{i \neq j} \sum_{m \in \mathbb{N}} (-1)^{m+1} (m + 1)\,
    \tilde{\wp}_{m+2}(z_i - z_j, q)\,
    u_1 \otimes \cdots \otimes u_{(m)} u_j \otimes \cdots \otimes u_n,
\end{split}
\end{align}

    \item when $\phi = 1$:
    \begin{align} \label{22}
    \sum_{i=1}^{n} u_1 \otimes \cdots \otimes u_{(0)}u_i \otimes \cdots \otimes u_n,
    \end{align}

    \item when $\phi = 1$ and $(-1)^t \theta \neq 1$:
   \begin{align} \label{33}
\begin{split}
    & u_1 \otimes \cdots \otimes u_{(-2)} u_i \otimes \cdots \otimes u_n \\
    &\quad + \sum_{k \geq 2} \left(
        (k - 1)\, \tilde{G}_{k} \genfrac[]{0pt}{0}{(-1)^t \theta}{\phi}(q)
        - \delta_{k,2} \cdot \frac{2\pi i}{1 - \theta^{-1}}
    \right) \cdot \\
    &\qquad \cdot u_1 \otimes \cdots \otimes u_{(k - 2)} u_i \otimes \cdots \otimes u_n \\
    &\quad - \sum_{i \neq j} \sum_{m \in \mathbb{N}} \left(
        (-1)^{m+1} (m + 1)\, \tilde{P}_{m + 2}
        \genfrac[]{0pt}{0}{(-1)^t \theta}{\phi}(z_i - z_j, q)
        - \delta_{0,m} \cdot \frac{2\pi i}{1 - \theta^{-1}}
    \right) \cdot \\
    &\qquad \cdot u_1 \otimes \cdots \otimes u_{(m)} u_j \otimes \cdots \otimes u_n,
\end{split}
\end{align}

\begin{align} \label{106}
\begin{split}
    & u_1 \otimes \cdots \otimes u_{(-1)} u_j \otimes \cdots \otimes u_n \\
    &\quad - \sum_{k \in \mathbb{Z}_+} \left(
        \tilde{G}_{k} \genfrac[]{0pt}{0}{(-1)^t \theta}{\phi}(q)
        - \delta_{k,1} \cdot \frac{2\pi i}{1 - \theta^{-1}}
    \right) \cdot \\
    &\qquad \cdot u_1 \otimes \cdots \otimes u_{(k-1)} u_j \otimes \cdots \otimes u_n \\
    &\quad - \sum_{i \neq j} \sum_{m \in \mathbb{N}} \left(
        (-1)^{m+1} \tilde{P}_{m+1}
        \genfrac[]{0pt}{0}{(-1)^t \theta}{\phi}(z_i - z_j, q)
        - \delta_{0,m} \cdot \frac{2\pi i}{1 - \theta^{-1}}
    \right) \cdot \\
    &\qquad \cdot u_1 \otimes \cdots \otimes u_{(m)} u_i \otimes \cdots \otimes u_n
\end{split}
\end{align}
\begin{align*}
\begin{split}
    &= u_1 \otimes \cdots \otimes u_{(-1)} u_l \otimes \cdots \otimes u_n \\
    &\quad - \sum_{k \in \mathbb{Z}_+} \left(
        \tilde{G}_{k} \genfrac[]{0pt}{0}{(-1)^t \theta}{\phi}(q)
        - \delta_{k,1} \cdot \frac{2\pi i}{1 - \theta^{-1}}
    \right) \cdot \\
    &\qquad \cdot u_1 \otimes \cdots \otimes u_{(k-1)} u_l \otimes \cdots \otimes u_n \\
    &\quad - \sum_{i \neq j} \sum_{m \in \mathbb{N}} \left(
        (-1)^{m+1} \tilde{P}_{m+1}
        \genfrac[]{0pt}{0}{(-1)^t \theta}{\phi}(z_i - z_l, q)
        - \delta_{0,m} \cdot \frac{2\pi i}{1 - \theta^{-1}}
    \right) \cdot \\
    &\qquad \cdot u_1 \otimes \cdots \otimes u_{(m)} u_l \otimes \cdots \otimes u_n,
\end{split}
\end{align*}

    \item when $\phi \neq 1$:
   \begin{align} \label{44}
\begin{split}
    & u_1 \otimes \cdots \otimes u_{(-2)} u_i \otimes \cdots \otimes u_n \\
    &\quad + \sum_{k \geq 2} (k - 1)\,
    \tilde{G}_{k} \genfrac[]{0pt}{0}{(-1)^t \theta}{\phi}(q) \cdot \\
    &\qquad \cdot u_1 \otimes \cdots \otimes u_{(k - 2)} u_i \otimes \cdots \otimes u_n \\
    &\quad - \sum_{i \neq j} \sum_{m \in \mathbb{N}} (-1)^{m + 1} (m + 1)\,
    \tilde{P}_{m + 2} \genfrac[]{0pt}{0}{(-1)^t \theta}{\phi}(z_i - z_j, q) \cdot \\
    &\qquad \cdot u_1 \otimes \cdots \otimes u_{(m)} u_j \otimes \cdots \otimes u_n,
\end{split}
\end{align}

    \begin{align} \label{108}
    \begin{split}
        & u_1 \otimes \cdots \otimes u_{(-1)}u_j \otimes \cdots \otimes u_n \\
        & - \sum_{k \in \mathbb{Z}_+} \tilde{G}_{k} \genfrac[]{0pt}{0}{(-1)^t \theta}{\phi}(q)
        u_1 \otimes \cdots \otimes u_{(k-1)}u_j \otimes \cdots \otimes u_n \\
        & - \sum_{i \neq j} \sum_{m \in \mathbb{N}} (-1)^{m+1} \tilde{P}_{m+1} \genfrac[]{0pt}{0}{(-1)^t \theta}{\phi}(z_i - z_j, q)
        u_1 \otimes \cdots \otimes u_{(m)}u_i \otimes \cdots \otimes u_n \\
        &= u_1 \otimes \cdots \otimes u_{(-1)}u_l \otimes \cdots \otimes u_n \\
        & - \sum_{k \in \mathbb{Z}_+} \tilde{G}_{k} \genfrac[]{0pt}{0}{(-1)^t \theta}{\phi}(q)
        u_1 \otimes \cdots \otimes u_{(k-1)}u_l \otimes \cdots \otimes u_n \\
        & - \sum_{i \neq j} \sum_{m \in \mathbb{N}} (-1)^{m+1} \tilde{P}_{m+1} \genfrac[]{0pt}{0}{(-1)^t \theta}{\phi}(z_i - z_l, q)
        u_1 \otimes \cdots \otimes u_{(m)}u_l \otimes \cdots \otimes u_n.
    \end{split}
    \end{align}
\end{itemize}
According to \cite{deligne1973schemas}, $M(T, T_1)$ is a Noetherian ring. Thus, $N_s(T,T_1)$ is also Noetherian.

We say that $\tilde{G}_{2k}(q)$ for $k \geq 1$ has modular weight $2k$, and the element $\tilde{\wp}_{m}(z_{i}-z_{j};q)$ for any $m \geq 2$ has modular weight $m$. Similarly, we define the modular weights of $\tilde{P}_{k}\genfrac[]{0pt}{0}{\theta}{\phi}(z_{i}-z_{j};q)$ for $k \geq 1$ and $\tilde{G}_{m}\genfrac[]{0pt}{0}{\theta}{\phi}(q)$ for $m \geq 0$ to be $k$ and $m$, respectively.

For $m \in \mathbb{Z}_+$, let $R_{m}'$ be the subspace of elements in $R'$ of modular weight $m$. It is clear that
\[
R' = \bigoplus_{m \in \mathbb{Z_+}} R_{m}'.
\]
 The gradings on $W_i$, for $i = 1, \dots, n$, induce a grading on $\T'$. We denote by $\T'_{(r)}$ the homogeneous subspace of $\T'$ of weight $r$.
We define a filtration on $\T'$ by letting
\[
H_{r}(\T') = \bigoplus_{i \leq r} \T'_{(i)}.
\]
This induces filtrations on $\J'$ and on $\T'/\J'$:
\[
H_r(\J') = \J' \cap H_r(\T'), \quad H_r(\T'/\J') = H_r(\T') / H_r(\J').
\]
The associated graded algebra of $\T'/\J'$ with respect to the filtration $H$ is
\[
\operatorname{gr}^H(\T'/\J') = \operatorname{gr}^H(\T') / \operatorname{gr}^H(\J') \cong \T' / \operatorname{gr}^{H}(\J'),
\]
where the last isomorphism is as vector spaces.
\begin{lemma}\label{finitgene}
The $R'$-module $\T'/\J'$ is finitely generated if
\begin{align}\label{confinitenesscon}
\dim\left( R_{W_1} \otimes_{R_V} \cdots \otimes_{R_V} R_{W_n} \big/ \{(R_V)^g, R_{W_1} \otimes_{R_V} \cdots \otimes_{R_V} R_{W_n} \} \right) < \infty.
\end{align}
\end{lemma}

\begin{proof}
Let $\gamma = \sum_{i=1}^{n} \mathrm{wt}(u_i) + \mathrm{wt}(u) + 1$. Since
\begin{align*}
    & u_1 \otimes \cdots \otimes u_{(-2)}u_i \otimes \cdots \otimes u_n \in H_\gamma(T'),\\
    & u_1 \otimes \cdots \otimes u_{(k)}u_i \otimes \cdots \otimes u_n \in H_{\gamma-1}(T') \quad \text{for } k\in \mathbb{N},
\end{align*}
by (\ref{11}), (\ref{33}), and (\ref{44}), and the definition of $\mathrm{gr}^{H}(\J')$, we have
\begin{align}\label{inclusiona}
    1 \otimes u_1 \otimes \cdots \otimes u_{j-1} \otimes u_{(-2)}u_j \otimes u_{j+1} \otimes \cdots \otimes u_n \in \mathrm{gr}^{H}(\J')
\end{align}
for $1 \leq j \leq n$.

Similarly, by (\ref{22}), (\ref{106}), and (\ref{108}), we obtain
\begin{align}\label{tensorcrucial}
    u_1 \otimes \cdots \otimes u_{(-1)}u_j\otimes \cdots \otimes u_n 
    - u_1 \otimes \cdots \otimes u_{(-1)}u_l \otimes \cdots \otimes u_n 
    \in \mathrm{gr}^{H}(\J')
\end{align}
for $1 \leq j, l \leq n$ and $u \in V$, and
\begin{align}\label{87}
    \sum_{i=1}^{n} u_1 \otimes \cdots \otimes u_{(0)}u_i \otimes \cdots \otimes u_n \in \mathrm{gr}^{H}(\J')
\end{align}
for $u \in V^g$.

We now define a linear map
\[
\Theta': R_{W_1} \otimes_{R_V} \cdots \otimes_{R_V} R_{W_n} \rightarrow \T' / \mathrm{gr}^{H}(\J')
\]
by
\[
\Theta'\left( \overline{a_1} \otimes \cdots \otimes \overline{a_n} \right) = \overline{a_1 \otimes \cdots \otimes a_n},
\]
where $a_i \in W_i$, $i = 1, \dots, n$. This map is well-defined: Suppose $\overline{a_1} \otimes \cdots \otimes \overline{a_n} = \overline{a_1'} \otimes \cdots \otimes \overline{a_n}$, where $a_1 - a_1' \in C_2(W_1)$. Then by (\ref{inclusiona}), we have
\[
\Theta'(a_1 \otimes a_2 \otimes \cdots \otimes a_n) = \Theta'(a_1' \otimes a_2 \otimes \cdots \otimes a_n).
\]
Repeating this argument inductively for all $i = 1, \dots, n$, we conclude that
\[
\Theta'\left( \overline{a_1} \otimes \cdots \otimes \overline{a_n} \right) = \Theta' \left( \overline{a_1'} \otimes \cdots \otimes \overline{a_n'} \right)
\quad \text{if } a_i - a_i' \in C_2(W_i).
\]
Moreover, identities (\ref{tensorcrucial}) and (\ref{87}) imply that $\Theta'$ descends to a map over the quotient
\[
R_{W_1} \otimes_{R_V} \cdots \otimes_{R_V} R_{W_n} \big/ \{ (R_V)^g, R_{W_1} \otimes_{R_V} \cdots \otimes_{R_V} R_{W_n} \}.
\]
It is also clear that $\Theta'$ is surjective.

Therefore, $\T'/\mathrm{gr}^{H}(\J')$ is finitely generated over $R'$ if the quotient on the right-hand side of (\ref{confinitenesscon}) is finite-dimensional. The lemma follows.
\end{proof}

For any $u_i \in W_i$, $i = 1, \dots, n$, define the element
\begin{align}
\begin{split}
    F_{\mathcal{Y}_{1}, \dots, \mathcal{Y}_{n}}(u_1, \dots, u_n; z_1, \dots, z_n; h; q)
    :=\; \\
    \operatorname{tr}_{\tilde{W}_n} \,
    \mathcal{Y}_1(\mathcal{U}(q_{z_1}) u_1, q_{z_1}) \cdots 
    \mathcal{Y}_n(\mathcal{U}(q_{z_n}) u_n, q_{z_n}) 
    e^{2\pi i s h_{(0)}} q^{L_{(0)} - \frac{c}{24}}
\end{split}
\end{align}
in $\mathbb{G}_{|z_1| > \cdots > |z_n|} \otimes \mathbb{C}[[q_s^{\frac{1}{T_1}}, q_s^{-\frac{1}{T_1}}]]((q^{\frac{1}{T}}))$.

Define a linear map
\[
\Upsilon': \T' \rightarrow \mathbb{G}_{|z_1| > \cdots > |z_n|} \otimes \mathbb{C}[[q_s^{\frac{1}{T_1}}, q_s^{-\frac{1}{T_1}}]]((q^{\frac{1}{T}}))
\]
by
\[
u_1 \otimes \cdots \otimes u_n \mapsto F_{\mathcal{Y}_{1}, \dots, \mathcal{Y}_{n}}(u_1, \dots, u_n; z_1, \dots, z_n; h; q).
\]

\noindent
By (\ref{iffq210}), $\J'$ is contained in the kernel of $\Upsilon'$; thus, one has a well-defined map
\[
\overline{\Upsilon'}: \T' / \J' \rightarrow \mathbb{G}_{|z_1| > \cdots > |z_n|} \otimes \mathbb{C}[[q_s^{\frac{1}{T_1}}, q_s^{-\frac{1}{T_1}}]]((q^{\frac{1}{T}})).
\]
Moreover, by using (\ref{gotcf123}), one obtains the following {\em $L_{(-1)}$-derivative property} for $1\leq j\leq n$:
\begin{align}
\begin{split}
    \frac{\partial}{\partial z_j}
    F_{\mathcal{Y}_{1}, \dots, \mathcal{Y}_{n}}(u_1, \dots, u_n; z_1, \dots, z_n; h; q)
    =\; & \\
    F_{\mathcal{Y}_{1}, \dots, \mathcal{Y}_{n}}(
    u_1, \dots, u_{j-1}, L_{(-1)} u_j, u_{j+1}, \dots, u_n;
    z_1, \dots, z_n; h; q).
\end{split}
\end{align}

By taking $u = \omega$ and $l = 1$ in (\ref{iffq212}), and using (\ref{virasoro}) together with the $L_{(-1)}$-derivative property, one obtains the following identity (see \cite{huang2005differential}).

\begin{lemma}\label{operator}
We have 
\begin{align}
\begin{split}
    & \left(
        (2\pi i)^2 q \frac{\partial}{\partial q}
        + \tilde{G}_2(q) \sum_{i=1}^{n} \operatorname{wt} w_i
        + \tilde{G}_2(q) \sum_{i=1}^{n} z_i \frac{\partial}{\partial z_i}
        - \sum_{i \neq j} \tilde{\wp}_1(z_i - z_j, q) \frac{\partial}{\partial z_i}
    \right) \\
    &\quad \cdot 
    F_{\mathcal{Y}_1, \dots, \mathcal{Y}_n}(u_1, \dots, u_n; z_1, \dots, z_n; h; q) \\
    &= 
    F_{\mathcal{Y}_1, \dots, \mathcal{Y}_n}(
    u_1, \dots, u_{j-1}, L_{(-2)} u_j, u_{j+1}, \dots, u_n;
    z_1, \dots, z_n; h; q) \\
    &\quad - \sum_{k \in \mathbb{Z}_+} \tilde{G}_{2k+2}(q) \cdot 
    F_{\mathcal{Y}_1, \dots, \mathcal{Y}_n}(
    u_1, \dots, u_{j-1}, L_{(2k)} u_j, u_{j+1}, \dots, u_n; 
    z_1, \dots, z_n; h; q) \\
    &\quad + \sum_{i \neq j} \sum_{m \in \mathbb{Z}_+} (-1)^m \tilde{\wp}_{m+1}(z_i - z_j, q) \cdot \\
    &\qquad \cdot F_{\mathcal{Y}_1, \dots, \mathcal{Y}_n}(
    u_1, \dots, u_{i-1}, L_{(m-1)} u_i, u_{i+1}, \dots, u_n;
    z_1, \dots, z_n; h; q).
\end{split}
\end{align}
\end{lemma}

\noindent Let $h$ be a primary vector with respect to the conformal vector $\omega$ and fixed by $g$. Let $u = h$ and $l = 1$ in (\ref{iffq212}). 
\begin{lemma}\label{6.13}
We have
\begin{align}
\begin{split}
& q_s \frac{\partial}{\partial q_s}
F_{\mathcal{Y}_{1}, \dots, \mathcal{Y}_{n}}(
u_1, \dots, u_n; z_1, \dots, z_n; h; q) \\
&= F_{\mathcal{Y}_{1}, \dots, \mathcal{Y}_{n}}(
u_1, \dots, u_{j-1}, h_{(-1)} u_j, u_{j+1}, \dots, u_n;
z_1, \dots, z_n; h; q) \\
&\quad - \sum_{k \in \mathbb{Z}_+} (2k + 1)\, \tilde{G}_{2k+2}(q) \cdot \\
&\qquad \cdot F_{\mathcal{Y}_{1}, \dots, \mathcal{Y}_{n}}(
u_1, \dots, u_{j-1}, h_{(2k+1)} u_j, u_{j+1}, \dots, u_n;
z_1, \dots, z_n; h; q) \\
&\quad - \sum_{i \neq j} \sum_{m \in \mathbb{N}} (m + 1)\, \tilde{\wp}_{m+1}(z_i - z_j, q) \cdot \\
&\qquad \cdot F_{\mathcal{Y}_{1}, \dots, \mathcal{Y}_{n}}(
u_1, \dots, u_{i-1}, h_{(m)} u_i, u_{i+1}, \dots, u_n;
z_1, \dots, z_n; h; q) \\
&\quad - \tilde{G}_2(q) \sum_{i=1}^{n}
F_{\mathcal{Y}_{1}, \dots, \mathcal{Y}_{n}}(
u_1, \dots, u_{i-1}, (h_{(1)} + h_{(0)} z_i) u_i, u_{i+1}, \dots, u_n;
z_1, \dots, z_n; h; q).
\end{split}
\end{align}
\end{lemma}

One defines, for any $\alpha \in \mathbb{C}$,
\begin{align}
\mathcal{O}_{j}^{h}(\alpha) 
&= (2\pi i)^2 q \frac{\partial}{\partial q}
+ \tilde{G}_{2}(q)\alpha 
+ \tilde{G}_{2}(q) \sum_{i=1}^{n} z_{i} \frac{\partial}{\partial z_{i}} 
- \sum_{i \neq j} \tilde{\wp}_{1}(z_{i} - z_{j}; q) \frac{\partial}{\partial z_{i}},
\end{align}
for $j = 1, \dots, n$, and
\begin{align}
\prod_{j=1}^{m} \mathcal{O}^{h}(\alpha_{j}) 
:= \mathcal{O}^{h}(\alpha_{1}) \cdots \mathcal{O}^{h}(\alpha_{m}).
\end{align}

\begin{theorem}\label{maint1}
 If $\text{dim}(R_{W_1}\otimes_{R_V} \cdots \otimes_{R_V} R_{W_n}/\{(R_V)^g, R_{W_1}\otimes_{R_V} \cdots \otimes_{R_V} R_{W_n}\}) < \infty$ and $\mathcal{Y}_{i}$, $i = 1, \dots, n$, are intertwining operators of type $\displaystyle\binom{\tilde{W}_{i-1}}{W_{i}\;\; \tilde{W}_{i}}$, where $\tilde{W}_0 = \tilde{W}_n$, and $h$ is a Cartan element of $V$ such that $\tilde{W}_n$ is $h$-stable, then

\begin{itemize}
    \item For any homogeneous elements $u_{i} \in W_{i}$, $i = 1, \dots, n$, there exist
\begin{align*}
a_{p,i}(z_1, \dots, z_n; q; h) &\in R'_{p}, \\
b_{p,i}(z_1, \dots, z_n; q; h) &\in R'_{2p}, \\
c_{p,i}(z_1, \dots, z_n; q; h) &\in R'_{l} \otimes (\mathbb{C}[z_1, \dots, z_n])_m \quad (l + m = p).
\end{align*}    for $p = 1, \dots, m$ and $i = 1, \dots, n$, where $(\mathbb{C}[z_1, \dots, z_n])_m$ denotes the space of degree-$m$ polynomials, such that $F_{\mathcal{Y}_{1}, \dots, \mathcal{Y}_{n}}(u_1, \dots, u_n; z_1, \dots, z_n; q; h)$ satisfies the following system of differential equations:

\begin{align}\label{equatio1}
    \frac{\partial^{m} f}{\partial z_i^{m}} + \sum_{p=1}^{m} a_{p,i}(z_1, \dots, z_n; q; h) \frac{\partial^{m-p} f}{\partial z_i^{m-p}} = 0,
\end{align}

\begin{align}\label{91'}
\begin{split}
    & \prod_{k=1}^{m} 
    \mathcal{O}^{h}_i\left( \sum_{i=1}^{n} \operatorname{wt} u_i + 2(m - k) \right) f \\
    &\quad + \sum_{p=1}^{m} b_{p,i}(z_1, \dots, z_n; q; h) \cdot 
    \prod_{k=1}^{m-p} 
    \mathcal{O}^{h}_i\left( \sum_{i=1}^{n} \operatorname{wt} u_i + 2(m - p - k) \right) f \\
    &= 0,
\end{split}
\end{align}

\begin{align}\label{98}
\begin{split}
    \left(q_s \frac{\partial}{\partial q_s}\right)^m f
    + \sum_{p=1}^{m} c_{p,i}(z_1, \dots, z_n; q; h)
    \left(q_s \frac{\partial}{\partial q_s}\right)^{m-p} f = 0,
\end{split}
\end{align}
for $i = 1, \dots, n$, in the region $1 > |q_{z_1}| > \cdots > |q_{z_n}| > |q| > 0$, $0 < |q_s| < 1$.

\item In the region
\[
\left\{(z_1, \dots, z_n, \tau, s) \,\middle|\,
1 > |q_{z_1}| > \cdots > |q_{z_n}| > |q_{\tau}| > 0,\;\; 0 < |q_s| < 1 \right\},
\]
$F_{\mathcal{Y}_{1}, \dots, \mathcal{Y}_{n}}(u_1, \dots, u_n; z_1, \dots, z_n; q_\tau; h)$ is absolutely convergent and can be analytically extended to a multivalued analytic function in the region where $\tau \in \mathbb{H}$, $|q_s| < 1$, and $z_i \neq z_j + k \tau + l$ for $i \neq j$, $k, l \in \mathbb{Z}$. We denote this extension by
\[
\overline{F}_{\mathcal{Y}_{1}, \dots, \mathcal{Y}_{n}}(u_1, \dots, u_n; z_1, \dots, z_n; h; q_\tau).
\]
\end{itemize}
\end{theorem}
\begin{proof}
We follow a similar argument as in \cite[Theorem 3.9]{huang2005differential}. First, let $\Pi_{i}$ ($i = 1, \dots, n$) be the $R'$-submodules of $\T'/\J'$ generated by
\begin{align*}
    \left[1 \otimes u_{1} \otimes \cdots \otimes u_{i-1} \otimes L^{k}_{(-1)} u_{i} \otimes u_{i+1} \otimes \cdots \otimes u_{n} \right]
\end{align*}
for $k \in \mathbb{N}$. Since $R'$ is Noetherian and $\T'/\J'$ is a finitely generated $R'$-module, each submodule $\Pi_i$ is also finitely generated. Therefore, there exists a sufficiently large integer $m$ such that for each $i = 1, \dots, n$, we have
\begin{align}\label{identity1}
\begin{split}
    &\left[
    1 \otimes u_1 \otimes \cdots \otimes u_{i-1} \otimes
    L^m_{(-1)} u_i \otimes u_{i+1} \otimes \cdots \otimes u_n
    \right] \\
    &\quad + \sum_{p=1}^{m} a_{p,i}(z_1, \dots, z_n; q) \cdot
    \left[
    1 \otimes u_1 \otimes \cdots \otimes u_{i-1} \otimes
    L^{m-p}_{(-1)} u_i \otimes u_{i+1} \otimes \cdots \otimes u_n
    \right] \\
    &= 0,
\end{split}
\end{align}
where each $a_{p,i}(z_{1}, \dots, z_{n}; q)$ can be chosen to have modular weight $p$. Applying $\overline{\Upsilon'}$ to both sides of (\ref{identity1}) and using the $L(-1)$-derivative property, we find that $F_{\mathcal{Y}_{1}, \dots, \mathcal{Y}_{n}}(u_1, \dots, u_n; z_1, \dots, z_n; q)$ satisfies (\ref{equatio1}).

Next, define $\mathcal{Q}_i : \T' \rightarrow \T'$ by
\begin{align*}
\mathcal{Q}_i(1 \otimes \cdots \otimes u_n)
&= 1 \otimes u_1 \otimes \cdots \otimes u_{i-1} \otimes 
L(-2) u_i \otimes u_{i+1} \otimes \cdots \otimes u_n \\
&\quad - \sum_{k \in \mathbb{Z}_+} \tilde{G}_{2k+2}(q) \cdot \\
&\qquad \cdot 1 \otimes u_1 \otimes \cdots \otimes u_{i-1} \otimes 
L(2k) u_i \otimes u_{i+1} \otimes \cdots \otimes u_n \\
&\quad + \sum_{j \neq i} \sum_{m \in \mathbb{Z}_+} (-1)^m\, \tilde{\wp}_{m+1}(z_j - z_i; q) \cdot \\
&\qquad \cdot 1 \otimes u_1 \otimes \cdots \otimes u_{i-1} \otimes 
L(m{-}1) u_i \otimes u_{i+1} \otimes \cdots \otimes u_n.
\end{align*}
For fixed $u_i \in W_i$, $i = 1, \dots, n$, define $\Lambda_i$ as the $R'$-submodule of $\T'/\J'$ generated by $[\mathcal{Q}_i^k(1 \otimes u_1 \otimes \cdots \otimes u_n)]$ for $k \in \mathbb{N}$. Since $\T'/\J'$ is a finitely generated $R'$-module and $R'$ is Noetherian, $\Lambda_i$ is also finitely generated. Hence there exist $b_{p,i}(z_1, \dots, z_n; q) \in R'$ for $p = 1, \dots, m$ and $i = 1, \dots, n$ such that in $\Lambda_i$:
\begin{align}\label{identity2}
    [\mathcal{Q}_{i}^{m}(1 \otimes u_1 \otimes \cdots \otimes u_n)]
    + \sum_{p=1}^{m} b_{p,i}(z_1, \dots, z_n; q)
    [\mathcal{Q}_{i}^{m-p}(1 \otimes u_1 \otimes \cdots \otimes u_n)] = 0,
\end{align}
where each $b_{p,i}$ has modular weight $2p$. Applying $\overline{\Upsilon'}$ to both sides of (\ref{identity2}) and using Lemma \ref{operator}, we conclude that $F_{\mathcal{Y}_1, \dots, \mathcal{Y}_n}(u_1, \dots, u_n; z_1, \dots, z_n; h; q)$ satisfies (\ref{91'}).

Note that $R' \otimes \mathbb{C}[z_1, \dots, z_n]$ is Noetherian. Then (\ref{98}) follows similarly from Lemma \ref{operator} and Lemma \ref{6.13}.

For the second part of the theorem, we note that the coefficients of
\begin{align}\label{91}
F_{\mathcal{Y}_{1}, \dots, \mathcal{Y}_{n}}(u_1, \dots, u_n; z_1, \dots, z_n; h; q_\tau)
\end{align}
as a $(q_s, q_\tau)$-series are absolutely convergent when $1 > |q_{z_1}| > \cdots > |q_{z_n}| > 0$. So for fixed $z_1, \dots, z_n$ in this region, (\ref{91}) satisfies equations (\ref{91'}) and (\ref{98}) with respect to $q_\tau$ and $q_s$. Since these differential equations have regular singularities at $q_\tau = 0$ and $q_s = 0$, and their coefficients are analytic in $z_1, \dots, z_n$, the sum of (\ref{91}) as a power series in $q_s$ and $q_\tau$ is also analytic in $z_1, \dots, z_n$.

In particular, (\ref{91}) defines an expansion of a multivariable analytic function in the region $1 > |q_{z_1}| > \cdots > |q_{z_n}| > |q_\tau| > 0$, $0 < |q_s| < 1$, and is absolutely convergent as a multiple sum series. Since the coefficients of (\ref{91'}) are analytic in $z_1, \dots, z_n, \tau, s$, and only have singularities at $z_i = z_j + k\tau + l$ for $i \neq j$, $k, l \in \mathbb{Z}$, (\ref{91}) can be analytically extended to a multivalued analytic function in the region described in the theorem.
\end{proof}

\subsection{Genus-one $1$-point functions}
Following the same notation and assumptions as in the previous section, we now consider the 1-point function. Let $W$ be a weak $V$-module, and  $\tilde{W}$ be a weak $g$-twisted $V$-module. Let $\mathcal{Y}$ be the intertwining operator of type $\binom{\tilde{W}}{W\; \tilde{W}}$.

 For this case, we have the following. 
\begin{lemma}For any $u_1\in W$, we have
\[
\frac{\partial}{\partial z}F_{\mathcal{Y}}(u_1; z; q; h) = 0.
\]
\end{lemma}

\begin{proof}
Since
\[
\frac{\partial}{\partial z}(2\pi i e^{2\pi i z})^{L_{(0)}} = 2\pi i L_{(0)} (2\pi i e^{2\pi i z})^{L_{(0)}},
\quad
\frac{\partial}{\partial z} \mathcal{Y}(u_1, q_z) = 2\pi i q_z L_{(-1)} \mathcal{Y}(u_1, q_z),
\]
we have
\begin{align*}
\frac{\partial}{\partial z} F_{\mathcal{Y}}(u_1; z; q; h)
&= \frac{\partial}{\partial z} \operatorname{tr}_{W} \mathcal{Y}((2\pi i q_z)^{L_{(0)}} e^{-L_{+}(A)} u_1, q_z) e^{2\pi i s h_{(0)}} q^{L_{(0)} - \frac{c}{24}} \\
&= \operatorname{tr}_{W} \mathcal{Y}((2\pi i L_{(0)} + 2\pi i q_z L_{(-1)}) \mathcal{U}(q_z) u_1, q_z) e^{2\pi i s h_{(0)}} q^{L_{(0)} - \frac{c}{24}} \\
&= \operatorname{tr}_{W} [L_{(0)}, \mathcal{Y}(\mathcal{U}(q_z) u_1, q_z)] e^{2\pi i s h_{(0)}} q^{L_{(0)} - \frac{c}{24}} = 0.
\end{align*}
\end{proof}

\noindent
Recall the definition of the {\rm Serre derivation} $\partial_{\alpha}$ and the {\rm iterate Serre derivation} $\partial^m$ (\ref{serrederi}). For 1-point functions and $\alpha \in \mathbb{C}$, we have
\[
\frac{\mathcal{O}^{h}(\alpha)}{(2\pi i)^2} = \partial_\alpha,
\quad
\prod_{k=1}^{m} \frac{\mathcal{O}^{h}(\alpha + 2(m - k))}{(2\pi i)^2} = \partial^m.
\]
Moreover, in this case, $R'$ is simply the ring $N_s$.

We now describe the structure of the ideal $\J'$ for the genus-one 1-point function case. The generating relations for $\J'$ depend on the twisting parameters $\theta$ and $\phi$ as follows: for arbitrary $u\in V$, and $u_1\in W$,  
\begin{align}\label{genus1Jrelations}
\begin{cases}
u_{(-2)} u_1 + \sum_{k\in \Z_+} (2k+1)\tilde{G}_{2k+2}(q) u_{(2k)} u_1 & \text{if } \theta = \phi = 1,\, t = 0, \\
u_{(0)} u_1 & \text{if } \phi = 1, \\
u_{(-2)} u_1 + \sum_{k \geq 2} \left(
(k - 1)\tilde{G}_{k}\genfrac[]{0pt}{0}{(-1)^{t}\theta}{\phi}(q)
\right. \\
\left.\hspace{8em} -\, \delta_{k,2} \frac{2\pi i}{1 - \theta^{-1}}\right)
u_{(k-2)} u_1 & \text{if } \phi = 1,\, (-1)^t \theta \neq 1, \\
u_{(-2)} u_1 + \sum_{k \geq 2} (k - 1) \tilde{G}_k\genfrac[]{0pt}{0}{(-1)^t \theta}{\phi}(q) u_{(k-2)} u_1 & \text{if } \phi \neq 1.
\end{cases}
\end{align}
Using the same strategy as in Theorem \ref{maint1}, we can now derive the genus-one differential equation for the 1-point function.
\begin{theorem}\label{1pointfun}
Let $V$ be a vertex superalgebra, $\tilde{W}$ be its $h$-stable weak $g$-twisted module, and $W$ be its weak module 
satisfying the condition
\begin{align}\label{6.15}
\operatorname{dim}\left(R_W/\{(R_V)^g, R_W\}\right) < \infty.
\end{align}
For any homogeneous element $u_1 \in W$, there exist
\[
b_p(z; q; h) \in (N_s)_{2p}, \quad
c_p(z; q; h) \in (N_s)_l \otimes (\mathbb{C}[z])_m \quad (l + m = p),
\]
for $p = 1, \dots, m$, such that $F_{\mathcal{Y}}(u_1; z; q; h)$ satisfies the following differential equations:
\begin{align}\label{diffequ1}
\partial^m F_{\mathcal{Y}}(u_1; z; q; h) + \sum_{p=1}^{m} b_p(z; q; h) \, \partial^{m-p} F_{\mathcal{Y}}(u_1; z; q; h) = 0,
\end{align}
\begin{align}\label{diffequi}
\left(q_s \frac{\partial}{\partial q_s}\right)^{m} F_{\mathcal{Y}}(u_1; z; q; h)
+ \sum_{p=1}^{m} c_p(z; q; h) \left(q_s \frac{\partial}{\partial q_s}\right)^{m-p} F_{\mathcal{Y}}(u_1; z; q; h) = 0,
\end{align}
in the region $1 > |q_z| > |q| > 0$, $0 < |q_s| < 1$.

In particular, equation \eqref{diffequ1} can be rewritten as
\begin{align}\label{diffequ12}
\left(q^{\frac{1}{T}} \frac{d}{d q^{\frac{1}{T}}} \right)^{m} F_{\mathcal{Y}}(u_1; z; q; h)
+ \sum_{i=0}^{m-1} r_i(\tau) \left(q^{\frac{1}{T}} \frac{d}{d q^{\frac{1}{T}}} \right)^{i} F_{\mathcal{Y}}(u_1; z; q; h) = 0,
\end{align}
where each $r_i(\tau)$ belongs to the ring generated by $G_2(\tau)$ and $N_s$, and $q^{\frac{1}{T}} = e^{2\pi i \tau/T}$.
\end{theorem}

One calls \eqref{diffequ1} the \emph{twisted modular linear differential equation} (twisted MLDE). When $W = V$, $\mathcal{Y}$ is the vertex operator of $V$ on a $g$-twisted module $\tilde{W}$, $u_1 = \mathbf{1}$, 
$e^{2\pi i s h_{(0)}} = e^{\pi i \sigma}$, where $\sigma(v) = 2v$ if $v \in W_{\overline{0}}$ and $\sigma(v) = v$ if $v \in W_{\overline{1}}$, then $F_{\mathcal{Y}}(u_1; z; q_\tau; h)$ is the supercharacter of $\tilde{W}$.

\begin{corollary}\label{twistedmlde}
If $V$ is a quasi-lisse vertex superalgebra, then the supercharacter of its simple $g$-twisted module satisfies the twisted modular linear differential equation.
\end{corollary}

\begin{proof}
According to Corollary \ref{leaves'}, condition \eqref{6.15} is satisfied by any quasi-lisse vertex superalgebra.
\end{proof}

\begin{remark}
In fact, we can further weaken our condition to:
\[
\operatorname{dim} \frac{R_W}{\{ (R_V)^{\langle g \rangle}, R_W \} + (R_V)^+ \cdot R_W} < \infty,
\]
where $R_V^+$ is the complement of $(R_V)^g$ in $R_V = (R_V)^g \oplus R_V^+$. This condition for $V$ (when $W = V$) to satisfy the twisted MLDE is strictly weaker than the relative cofiniteness condition in \cite{arakawa2019modularity}.
\end{remark}

\begin{example}[{\cite{beem2018vertex}}]
For $\mathcal{N} = 4$ super Yang–Mills theory with gauge algebra $\mathfrak{sl}_2$, it corresponds to the vertex operator superalgebra $L_{-9}^{N=4}$ (see Example \ref{N=2}), whose the associated variety is the nilpotent cone $\mathcal{N}$ of $\mathfrak{sl}_2$. It is quasi-lisse. Its supercharacter, $\operatorname{sch}(L_{-9}^{N=4})$, satisfies the following twisted MLDE:
\[
\left( \partial^2 - 2\frac{G_2\genfrac[]{0pt}{0}{-1}{1}(\tau)}{(2\pi i)^2} \partial - 18\frac{G_4(\tau)}{(2\pi i)^4} + 18\frac{G_4\genfrac[]{0pt}{0}{-1}{1}(\tau)}{(2\pi i)^4} \right) f(\tau) = 0.
\]
\end{example}

\begin{example}
Let $\mathfrak{sl}_2$ be spanned by $e, f, h$ with relations $[h, e] = 2e$, $[h, f] = -2f$, $[e, f] = h$. For $s \in \mathbb{Q}$, define $\mathbf{h}^s = -s \frac{h}{2}$. Define Li's operator \cite{li1996local}
\[
\Delta_s(z) = z^{\mathbf{h}^s_{(0)}} \exp\left( \sum_{n = 1}^\infty \frac{\mathbf{h}^s_{(n)}}{-n} (-z)^{-n} \right).
\]
Let $\sigma_{\mathbf{h}^s} := \exp(2\pi i \mathbf{h}^s_{(0)})$ be an inner automorphism of finite order. The following result is due to Haisheng Li \cite[Proposition 5.4]{li1996local}.

\begin{proposition}
For any $g$-twisted $L_k({\mathfrak{sl}_2})$-module $M$, define
\[
(\pi_s(M), Y_M^s(\cdot, z)) := 
(M, Y_M(\Delta_s(z)\cdot, z)).
\]
Then this is a weak $g \sigma_{\mathbf{h}^s}$-twisted module.
\end{proposition}

The character of any simple $\sigma_{\mathbf{h}^s}$-twisted module of $L_k({\mathfrak{sl}_2})$ at admissible level satisfies the twisted MLDE. Let us consider the case $s = -\frac{1}{2}$ and $k = -\frac{4}{3}$. As shown in Example \ref{twistedzhualgebra}, there are three irreducible twisted modules. By calculation, we have two independent $q$-series of characters among them, and they satisfy the following twisted MLDE:
\[
\left( \partial^2 - \frac{1}{96} \Theta_{1,1}(\tau) \right) f(\tau) = 0,
\]
where $\Theta_{1,1} = 2 \vartheta_{00}(\tau) \vartheta_{10}(\tau)$, and
\[
\vartheta_{00}(\tau) := \sum_{n \in \mathbb{Z}} q^{\frac{n^2}{2}}, \quad
\vartheta_{10}(\tau) := \sum_{n \in \mathbb{Z}} q^{\frac{1}{2} (n + \frac{1}{2})^2}.
\]

\noindent Further examples of twisted MLDEs satisfied by characters of modules over quasi-lisse vertex algebras are in \cite{LLY2023spectral}. See also \cite{zheng2022surface} for more examples arising from Class $\mathcal{S}$ theories of type $A_1$.
\end{example}

\subsection{Example of affine vertex operator algebras} 
This section considers the highest weight modules and their contragredient modules for $V = L_k(\mathfrak{sl}_2)$ at an admissible level. It was shown in \cite{arakawa2015associated} that $V$ is quasi-lisse. As always, let $\{e, f, h\}$ be the standard generators of $\mathfrak{sl}_2$, and note that $R_V = S(\mathfrak{sl}_2)/I$, where $I$ is the image of the maximal submodule of $V$. We shall show that these modules satisfy the convergence property discussed above.

\begin{proposition}\label{finitenessex}
Let $W_i$ for $i = 1, \dots, m$ ($m\in \Z_+$), be highest weight $L_k(\mathfrak{sl}_2)$-modules or their contragredient modules. Then
\begin{align}\label{Fcn1}
   R_{W_1} \otimes_{R_V} \cdots \otimes_{R_V} R_{W_m} \big/ \left\{ R_V, R_{W_1} \otimes_{R_V} \cdots \otimes_{R_V} R_{W_m} \right\}
\end{align}
is finite-dimensional. 
\end{proposition}
\begin{proof}
We prove the case where $W_i$, $i=1,\ldots,m$ are highest weight modules that are infinitely strongly generated over $V$. The remaining cases can be proved similarly.

In the following, define
\[
u^n =
\begin{cases}
1 & \text{if } n=0, \\
0 & \text{if } n<0,
\end{cases}
\]
where $u = e, f,$ or $h$.

Let $w_i$, $i=1,\ldots,m$ be the highest weight vectors in $W_i$, and set $v := w_1 \otimes \cdots \otimes w_m$. 

First, for $n \in \mathbb{Z}_+$, one has
\[
\{f, f^{n-1} h v\} = 2f^n v + f^{n-1} h(f.v),
\]
where $f.v$ denotes $\{f, v\}$.

Using the relations
\begin{align}
& \{e, h f^{n-1}(f^2.v)\} = -2e f^{n-1}(f^2.v) + h^2 f^{n-2}(f^2.v) + (2k-2)h f^{n-1}(f.v), \label{for1} \\
& \{f, e f^{n-1}(f.v)\} = -h f^{n-1}(f.v) + e f^{n-1}(f^2.v), \label{for2} \\
& \{f, h^2 f^{n-2}(f.v)\} = 2h f^{n-1}(f.v) + h^2 f^{n-2}(f^2.v), \label{for3}
\end{align}
we obtain
\[
-2(\ref{for2}) + (\ref{for3}) - (\ref{for1}) = (6 - 2k) h f^{n-1}(f.v) \in \{R_V, R_{W_1} \otimes \cdots \otimes R_{W_m}\}.
\]
Moreover,
\[
\{h, h f^{n-1}(f.v)\} = (-2n + k) h f^{n-1}(f.v),
\]
so in particular,
\[
h f^{n-1}(f.v) \in \{R_V, R_{W_1} \otimes \cdots \otimes R_{W_m}\}.
\]

Now, for arbitrary $s \in \Z_+ $ and $l,t \in \mathbb{N}$, we have
\begin{align}
\{f, e^s f^l h^t v\} &= -s e^{s-1} f^l h^{t+1} v + 2t e^s f^{l+1} h^{t-1} v + e^s f^l h^t (f.v), \label{for4}
\end{align}
and a sequence of recursive relations
\begin{align}\label{for5}
\begin{split}
\{e, e^{s-1} f^{l+1} h^t v\} &= (l+1) e^{s-1} f^l h^{t+1} v - 2t e^s f^{l+1} h^{t-1} v, \\
\{e, e^s f^{l+2} h^{t-2} v\} &= (l+2) e^s f^{l+1} h^{t-1} v - 2(t-2) e^{s+1} f^{l+2} h^{t-3} v, \\
&\ \vdots \\
\{e, e^{s+n-1} f^{l+n+1} h^{t-2n} v\}
&= (l+n+1) e^{s+n-1} f^{l+n} h^{t-2n+1} v \\
&\quad - 2(t-2n) e^{s+n} f^{l+n+1} h^{t-2n-1} v.
\end{split}
\end{align}

If $t = 2n$, then from (\ref{for5}) we get
\[
e^{s-1} f^l h^{t+1} v,\ e^s f^{l+1} h^{t-1} v \in \{R_V, R_{W_1} \otimes \cdots \otimes R_{W_m}\},
\]
which further implies
\begin{align}\label{for6}
e^s f^l h^t (f.v) \in \{R_V, R_{W_1} \otimes \cdots \otimes R_{W_m}\}.
\end{align}

When $t=2n+1$, the final relation in (\ref{for5}) becomes
\[
\{e, e^{s+n-1} f^{l+n+1} h\} = (l+n+1) e^{s+n-1} f^{l+n} h^2 - 2e^{s+n} f^{l+n+1}.
\]
From \cite[Proposition 3.3]{arakawa2018quasi}, $\overline{\omega} = 4ef + h^2$ is nilpotent in $R_V$. Suppose $\overline{\omega}^{p+1} = 0$ in $R_V$, then
\[
\{e^p, \overline{\omega}^{p+1}\} = e^p (ef + \text{const} \cdot h^2),
\quad \text{and} \quad
f^p (ef + \text{const} \cdot h^2) = 0 \text{ in } R_V.
\]
Thus, for $s + l + t - 2 > 2p$, one also has (\ref{for6}).

For the case $s = 0$, we proceed similarly. When $t \geq 2$,
\[
\{f, f^l h^t v\} = 2t f^{l+1} h^{t-1} v + f^l h^t(f.v),
\]
and applying recursive relations
\begin{align}\label{for7}
\begin{split}
\{e, f^{l+2} h^{t-2} v\} &= (l+2) f^{l+1} h^{t-1} v - 2(t-2) f^{l+2} e h^{t-3} v, \\
\{e, f^{l+3} e h^{t-4} v\} &= (l+3) f^{l+2} e h^{t-3} v - 2(t-4) f^{l+3} e^2 h^{t-5} v, \\
&\ \vdots
\end{split}
\end{align}
For $t = 2n$, we get
\begin{align}\label{for8}
f^l h^t (f.v) \in \{R_V, R_{W_1} \otimes \cdots \otimes R_{W_m}\}.
\end{align}
When $t=1$, this was already shown. When $t=0$, (\ref{for8}) follows from
\[
\{f, f^l v\} = f^l(f.v).
\]

Let ${\rm mon}^d$ be any monomial of degree $d$ in $R_V$. Then, for $d - 2 > 2p$, we have
\[
{\rm mon}^d(f.v) \in \{R_V, R_{W_1} \otimes \cdots \otimes R_{W_m}\}.
\]
By induction,
\[
{\rm mon}^d(f^s.v) \in \{R_V, R_{W_1} \otimes \cdots \otimes R_{W_m}\} \quad \text{for all } s \geq 1.
\]
Therefore, using the bracket
\[
\{e, {\rm mon}^d(f^s.v)\} \in \{e, {\rm mon}^d\}(f^s.v) + \text{const} \cdot {\rm mon}^d(f^{s-1}.v),
\]
we obtain
\[
{\rm mon}^d(f^{s-1}.v) \in \{R_V, R_{W_1} \otimes \cdots \otimes R_{W_m}\}.
\]
Hence
\[
{\rm mon}^d v \in \{R_V, R_{W_1} \otimes \cdots \otimes R_{W_m}\}.
\]

Let $v_{(i_1, \ldots, i_m)} := f^{i_1} w_1 \otimes \cdots \otimes f^{i_m} w_m$. We apply the above process to all such vectors in lexicographic order (as described), and conclude
\[
{\rm mon}^d v_{(k_1, \ldots, k_m)} \in \{R_V, R_{W_1} \otimes \cdots \otimes R_{W_m}\}, \quad \text{for } d - 2 > 2p.
\]

Finally, this shows that the quotient
\[
R_{W_1} \otimes \cdots \otimes R_{W_m} / \{R_V, R_{W_1} \otimes \cdots \otimes R_{W_m}\}
\]
is spanned by vectors of the form ${\rm mon}^d v_{(k_1,\ldots,k_m)}$ with $d \leq 2p + 2$. If the eigenvalue of the operator \( \{ h, \cdot \} \) acting on \( w_1 \otimes \cdots \otimes w_m \) lies in \( \mathbb{Z} \), then only finitely many such vectors are annihilated by \( \{ h, \cdot \} \). Denote their number by $N$. Hence the dimension of the quotient is at most $N$.
\end{proof}

\appendix
\section{Coordinate transformation}
Let $V$ be a vertex operator (super)algebra and $M$ be its module. 
Let $\mathcal{O}$ be the complete topological $\mathbb{C}$-algebra $\mathbb{C}[[z]]$. A coordinate transformation of the infinitesimal disk $D=\operatorname{Spec}(\mathcal{O})$ is given by $\rho \in \operatorname{Aut}(\mathcal{O})$, which is uniquely determined by $\rho(z) \in z\mathbb{C}[[z]]$. We may write
\[
\rho(z)=\rho'(0)\cdot \exp\left(\sum_{n\in \Z_+}c_{n}z^{n+1}\partial_{z}\right)(z).
\]
Define
\[
\mathcal{U}(\rho)=\rho'(0)^{L_{(0)}}\exp\left(\sum_{n\in \Z_+} c_{n}L_{(n)}\right).
\]
Note that $\mathcal{U}(\rho(z)) \in \operatorname{End}(M^{\leq n})\otimes \mathcal{O}$, where $M^{\leq n} = \bigoplus_{i\leq n} M{(i)}$. The operator $\mathcal{U}(\rho(z))$ gives a group representation of $\operatorname{Aut}(\mathcal{O})$ on $M$.

If we change the coordinate of the small disk from $z$ to $\alpha(z)\in \mathcal{O}$, then by \cite{huang1997two}, we have the coordinate transformation formula:
\[
\mathcal{U}(\alpha(z))Y_{M}(v,z)\mathcal{U}(\alpha(z))^{-1} = Y_M((\mathcal{U}(\varrho(\alpha|\mathbf{1}))_{z})v,\alpha(z)),
\]
where $\varrho(\alpha|\mathbf{1})_{z}(t) = \alpha(z+t) - \alpha(z)$.

Analytically, we may interpret $\mathcal{U}(\varrho(\alpha|\mathbf{1})_{z})$ as a coordinate transformation $\rho$ from $\xi - z$ to $\alpha - \alpha(z)$ at a point $z$ in a small open set $U$ of a curve $C$, where $\xi$ is the standard coordinate on this open set. That is,
\[
\mathcal{U}(\varrho(\alpha|\mathbf{1})_{z}): M^{\leq n} \otimes \mathcal{O}_{U} \rightarrow M^{\leq n} \otimes \mathcal{O}_{U}.
\]

Therefore, we may write $\mathcal{U}(\varrho(\alpha|\mathbf{1})_{z}) = \mathcal{U}_{\varrho}(\alpha)\mathcal{U}_{\varrho}(\xi)$, and define $\mathcal{U}_{\varrho}(\alpha): \mathcal{M} \otimes \mathcal{O}_{U} \rightarrow M \otimes \mathcal{O}_{U}$ as a local trivialization of the sheaf with values in $\mathcal{M}$. In particular, we have $\mathcal{U}_{\varrho}(\alpha)|_{x} = \mathcal{U}(\alpha - x)$. 

Let
\[
\mathcal{U}(\varrho(\alpha|\beta)_{x}) = \mathcal{U}_{\varrho}(\alpha) \mathcal{U}_{\varrho}(\beta)^{-1}: M \otimes \mathcal{O}_U \rightarrow M \otimes \mathcal{O}_U,
\]
where $\varrho(\alpha|\beta)_{x}$ is the coordinate transformation from $\alpha - \alpha(x)$ to $\beta - \beta(x)$. Since $\varrho(\alpha|\beta)$ satisfies the cocycle condition, so does $\mathcal{U}(\varrho(\alpha|\beta))$. The sheaf $\mathcal{M}_C$ is called the sheaf of VOA if $M = V$.

We also have the following formula for the change of variable:
\[
\operatorname{Res}_{w} Y(v,w) = \operatorname{Res}_{z} Y(v,\rho(z))\frac{d}{dz}\rho(z).
\]

\begin{example}\label{ctft}
There are two local coordinate descriptions of a torus. One is via annuli $\mathbb{C} \setminus \{0\} / \{w \sim w q^{n}\}$, where $q$ is a complex number. The other is via the parallelogram $\mathbb{C} / \{m\tau + n\}$, with $\operatorname{Im}(\tau) > 0$ and $m,n \in \mathbb{Z}$. These two are related by the coordinate transformation $\varphi(x) = e^{2\pi i x} - 1$. 

We compute
\[
\varrho(\varphi|\mathbf{1})_{z}(t) = e^{2\pi i z}(e^{2\pi i t} - 1).
\]
Let
\[
\frac{1}{2\pi i}(e^{2\pi i y} - 1) = \exp\left(-\sum_{j \in \mathbb{Z}_{+}} A_{j} y^{j+1} \frac{\partial}{\partial y}\right)y.
\]
In particular, a direct calculation gives
\begin{align}\label{cdtfms}
\begin{split}
A_{1} &= -\pi, \\
A_{2} &= -\frac{\pi^{2}}{3}.
\end{split}
\end{align}
We denote the operator $\sum_{j \in \mathbb{Z}_{+}} A_{j} L_{(j)}$ on $M$ by $L_{+}(A)$. Moreover,
\[
e^{2\pi i z} - 1 = (2\pi i)^{z \frac{\partial}{\partial z}} \exp\left(-\sum_{j \in \mathbb{Z}_{+}} A_{j} z^{j+1} \frac{\partial}{\partial z}\right) z.
\]
Therefore,
\[
\mathcal{U}(e^{2\pi i z} - 1) = (2\pi i)^{L_{(0)}} e^{-L_{+}(A)}.
\]

It is clear in this case that
\[
\mathcal{U}(\varrho(\varphi|\mathbf{1})_{z}(t)) = (2\pi i e^{2\pi i z})^{L_{(0)}} e^{-L_{+}(A)}.
\]
By the coordinate transformation formula, we obtain
\begin{align}\label{ccoor2}
  &(2\pi i)^{L_{(0)}} e^{-L_{+}(A)} Y_{M}(v,z) 
  \left((2\pi i)^{L_{(0)}} e^{-L_{+}(A)}\right)^{-1} \notag\\
  &\quad = Y\left((2\pi i e^{2\pi i z})^{L_{(0)}} e^{-L_{+}(A)} v,\, e^{2\pi i z} - 1\right).
\end{align}
which is equivalent to
\begin{align}\label{cccoor1}
\begin{split}
& (2\pi i)^{L_{(0)}} e^{-L_{+}(A)}\,
Y_{M}\left( \left((2\pi i)^{L_{(0)}} e^{-L_{+}(A)}\right)^{-1} v,\, z \right) \\
&\quad \cdot \left((2\pi i)^{L_{(0)}} e^{-L_{+}(A)}\right)^{-1}
= Y\left((e^{2\pi i z})^{L_{(0)}} v,\, e^{2\pi i z} - 1\right).
\end{split}
\end{align}

The LHS of (\ref{cccoor1}) is denoted by $Y[v,z]$, first introduced by Zhu. Meanwhile, Huang introduced the geometrically-modified vertex operator, which is the modification of the LHS of (\ref{ccoor2}). More precisely, with abuse of notation, define
\[
\mathcal{U}(x) = (2\pi i x)^{L_{(0)}} e^{-L_{+}(A)} \in \operatorname{End}M\{x\},
\]
where $x$ is any number or formal expression that makes sense. Then $Y(\mathcal{U}(x)v, x)$ is the geometrically-modified vertex operator.

Both $Y[v,z]$ and the geometrically-modified vertex operators are essential tools in the study of the modularity of genus-one correlation functions.

One can compute \cite[Lemma 1.1]{huang2005differential}
\begin{align}\label{virasoro}
\begin{split}
\mathcal{U}(1)\omega 
&= (2\pi i)^{L_{(0)}} e^{-L_{+}(A)} \omega \\
&= (2\pi i)^{L_{(0)}} e^{-L_{+}(A)} L_{(-2)} \mathbf{1} \\
&= (2\pi i)^{L_{(0)}} L_{(-2)} \mathbf{1} - (2\pi i)^{L_{(0)}} A_2 L_{(2)} L_{(-2)} \mathbf{1} \\
&= (2\pi i)^2 L_{(-2)} \mathbf{1} + \frac{\pi^2 c}{6} \mathbf{1} \\
&= (2\pi i)^2 \left(\omega - \frac{c}{24} \mathbf{1}\right),
\end{split}
\end{align}
where the fourth identity uses (\ref{cdtfms}). Thus, when considering correlation functions on a parallelogram torus, one must add a factor of $q_{\tau}^{-c/24}$.
\end{example}
Now we consider the twisted case. Let $\sigma$ be an automorphism of the disk $D$ of order $N$.

Define
\begin{align*}
    \operatorname{Aut}_{N}(D) &= \left\{ t \in \mathbb{C}[[z]] \,\middle|\, \sigma t = e^{2\pi i/N}t \right\}, \\
    \operatorname{Aut}_{N}(\mathcal{O}) &= \left\{ \rho(z^{\frac{1}{N}}) = \sum_{n \in \frac{1}{N} + \mathbb{Z},\, n > 0} c_n z^n \,\middle|\, c_{\frac{1}{N}} \neq 0 \right\}.
\end{align*}

It is clear that $\operatorname{Aut}_{N}(\mathcal{O})$ acts on $\operatorname{Aut}_{N}(D)$, and that there is a homomorphism from $\operatorname{Aut}_{N}(\mathcal{O})$ to $\operatorname{Aut}(\mathcal{O})$ given by $\rho(z) \mapsto \rho(z)^N$.

For any $\rho \in \operatorname{Aut}_{N}(\mathcal{O})$, we write
\[
\rho(z^{\frac{1}{N}}) = \exp\left( \sum_{k \in \mathbb{Z},\, k \geq 0} v_k z^{k + \frac{1}{N}} \partial_{z^{\frac{1}{N}}} \right) z^{\frac{1}{N}}.
\]

Define
\[
\mathcal{U}^{\sigma}(\rho) := \exp\left( \sum_{k \in \mathbb{Z},\, k \geq 0} v_k N L_{(k)} \right).
\]
This gives a natural group representation of $\operatorname{Aut}_{N}(\mathcal{O})$ on a $g$-twisted module $M^{g}$, where $g$ has order $N$. Then, by \cite{frenkel2004twisted}, one has the coordinate transformation formula
\[
\mathcal{U}(\rho) Y_{M^{g}}(u, z) \mathcal{U}(\rho)^{-1} = Y_{M^{g}}\left( \mathcal{U}\left( \varrho(\rho|\mathbf{1})_{z} \right) u, \rho(z) \right).
\]

Let $C$ be a smooth projective curve, and let $H \subset \operatorname{Aut}(C)$ be a finite subgroup of automorphisms of order $N$, which also acts on $V$ as a conformal automorphism.

One can induce an $H$-equivariant structure on the sheaf of VOAs $\mathcal{V}_{C}$ by defining:
\[
h(p, (u, z)) := (h(p), (u, z \circ h^{-1})),
\]
where $p \in C$ with local coordinate $z$, and $u \in V$.

Let $\mathring{C}$ be the complement of the ramification points on $C$ of the quotient map $v: C \to C/H$. Then $v$ induces a principal $H$-bundle $v: \mathring{C} \to \mathring{C}/H := \mathring{X}$.

Fix a point $p \in v^{-1}(x)$ for some $x \in \mathring{X}$. The point $p$ has a cyclic stabilizer subgroup $h_p \in H$ of order $N$. The local coordinate on the small disk $D_p$ around $p$ is a special coordinate $t \in \mathbb{C}[[z]]$ satisfying $h_p(t) = e^{2\pi i/N} t$.

We can attach to $D_p$ a $h_p$-twisted $V$-module $M^{h_p}$. Then, there exists a local trivialization:
\[
\mathcal{U}^{\sigma}_{\varrho}(\rho(z^{\frac{1}{N}})): \mathcal{M}^{h_p}_{D_p} \to M^{h_p} \otimes D_p.
\]

Since $H$ acts transitively on $v^{-1}(x)$, we attach to every open neighborhood of $h(p)$ the $h h_p h^{-1}$-twisted module $h \circ M^{h_p}$.

\section{Modular forms}
In this and the next section, we review some basics of modular forms and fix the notations. For more details, readers can refer to   \cite{zhu1996modular, mason2010vertex, zagier2008elliptic}.

The {\em Eisenstein series} $G_{2k}(\tau)$ ($k \in \mathbb{Z}_+$) are defined by
\begin{align}
\begin{split}
G_{2k}(\tau) &:= \sum_{(m,n)\in \mathbb{Z}^2 \setminus \{(0,0)\}} \frac{1}{(m+n\tau)^{2k}} \\
             &= (2\pi i )^{2k} \left( -\frac{B_{2k}}{(2k)!} + \frac{2}{(2k-1)!} \sum_{n\in \mathbb{Z}_+} \frac{n^{2k-1} q^{n}}{1 - q^n} \right),
\end{split}
\end{align}
where $B_{k}$ is the $k$th Bernoulli number. Denote $\mathbb{E}_{2k}(\tau) := \frac{G_{2k}(\tau)}{(2\pi i)^{2k}}$.

For $k \geq 2$, $G_{2k}(\tau)$ is a modular form of weight $2k$ for the modular group $SL_2(\mathbb{Z})$, i.e.
\[
G_{2k}\left( \frac{a\tau + b}{c\tau + d} \right) = (c\tau + d)^{2k} G_{2k}(\tau).
\]
On the other hand, $G_2(\tau)$ is a quasi-modular form that obeys the transformation law
\[
G_2\left( \frac{a\tau + b}{c\tau + d} \right) = (c\tau + d)^2 G_2(\tau) - 2\pi i c (c\tau + d).
\]

One can renormalize $G_{2k}(\tau)$ ($k \in \mathbb{Z}_+$) so that its constant term is $1$, and call them $E_{2k}(\tau)$. The $q$-series expansion of $E_{2k}(\tau)$ is
\begin{align}
E_{2k}(\tau) = 1 - \frac{4k}{B_{2k}} \sum_{n=1}^{\infty} \frac{n^{2k-1} q^n}{1 - q^n},
\end{align}
where $q = e^{2\pi i \tau}$.

The polynomial algebra generated by $E_2(\tau)$, $E_4(\tau)$, and $E_6(\tau)$ over $\mathbb{C}$ is called the algebra of {\em quasi-modular forms}. The ring of all holomorphic modular forms on $\mathbb{H}$ is the graded algebra
\[
A = \bigoplus_{k} A_k = \mathbb{C}[E_4(\tau), E_6(\tau)],
\]
where $A_k$ is the subspace of $A$ spanned by all modular forms of weight $k$.

Given a function $f$ on the upper half-plane, we denote by $\tilde{f}(q)$ its $q$-expansion.

Given an element $f \in A_k$, the {\em Serre derivation} of weight $k$, denoted by $\partial_k$, is defined by
\[
\partial_k f := f' + k \mathbb{E}_2(\tau) f.
\]
The Serre derivation preserves holomorphy and modularity: $\partial_k f \in A_{k+2}$. The $i$th iterated Serre derivative of weight $k$ is defined by
\[
\partial_k^{i} := \partial_{k+2(i-1)} \circ \cdots \circ \partial_{k+2} \circ \partial_k, \quad \text{with } \partial_k^0 := \operatorname{id}.
\]

A {\em modular linear differential equation (MLDE)} of weight $k$ is a linear differential equation of the form
\[
\partial_k^n f + \sum_{j=0}^{n-1} g_j \, \partial_k^j f = 0,
\]
where $g_j \in A$ and the weight of $g_j$ is $2n - 2j$ for each $0 \leq j \leq n-1$. The solutions of an MLDE lie in $A$.

Since we will work with $q$-series later, we define the {\em formal Serre derivation} $\partial_k$ and the {\em formal iterated Serre derivation} $\partial^i$ of weight $k$ by
\begin{align}\label{serrederi}
\begin{split}
\partial_k f(q) &= q \frac{d}{dq} f(q) + k \tilde{\mathbb{E}}_2(q) f(q), \\
\partial^i f(q) &= \partial_{k+2i-2} \left( \partial^{i-1} f(q) \right),
\end{split}
\end{align}
where $f \in A_k$ and $\partial^0 := \operatorname{id}$.

\section{Twisted modular forms}
The main references for this section are \cite{dong2000modular, mason2008torus}. Let $(\theta,\phi)$ be a pair of points in the unit disk on the complex plane with $\phi=e^{2\pi i\lambda}$.

For each integer $k=1,2,\cdots$, one defines the {\em twisted Weierstrass functions} $P_{k}\genfrac[]{0pt}{0}{\theta}{\phi}$ on $\mathbb{C}\times \mathbb{H}$ as follows:
\begin{align}\label{pqfunction1}
  P_{k}\genfrac[]{0pt}{0}{\theta}{\phi}(z,\tau)=\frac{(-2\pi i)^{k}}{(k-1)!}\sum_{n\in \lambda+\mathbb{Z}}'\frac{n^{k-1}q_{z}^{n}}{1-\theta^{-1} q_{\tau}^{n}},
\end{align}
where the sign $\sum'$ means we omit the term $n=0$ if $(\theta,\phi)=(1,1)$. The function $P_{k}\genfrac[]{0pt}{0}{\theta}{\phi}(z,\tau)$ converges uniformly and absolutely on compact subsets of the region $|q_{\tau}|<|q_{z}|<1$. Note
\begin{align}\label{pqfunction6}
    P_{k}\genfrac[]{0pt}{0}{\theta}{\phi}(z,\tau)=\frac{(-1)^{k-1}}{(k-1)!}\frac{\partial^{k-1}}{\partial z^{k-1}}P_{1}\genfrac[]{0pt}{0}{\theta}{\phi}(z,\tau).
\end{align}

Next, we introduce the {\em twisted $P$-type function}
\begin{align}\label{pqfunction3}
  P_{1}^{\theta,\phi}(z,\tau)=\begin{cases}
               2\pi i \sum_{n\in \mathbb{Z},\, n\neq 0}\frac{q_{z}^{n}}{1-\theta^{-1} q_{\tau}^{n}}   & \phi=1,\\
               2\pi i \sum_{n\in \lambda+\mathbb{Z}}\frac{q_{z}^{n}}{1-\theta^{-1} q_{\tau}^{n}}      & \phi\neq 1.
            \end{cases}
\end{align}

\begin{remark}\label{pqfunction5}
Note that $-P_{1}\genfrac[]{0pt}{0}{\theta}{\phi}(z,\tau)= P_{1}^{\theta,\phi}(z,\tau)$ except when $\phi=1,\ \theta\neq 1$. In that case,
\begin{align}
  -P_{1}\genfrac[]{0pt}{0}{\theta}{\phi}(z,\tau)= P_{1}^{\theta,\phi}(z,\tau)+2\pi i \frac{1}{1-\theta^{-1}}.
\end{align}
\end{remark}

The {\em twisted Eisenstein series} for $n\geq 1$ is defined by
\begin{align}
    \begin{split}
       G_{n}\genfrac[]{0pt}{0}{\theta}{\phi}(\tau)=(2\pi i )^n &\left( -\frac{B_{n}(\lambda)}{n!} + \frac{1}{(n-1)!} \sum_{r\geq 0}' \frac{(r+\lambda)^{n-1}\theta^{-1}q_{\tau}^{r+\lambda}}{1-\theta^{-1}q_{\tau}^{r+\lambda}} \right. \\
       &\left. + \frac{(-1)^{n}}{(n-1)!} \sum_{r\geq 1} \frac{(r-\lambda)^{n-1}\theta q_{\tau}^{r-\lambda}}{1-\theta q_{\tau}^{r-\lambda}} \right),
    \end{split}
\end{align}
where $\sum'$ omits $r=0$ if $(\theta, \phi)=(1,1)$, and $B_n(\lambda)$ is the $n$th Bernoulli polynomial. In particular,
\begin{align}\label{pqfunction3}
G_{n}\genfrac[]{0pt}{0}{1}{1}(\tau)=\begin{cases}
              G_{2k}(\tau)   & n=2k,\ k\geq 1, \\
              \pi i\,\delta_{n,1} & n=2k-1,\ k\geq 1.
            \end{cases}
\end{align}

\begin{lemma}[\cite{mason2008torus}]\label{Mason}
Suppose that $|q_{\tau}|<|q_{z}|<1$. Then
\begin{align}\label{pqfunction4}
\begin{split}
\displaystyle P_{m}\genfrac[]{0pt}{0}{\theta}{\phi}(z,\tau)=\frac{1}{z^{m}}+(-1)^{m} \sum_{n\geq m} \binom{n-1}{m-1} G_{n}\genfrac[]{0pt}{0}{\theta}{\phi}(\tau)\, z^{n-m}.
\end{split}
\end{align}
For $(\theta,\phi)\neq(1,1)$, the function $P_{k}\genfrac[]{0pt}{0}{\theta}{\phi}(z,\tau)$ is periodic in $z$ with periods $2\pi \tau$ and $2\pi i$, with multipliers $\theta$ and $\phi$, respectively.
\end{lemma}

\begin{remark}\label{weirstrasswp}
The expressions $P_{1}\genfrac[]{0pt}{0}{1}{1}(z,\tau)+zG_2(\tau)$ and $P_2\genfrac[]{0pt}{0}{1}{1}(z,\tau)-G_2(\tau)$ correspond to the Weierstrass zeta and $\wp$-functions, denoted respectively by $\wp_1$ and $\wp_2$. In general, for $m\geq 1$, define
\begin{align}
    \wp_{m+1}(z,\tau)=-\frac{1}{m} \frac{\partial}{\partial z}\wp_m(z,\tau).
\end{align}
Note that $\wp_m(z,\tau)=P_m\genfrac[]{0pt}{0}{1}{1}(z,\tau)$ for $m>2$.
\end{remark}

Define also the higher twisted $P$-type functions by
\begin{align}\label{pqfunction7}
    P_{m}^{\theta,\phi}(z,\tau):=\frac{1}{(m-1)!} \frac{\partial^{m-1}}{\partial z^{m-1}} P_{1}^{\theta,\phi}(z,\tau), \quad m\geq 2.
\end{align}

We will also consider the twisted $P$-type function as formal $q$-series over the ring $\mathbb{C}((q^{\lambda}))$. In this case, for $\phi \neq 1$,
\begin{align}
\tilde{P}_{m}^{\theta,\phi}(x,q)=\frac{1}{m!}\left(\sum_{n\geq 0}\frac{(n+\lambda)^{m}x^{n+\lambda}}{1-\theta^{-1} q^{n+\lambda}} - \sum_{n>0} \frac{(-1)^{m} \theta (n-\lambda)^{m} x^{-(n-\lambda)} q^{n-\lambda}}{1-\theta q^{n-\lambda}}\right),
\end{align}
and when $\phi=1$,
\begin{align}\label{pqfunction8}
\tilde{P}_{m}^{\theta,1}(x,q)= (2 \pi i )^{m} \frac{1}{m!} \left( \sum_{n>0} \frac{n^{m}x^{n}}{1-\theta^{-1} q^{n}} - \sum_{n>0} \frac{(-1)^{m} \theta n^{m} x^{-n} q^{n}}{1-\theta q^{n}} \right).
\end{align}

Finally, by Remark \ref{pqfunction5}, equations (\ref{pqfunction6}) and (\ref{pqfunction7}) imply
\begin{align}\label{pqfunction71}
    P_{m}^{\theta,\phi}(z,\tau)=\begin{cases}
            (-1)^{m} P_{m}\genfrac[]{0pt}{0}{\theta}{\phi}(z,\tau) - \delta_{1,m} 2\pi i\frac{1}{1-\theta^{-1}},   & \phi=1,\ \theta\neq 1,\\
            (-1)^{m} P_{m}\genfrac[]{0pt}{0}{\theta}{\phi}(z,\tau),                                             & \text{otherwise}.
    \end{cases}
\end{align}

\appendix

\bibliography{bibliography}
\bibliographystyle{alpha}

\end{document}